\documentclass{article}

\usepackage{amsthm,amssymb,amsfonts,amsmath,amscd}
\usepackage[all]{xy}
\newtheorem{dfn}{Definition}[section]
\newtheorem{thm}[dfn]{Theorem}
\newtheorem{cor}[dfn]{Corollary}
\newtheorem{lem}[dfn]{Lemma}
\newtheorem{pro}[dfn]{Proposition}
\newtheorem{rmk}[dfn]{Remark}
\newtheorem{ex}[dfn]{Example}

\newcommand{\bb}[1]{\mathbb#1}
\newcommand{\fra}[1]{\mathfrak#1}
\newcommand{\ca}[1]{\mathcal#1}
\newcommand{\innpro}[3]{<#1,#2>_{#3}}
\newcommand{\Innpro}[3]{\ll#1,#2\gg_{#3}}

\begin{document}
\title{Naturality of FHT isomorphism}
\author{Doman Takata}
\maketitle

\begin{abstract}
Freed, Hopkins and Teleman constructed an isomorphism (we call it FHT isomorphism) between twisted equivariant $K$-theory of compact Lie group $G$ and the ``Verlinde ring'' of the loop group of $G$ (\cite{FHT1} \cite{FHT2} \cite{FHT3}). However, this isomorphism does not hold naturality with respect to group homomorphisms. We construct two ``quasi functors'' $t.e.K$ and $RL$ so that FHT isomorphism is natural transformation between two ``quasi functors'' for tori, that is, we construct two ``induced homomorphisms'' 
$f^\#:t.e.K(T',\tau)\to t.e.K(T,f^*\tau)$ and $f^!:RL(T',\tau)\to RL(T,f^*\tau)$ for $f:T\to T'$ whose tangent map is injective. In fact, we construct another quasi functor $char$ and verify that three quasi functors are naturally isomorphic.

Moreover, we extend the quasi functor $t.e.K$ and $char$ to compact connected Lie group with torsion-free $\pi_1$ and homomorphism $f:H\to G$ satisfying the decomposable condition, and verify that they are isomorphic. This is a generalization of naturality of $K^{\tau+rank(G)}_G(G)\cong char(G,\tau)$ verified in \cite{FHT1}.
%such that it can be written as a composition of $H\xrightarrow{i_1} H\times T \xrightarrow{q} G$ for a torus $T$ such that $rank(H)+rank(T)=rank(G)$. Where $i_1$ is the inclusion into the first factor and $q$ is locally injective. This is an extension of naturality of the computation of $K^{\tau+rank(G)}_G(G)$ verified in \cite{FHT1}.
\end{abstract}

\section*{Introduction}

Study of loop groups of compact Lie groups is extensively developed especially from the view point of representation theory. On the other hand, twisted $K$-theory was introduced by Donovan and Karoubi \cite{Karoubi}, and interest in it was rekindled by its appearance in string theory. Freed, Hopkins and Teleman connected them (\cite{FHT2}).
In this paper, we study functorial aspects of these two objects.

Let us start from the representation theory of loop groups. Positive energy representations are special cases, defined by $S^1$ symmetry of loop groups. $S^1$ acts on loop groups via transformation of the parameter. When we make $S^1$ act on loop groups, we write $S^1$ as $\bb{T}_{rot}$. Let $G$ be a compact connected Lie group with torsion-free $\pi_1$ and $\ca{H}$ be a separable complex Hilbert space as a representation space.
\begin{dfn}[\cite{PS}]A positive energy representation is a continuous projective representation $\rho:LG\to PU(\ca{H})$ satisfying the following conditions.

(1) The action of $\bb{T}_{rot}$ on $LG$ lifts to the associated central extension by $U(1)$
$$LG^\tau:=LG\times_{PU(\ca{H})}U(\ca{H}).$$
Therefore, we can define the new Lie group $LG^\tau\rtimes \bb{T}_{rot}$.

(2) $\rho$ lifts to $\widehat{\rho}:LG^\tau\rtimes \bb{T}_{rot}\to U(\ca{H})$.

(3) There are no vectors $v$ on which rotation  $e^{i\theta}\in\bb{T}_{rot}$ acts by multiplication by $e^{in\theta}$ with sufficient small $n$.
\end{dfn}
\begin{rmk}
A central extension $\tau$ is also called level. A representation is at level $\tau$ if the induced central extension $LG\times_{PU(\ca{H})}U(\ca{H})$ is isomorphic to $LG^\tau$.
\end{rmk}
The direct sum of two representations at level $\tau$ is also at level $\tau$. So finitely reducible positive energy representations at level $\tau$ form a semigroup under the direct sum.
\begin{dfn}
The representation group $R^\tau(LG)$ of $LG$ at level $\tau$ is the Grothendieck completion of this semigroup.
\end{dfn}

Our second main object is twisted equivariant $K$-theory.

%Twisted $K$-theory was introduced by Donovan and Karoubi to extend Thom isomorphism for non $K$-orientable vector bundle \cite{Karoubi}. Later, it was extended by Rosenberg \cite{AS}. It has an equivariant version just as untwisted cases.

Let $G$ be a compact Lie group, $X$ be a topological space. We suppose that $G$ acts on $X$ continuously. Let $\tau$ be a $G$-equivariant twisting over $X$ and $\{K^{\tau+k}_G(X)\}_{k\in\bb{Z}}$ be the $\tau$-twisted $K$-groups, which are defined in Section 1. 

%Let $G$ be a compact connected Lie group with 
Let $G$ have torsion-free $\pi_1$ and   $\tau$ be a ``positive'' central extension of $LG$, which is defined in Section 1.
A central extension $\tau$ defines a $G$-equivariant twisting over $G$ (\cite{AS}), where $G$ acts on itself by conjugation.
Freed, Hopkins and Teleman constructed an isomorphism between $R^{\tau}(LG)$ and $K^{\tau+\sigma+rank(G)}_G(G)$, where $\sigma$ is the special central extension of $LG$ called the spin extension.

\begin{thm}[\cite{FHT2}]\label{The theorem of FHT}
In this situation, we have an isomorphism
$$FHT_G:R^{\tau}(LG)\to K^{\tau+\sigma+rank(G)}_G(G).$$
We call it FHT isomorphism.
\end{thm}

However, this isomorphism does not hold naturality. In this paper, we study naturality of FHT isomorphisms for tori.

Let us consider pull backs for two cases.
Let $f:T'\to T$ be a smooth group homomorphism between two tori $T$ and $T'$. We suppose that the tangent map $df$ is injective, we call such a map a local injection.

For the case of representation group, the pull back of finitely reducible representations along $f$ is not finitely reducible when $\dim(T)\neq \dim(T')$ (Corollary \ref{infinitely reducibility in Section 5}), that is, $f^*:R^\tau(LT)\to R^{f^*\tau}(LT')$ cannot be defined by use of the pull back of representations. %link
For the case of twisted equivariant $K$-theory, the $K$-theoretical pull back must be zero if $\dim(T)-\dim(T')$ is odd (Corollary \ref{zeroness of classical pull back}). %Link
%For the case of representation group, the map $f$ defines the associated homomorphism $Lf:LT\to LT'$ by $Lf(l):=f\circ l$ for $l\in LT$. Pull back $f^*\tau$ of a central extension $\tau$ of $LT'$ determines one of $LT$. $Lf$ seems to define the pull back of representations just like the case of compact Lie groups
%$$(Lf)^*:R^{\tau}(LT') \to R^{f^*\tau}(LT).$$
%However, it is not well-defined in general cases. $(Lf)^*$ does not preserve the finite reducibility. In fact, when 
%$T'=T\times T''$ for a torus $T''$ and $f$ is the inclusion into the first factor, the representation induced by $f$ is not finitely reducible. We verify it in Section 5.

%For the case of twisted equivariant $K$-theory, pull back $f^*\tau$ of a twisting $\tau$ is a $T$-equivariant twisting over $T$. $f$ defines the $K$-theoretical pull back
%$$f^*:K^{\tau+k}_{T'}(T')\to K^{f^*\tau+k}_T(T).$$
%This map must be zero-map if $dim(T')-dim(T)$ is odd. It can be verified easily from the description of $K^{\tau+k}_T(T)$ described in Theorem 0.9.%リンク

Motivated by these observations, we modify pull backs so that the map $FHT_T$ holds naturality. That is, we construct ``induced homomorphisms''
$$f^\#:K^{\tau+dim(T)}_{T}(T)\to K^{f^*\tau+dim(T')}_{T'}(T')$$
$$f^!:R^\tau(LT)\to R^{f^*\tau}(LT').$$
By use of these induced homomorphisms, we verify the following main theorem.

\begin{thm}\label{main corollary}
The following commutative diagram holds.
$$\begin{CD}
K^{\tau +dim(T)}_{T}(T) @>f^\#>> K^{f^*\tau+dim(T')}_{T'}(T') \\
@VFHT_{T}VV @VFHT_{T'}VV \\
R^{\tau}(LT) @>f^!>> R^{f^*\tau}(LT')
\end{CD}$$
\end{thm}
This theorem is verified by use of another objects $char(T,\tau)$, $char(T',f^*\tau)$ and homomorphism $char(f)$, which are defined later.

In fact, we verify the following two theorems. Vertical arrows are defined later.
\begin{thm}\label{main theorem for K for tori}
The following commutative diagram holds.
$$\begin{CD}
K^{\tau +dim(T)}_{T}(T) @>f^\#>> K^{f^*\tau+dim(T')}_{T'}(T') \\
@VM.d._TVV @VM.d._{T'}VV \\
char(T,\tau) @>char(f)>> char(T',f^*\tau)
\end{CD}$$
\end{thm}
\begin{thm}\label{main theorem for R for tori}
The following commutative diagram holds.
$$\begin{CD}
R^{\tau}(LT) @>f^!>> R^{f^*\tau}(LT') \\
@Vl.w._TVV @Vl.w._{T'}VV \\
char(T,\tau) @>char(f)>> char(T',f^*\tau) 
\end{CD}$$
\end{thm}

We can extend Theorem \ref{main theorem for K for tori} to $f:H\to G$ satisfying the decomposable condition (Section 2), where $H$ and $G$ can be not tori.
\begin{thm}\label{main theorem for K for general case}
The following commutative diagram holds.
$$\begin{CD}
K^{\tau +rank(G)}_{G}(G) @>f^\#>> K^{f^*\tau+rank(H)}_H(H) \\
@VM.d._GVV @VM.d._HVV \\
char(G,\tau) @>char(f)>> char(H,f^*\tau)
\end{CD}$$
\end{thm}

Let us explain the details and the backgrounds.

Firstly, we explain the details on representation groups. Let us start from the representation theory of compact Lie groups (Cartan-Weyl theory). Before that, we deal with tori. Representation theory for tori is fundamental. Let $T$ be an $n$ dimensional torus. Let us recall an explicit description of representation ring of $T$
$$R(T)\cong\bb{Z}[X_1,X_2,\cdots,X_n,(X_1X_2\cdots X_n)^{-1}],$$
where $R(T)$ is the Grothendieck completion of the semiring which finite dimensional representations of $T$ form. $X_j$ corresponds to the projection onto the $j$'th factor $p_j:T\cong (U(1))^n\to U(1)$. This isomorphism is verified as follows.
By Schur's lemma, any irreducible unitary representations are $1$ dimensional. So we can identify the set of isomorphism classes of irreducible representations with the set of homomorphisms from $T$ to $U(1)$. We write it as ${\rm Hom}(T,U(1))$. We can rewrite it by taking tangent maps. Let $\Pi_T:=\ker(\exp:\fra{t}\to T)$ and $\Lambda_T:={\rm Hom}(\Pi_T,\bb{Z})$. We have a canonical isomorphism ${\rm Hom}(T,U(1))\cong \Lambda_T$ by taking tangent maps. 
Since $\Lambda_T$ is isomorphic to the $\bb{Z}$ free module generated by $X_1,X_2,\cdots, X_n$, the isomorphism $R(T)\cong\bb{Z}[X_1,X_2,\cdots,X_n,(X_1X_2\cdots X_n)^{-1}]$ holds. If we regard $R(T)$ just as a group, it is isomorphic to the $\bb{Z}$ free module generated by $\Lambda_T$. So we write it as $\bb{Z}[\Lambda_T]$ also.

%Finite dimensional representations make a semiring under the direct sum and the tensor product of representations. $R(T)$, the representation ring of $T$ is defined by the ring completion of this semiring. When we take an isomorphism $T\cong(U(1))^n$, we have an explicit description of it. It is a polynomial ring which allows negative powers. That is,
%$R(T)\cong\bb{Z}[X_1,X_2,\cdots,X_n,(X_1X_2\cdots X_n)^{-1}]$, where $X_j$ corresponds to the projection onto the $j$'th factor $p_j:T\cong (U(1))^n\to U(1)$.

Let $G$ be a compact, connected, simply connected and simple Lie group, and $T$ be a chosen maximal torus. Let us recall a very beautiful correspondence
\begin{center}
\{isomorphism classes of irreducible representations\} $\cong$ \{dominant weights\}. 
\end{center}
Since for any weight $\lambda$, there exists $w\in W(G)$ such that $w.\lambda$ is a dominant weight, we have an isomorphism 
$R(G)\cong R(T)^{W(G)}$, the invariant subring of $R(T)$. Moreover, we have an isomorphism as groups 
$$R(G)\cong \bb{Z}[\Lambda_T/W(G)],$$
 where the right hand side $\bb{Z}[\Lambda_T/W(G)]$ is the $\bb{Z}$ free module generated by the set $\Lambda_T/W(G)$.
Let us explain this theorem. An isomorphism class of irreducible representation of $G$ is completely determined by the restriction to $T$. We can define a lexicographic order by taking an isomorphism $T\cong(U(1))^n$.
Let $\rho:G\to U(V)$ be an irreducible representation. We can decompose $V$ as irreducible representation of $T$. This decomposition determines the $W(G)$-invariant weighted finite subset in $\Lambda_T$,
where $W(G)$ is the Weyl group of $G$. This finite subset is determined by the maximum element (Cartan-Weyl's highest weight theory), which is ``dominant''.
Moreover, if we are given a dominant weight, we can define an irreducible representation whose highest weight is the given one (Borel-Weil theory). 

An analogue of Cartan-Weyl theory for loop groups is studied in \cite{PS}. Let us state a difficulty when we try to follow the above construction of compact Lie group case. Let $T$ be a torus and $LT$ be the set of smooth maps from $S^1$ to $T$, which is the loop group of $T$. The group structure of $LT$ is defined by the multiplication of $T$ pointwisely. %Since $\bb{T}_{rot}$ acts on $LT$ via transformation of parameter, we obtain the new Lie group $LT\rtimes \bb{T}_{rot}$.
However, the above mechanism for compact Lie group does not hold in this case, since the size of maximal connected commutative group $LT_0$ is too large, where $LT_0$ is the identity component of $LT$. Precisely, it is too difficult to deal with all representations. In fact, 
it has not been known to classify all representations of $LT$. In this paper, we focus on positive energy representations. 

Positive energy representations at level $\tau$ has been classified (\cite{PS}).
In order to describe it for tori, we need to introduce some terminologies.

Let $\tau$ be a positive central extension of $LT$.
We can restrict it to the set of constant loops, which is naturaly identified with $T$.

$$\xymatrix{
1 \ar[r]  & U(1) \ar[r]^i & LT^\tau \ar[r]^p & LT \ar[r]  & 1 \\\
1 \ar[r] \ar@{=}[u] & U(1) \ar[r]^i  \ar@{=}[u] & T^\tau \ar[r]^p  \ar[u] & T \ar[r] \ar[u]& 1\ar@{=}[u]
}$$

$\Lambda^\tau_T:=\{\rho:T^\tau\to U(1)|\rho\circ i=id\}$ is the set of $\tau$-twisted irreducible representations of $T^\tau$. A central extension $\tau$ defines an action of $\Pi_T$ on $\Lambda^\tau_T$ (Section 1).

\begin{thm}[\cite{PS}]\label{computation of R(LT)}
Isomorphism classes of positive energy representations at level $\tau$ are parametrized by the set of $\Pi_T$-orbits in $\Lambda^\tau_T$. That is, we have an isomorphism
$$l.w._T:R^\tau(LT)\xrightarrow{\cong}\bb{Z}[\Lambda_T^\tau/\Pi_T],$$
where $\bb{Z}[\Lambda_T^\tau/\Pi_T]$ is the $\bb{Z}$ free module generated by the set $\Lambda^\tau_T/\Pi_T$.
\end{thm}
The notation $l.w.$ comes from the ``lowest weight''. It is known that we can take a maximal torus of $LT^\tau\rtimes\bb{T}_{rot}$ as $T^\tau\times\bb{T}_{rot}$ if $\tau$ is positive. 
By formal computation in $LT^\tau\rtimes \bb{T}_{rot}$, Weyl group of $LT^\tau\times \bb{T}_{rot}$ is isomorphic to $\Pi_T$. Therefore, we can regard the above theorem as an analogue of the isomorphism $R(G)\cong R(T)^{W(G)}\cong \bb{Z}[\Lambda_T/W(G)]$.

Secondly, we explain the details on twisted equivariant $K$-theory.

Twisted $K$-theory was introduced by Donovan and Karoubi in order to extend Thom isomorphism for non $K$-orientable vector bundle \cite{Karoubi}. Later, it was extended by Rosenberg \cite{AS}. It has an equivariant version just like untwisted cases.

Let us start from untwisted equivariant $K$-theory.
Equivariant $K$-theory is concerned with representation theory (\cite{Segal}). Throughout this paper, we deal with compactly supported $K$-theory.
Let $X$ be a topological space and $G$ be a Lie group acting on $X$. If $X$ is a compact Hausdorff space, isomorphism classes of $G$-equivariant vector bundles form  a semigroup under the direct sum, and equivariant $K$-group $K_G(X)$ is defined by the Grothendieck completion of this semigroup. Generally, $K_G(X)$ is the set of $G$-equivariant homotopy classes of $G$-equivariant family of Fredholm operators parametrized by $X$. When an action of $G$ is special, we can describe $K_G(X)$ easily. When it is trivial, $K_G(X)\cong K(X)\otimes R(G)$, and when it is free, $K_G(X)\cong K(X/G)$.

Our second main object is twisted version of it. We start from non-equivariant twisted $K$-theory.
A twisting over $X$ is an element of the group $H^3(X)$.
It classifies $PU(\ca{H})$ principal bundles over $X$, where $\ca{H}$ is an infinite dimensional separable complex Hilbert space, since the functor $H^3$ has the classifying space $K(\bb{Z},3)$, and the homotopy equivalences $PU(\ca{H})\cong K(\bb{Z},2)$ and $BPU(\ca{H})\cong K(\bb{Z},3)$ hold, where $K(\bb{Z},3)$ and $K(\bb{Z},2)$ are Eilenberg MacLane spaces. Therefore, if $\tau\in H^3(X)$ is given, we have an associated $PU(\ca{H})$ principal bundle $Q$. And it determines a $\bb{P}(\ca{H})$ bundle $P:=Q\times _{PU({\ca{H}})} \bb{P}(\ca{H})$. We call such $P$ a projective bundle, that is, it is a $\bb{P}(\ca{H})$ bundle whose structure group is $PU(\ca{H})$.

Then we have a classifying bundle $Fred(P)$ of twisted $K$-group $K^{\tau+0}(X)$, where $Fred(P)|_x$ is the set of Fredholm operators on $\ca{H}_x$, a chosen lift of the fiber $P|_x$.
Twisted $K$-theory is defined as follows.
\begin{dfn}
The space of sections is defined by
$$\Gamma_c(X,Fred(P)):=\{s:X\to Fred(P)|supp(s)\text{ is compact}\},$$
where $supp(s):=\{x\in X|s_x\text{ is not invertible}\}$ is the support of $s$.

Twisted $K$-theory is defined by
$$K^{\tau+0}(X):=\pi_0(\Gamma_c(X,Fred(P)))$$
$$K^{\tau+k}(X)=K^{\tau-k}(X):=K^{p^*\tau+0}(X\times \bb{R}^k),$$
where $p:X\times \bb{R}^k\to X$ is the projection onto the first factor.
\end{dfn}
\begin{rmk}
Twisted $K$-theory is a two periodic generalized cohomology theory (\cite{AS}).
\end{rmk}

Equivariant version of twisted $K$-theory is defined just like untwisted cases.
Let $G$ be a compact Lie group acting on a topological space $X$ continuously. If we replace ``projective bundle'' with ``$G$-equivariant projective bundle'', ``sections'' with ``$G$-equivariant sections'' and ``homotopy'' with ``$G$-equivariant homotopy'', we obtain the twisted equivariant $K$-theory  $K_{G}^{\tau+k}(X)$, where $\tau\in H^3_G(X)$ is a $G$-equivariant twisting. Since this is a generalized cohomology theory for $G$-spaces, we have some tools to compute. In fact, Freed, Hopkins and Teleman computed $K_{G}^{\tau+k}(G)$ using spectral sequences or a ``Mackey decomposition'' (\cite{FHT1}, \cite{FHT3}) for $\tau$ which is induced by a positive central extension of $LG$, where $G$ acts on itself by conjugation. For simplicity, we deal with tori here. While this action of $T$ on itself is trivial, an action on a projective bundle can be non-trivial.
\begin{thm}[\cite{FHT1}, \cite{FHT3}]\label{computation of K_T(T)}
The following isomorphism holds.
$$
K^{\tau+k}_T(T)\cong\begin{cases}
0 & (k\equiv dim(T)+1\mod2) \\
\bb{Z}[\Lambda_T^\tau/\Pi_T] & (k\equiv dim(T)\mod2).
\end{cases}$$
The action of $\Pi_T$ on $\Lambda^\tau_T$ coincides with the one defined by the central extension $\tau$. $\bb{Z}[\Lambda_T^\tau/\Pi_T]$ is the $\bb{Z}$ free module generated by the set $\Lambda_T^\tau/\Pi_T$.
\end{thm}
We write the isomorphism of the above theorem as
$$M.d._T:K^{\tau+dim(T)}_T(T)\to\bb{Z}[\Lambda^\tau_T/\Pi_T].$$
 This notation comes from a ``Mackey decomposition''.

Combining Theorem \ref{computation of R(LT)} with \ref{computation of K_T(T)},
 we know that $R^\tau(LT)$ is isomorphic to $K^{\tau+dim(T)}_T(T)$.

Freed, Hopkins and Teleman constructed an isomorphism 
$$FHT_T:R^\tau(LT)\to K^{\tau+dim(T)}_T(T)$$
by use of the ``Dirac family''.

To explain it, let us explain ``local equivalence'' of ``groupoids''.
We start from a classical isomorphism for un-twisted equivariant $K$-theory. Let $G$ act on $X$, $N$ be a normal subgroup of $G$, and $N$ act on $X$ freely. Then we have an isomorphism $K_G(X)\cong K_{G/N}(X/N)$ (\cite{Segal} Proposition 2.1). In \cite{FHT1}, they dealt with such a phenomenon more generally. They defined twisted $K$-theory for groupoids. A groupoid is a generalization of a space equipped with an action of a Lie group. 
$S//G$ is a groupoid defined by $S$ which $G$ acts on.
For example, $K(S//G)\cong K_G(S)$. They verified that twisted $K$-theory is a local equivalence invariant. Local equivalence is an equivalence relation between two groupoids. In fact, if two groupoids $S_1//G_1$ and $S_2//G_2$ are locally equivalent, two orbit spaces $S_1/G_1$ and $S_2/G_2$ are homeomorphic. For example, $X//G$ and $(X/N)//(G/N)$ are locally equivalent in the above situation. By this construction, infinite dimensional spaces or Lie groups can be analyzed in some cases.

In fact, they constructed a local equivalence $hol:\ca{A}_{S^1\times T}//LT\to T//T$, the holonomy map, where $\ca{A}_{S^1\times T}$ is the set of connections over the trivial bundle $S^1\times T$. Let us recall that $LT$ acts on $\ca{A}_{S^1\times T}$ by the gauge transformation.
Passing through this local equivalence, elements of $K^{\tau+k}_T(T)$ can be obtained
from $LT$-equivariant families of Fredholm operators parametrized by $\ca{A}_{S^1\times T}$. If we have an irreducible positive energy representation of $LT$, we obtain a ``Dirac family'' which is a family of Fredholm operators parametrized by $\ca{A}_{S^1\times T}$(\cite{FHT2}). Since positive energy representations are projective, we need to twist equivaritant $K$-theory. By extending linearly, we obtain a group homomorphism 
$$FHT_T:R^{\tau}(LT)\to K^{\tau+dim(T)}_T(T).$$
In \cite{FHT2}, they verified that $FHT_T$ is an isomorphism by use of an explicit description of the both sides, Theorem \ref{computation of R(LT)} and Theorem \ref{computation of K_T(T)}.

Let us compare the common right hand side $\bb{Z}[\Lambda_T^\tau/\Pi_T]$ of Theorem \ref{computation of R(LT)}, \ref{computation of K_T(T)} with $R(T)\cong\bb{Z}[\Lambda_T]$, where we can see algebraic similarity between them. So we can define an analogue of the induced homomorphism of representation ring of tori. 

Firstly, we rewrite induced homomorphisms of representation ring of tori. Let $T$ and $T'$ be tori and $f:T'\to T$ be a group homomorphism. $\Lambda_T$ and $\Lambda_{T'}$ are the character groups of $T$ and $T'$ respectively. Representation rings of $T$ and $T'$ are isomorphic to the group rings of $\Lambda_T$ and $\Lambda_{T'}$ respectively. That is, $R(T)\cong \bb{Z}[\Lambda_T]$ and $R(T')\cong \bb{Z}[\Lambda_{T'}]$. $f^*:R(T')\to R(T)$ is defined by the pull back of representations. When we regard $\Lambda_T$ as ${\rm Hom}(\Pi_T,\bb{Z})$, $f^*(\lambda)$ corresponds to ${}^tdf(\lambda)$ ($\lambda\in\Lambda_{T'}$).
That is, if we have a representation $\rho:T\to U(n)$, we can decompose it as $\bigoplus_{i=1}^n\lambda_i$ for some $\lambda_i\in\Lambda_T$, and its image under $f^*$ is given by 
$$f^*(\rho)=f^*(\bigoplus_{i=1}^n\lambda_i)=
\sum_{i=1}^nf^*(\lambda_i)=\sum_{i=1}^n{}^tdf(\lambda_i).$$
In other words, if we regard $R(T)$ as ordinary $K$-group $K(\Lambda_T)$, $f^*$ corresponds to $({}^tdf)_!$, the push-forward, where $\Lambda_T$ has the discrete topology.

Let us define the induced homomorphism $char(f):\bb{Z}[\Lambda^\tau_{T}/\Pi_{T}]\to \bb{Z}[\Lambda^{f^*\tau}_{T'}/\Pi_{T'}]$.

\begin{dfn}
When $f^*\tau$ is also positive, we define $char(f)([\lambda]_T)$ by
$$char(f)([\lambda]_{T}):=\sum_{i=1}^k[\mu_i]_{T'},$$
where ${}^tdf(\Pi_{T}.\lambda)=\coprod_{i=1}^k\Pi_{T'}.\mu_i$, $[\lambda]_{T}$ is the $\Pi_{T}$-orbit including $\lambda\in\Lambda^\tau_{T}$ and $[\mu]_{T'}$ is the $\Pi_{T'}$-orbit including $\mu\in\Lambda^{f^*\tau}_{T'}$.
\end{dfn}
\begin{rmk}Well-definedness is verified in Section 3.\end{rmk}

If we introduce notations $char(T,\tau):=\bb{Z}[\Lambda^\tau_{T}/\Pi_{T}]$ and $char(T',f^*\tau):=\bb{Z}[\Lambda^{f^*\tau}_{T'}/\Pi_{T'}]$, $char(f)$ is a homomorphism from $char(T,\tau)$ to $char(T',f^*\tau)$. The correspondence $char$ seems to be a functor. However, it is not true. That is, there exist homomorphisms $T\xrightarrow{f}T'\xrightarrow{g}T''$ such that $char(g\circ f)\neq char(f)\circ char(g)$. We call such a correspondence between two categories a ``quasi functor''.

\begin{dfn}
Let $\ca{C}_1$ and $\ca{C}_2$ be categories. $F:\ca{C}_1\to\ca{C}_2$ is a quasi functor if the following data are given.

(1) For any $O\in Obj(\ca{C}_1)$, $F(O)\in Obj(\ca{C}_2)$.

(2) For any $f\in Mor_{\ca{C}_1}(O_1,O_2)$, $F(f)\in Mor_{\ca{C}_2}(F(O_2),F(O_1))$ satisfying $F(id)=id$.

That is, a quasi functor is a ``functor'' which is allowed $F(f_1\circ f_2)\neq F(f_2)\circ F(f_1)$.
\end{dfn}

However, the quasi functor $char$ holds functoriality we need. We verify the following theorem (Theorem \ref{decomposition of local injection}).%link
\begin{thm}\label{decomposition of a loca injection in Introduction}
For any local injection $f:T'\to T$ and positive central extension $\tau$ of $LT$, we can define the ``orthogonal completion torus'' $(T'/\ker(f))^\perp\subseteq T$ satisfying the following condition. Let $j:(T'/\ker(f))^\perp\hookrightarrow T$ be the natural inclusion. $f\cdot j:T'/\ker(f)\times(T'/\ker(f))^\perp\to T$ is defined by $f\cdot j(t_1,t_2)=f(t_1)j(t_2)$.

(1) $f\cdot j:T'/\ker(f)\times (T'/\ker(f))^\perp\to T$ is a finite covering.

(2) $(f\cdot j)^*\tau$ can be written as $p_1^*\tau_1+p_2^*\tau_2$, where $\tau_1$, $\tau_2$ are positive central extensions of $L(T'/\ker(f))$ and $L(T'/\ker(f))^\perp$ respectively.

Moreover, we can decompose $f$ as follows.
$$T'\xrightarrow{q}T'/\ker(f)\xrightarrow{i_1}T'/\ker(f)\times (T'/\ker(f))^\perp\xrightarrow{f\cdot j}T,$$
where $i_1$ is the natural inclusion into the first factor.
\end{thm}

``Functoriality we need'' is that $char(f)=char(q)\circ char(i_1)\circ char(f\cdot j)$ (Theorem \ref{partial functoriality}). From this theorem, it is sufficient to deal with only cases of direct products or finite coverings.\\

%Let us explain the outline of proof of our main theorem. We start from the explicit description of  $FHT_T$.
%\begin{thm}[\cite{FHT2}]\label{009}
%Let $T$ be a torus and $\tau$ be a positive central extension of $LT$. Then we have the following commutative diagram.
%$$\xymatrix{
%R^\tau(LT) \ar[rr]^{FHT_T} \ar[rd]^{l.w.} & & K_T^{\tau+dim(T)}(T) \ar[ld]^{M.d._T} \\
% & char(T,\tau) & 
%}$$
%\end{thm}

%The following two theorems are our main constructions. Let $T$, $T'$ be tori, $\tau$ be a positive central extension of $LT'$ and $f:T\to T'$ be a homomorphism such that $f^*\tau$ is also positive.

%\begin{thm}
%$$\begin{CD}
%K^{\tau +dim(T')}_{T'}(T') @>f^\#>> K^{f^*\tau+dim(T)}_T(T) \\
%@VM.d._{T'}VV @VM.d._TVV \\
%char(T',\tau) @>char(f)>> char(T,f^*\tau)
%\end{CD}$$
%\end{thm}
%\begin{rmk}
%It can be extended to $f:H\to G$ even if $H$ and $G$ not tori when the following conditions hold.

%(1) $H$, $G$ and $H/\ker(f)$ are compact connected Lie groups with torsion free $\pi_1$.

%(2) $f$ can be decompose $$H\xrightarrow{p}H/\ker(f)\xrightarrow{i_1}H/\ker(f)\times S^\perp\xrightarrow{q}G$$
%where $p$ is a natural projection, $i_1$ is the inclusion into the first factor, $q$ has the injective tangent map and $S^\perp$ is a torus associated with $\tau$ and $f$, whose dimension is $rank(G)-rank(H)$. We call $S^\perp$ the ``orthogonal completion torus'' in Section 3.
%\end{rmk}

%\begin{thm}
%$$\begin{CD}
%%R^{\tau}(LT') @>f^!>> R^{f^*\tau}(LT) \\
%@Vl.w._{T'}VV @Vl.w._TVV \\
%char(T',\tau) @>char(f)>> char(T,f^*\tau) 
%\end{CD}$$
%\end{thm}
%Theorem 0.2 is verified by the combining the above three theorems.

This paper consists of 5 sections.

In Section 1, we describe known results and definitions about twisted equivariant $K$-theory and positive energy representation of loop groups.

In Section 2, we define some categories and quasi functors and rewrite FHT isomorphism in our terminology. Then we verify our main theorem under the assumption that Theorem \ref{main theorem for K for tori} and \ref{main theorem for R for tori} hold.

In Section 3, we study the quasi functor $char$. We define the induced homomorphism $char(f)$, which is an analogue of pull back of representations of tori. The quasi functor $char$ is the most useful for computation. We deal with not only tori but also compact connected Lie groups with torsion-free $\pi_1$.

In Section 4, we study the quasi functor $t.e.K$, a modification of twisted equivariant $K$-theory. We define $f^\#$ so that it holds naturality with $char(f)$. We deal with not only tori but also compact connected Lie groups with torsion-free $\pi_1$.

In Section 5, we study the quasi functor $RL$, a modification of positive energy representation group of loop groups. We define $f^!$ so that it holds naturality with $char(f)$.

\tableofcontents

%知られた結果と定義%
\section{General definitions and known results}\label{General definition and theorems}
%Firstly, we explain definitions and known results about representation theory of loop groups and twisted equivariant $K$-theory.

%%%PERs

\subsection{Loop groups and their positive energy representations}\label{Positive energy representations of loop groups}
Representation group of $LT$ at level $\tau$ is one of our main objects, where $T$ is a torus.
In this section, we deal with not only tori but also general Lie group with torsion-free $\pi_1$. Firstly, we explain central extensions of loop groups.

\subsubsection{Definitions and some formulae about central extension}\label{definitions and some formulae about central extension}
Let $G$ be a compact connected Lie group with torsion-free $\pi_1$.
Let $LG$ be the group of smooth maps from $S^1$ to $G$, loop group of $G$. Its topology is defined so that $l_n$ converges to $l$ if and only if all derivatives $\frac{d^kl_n}{d\theta^k}$ converges to $\frac{d^kl}{d\theta^k}$ uniformly. This is one of the simplest infinite dimensional Fr\'echet Lie groups. $S^1$ acts on $LG$ via transformation of parameter, that is, $(\theta_0^*l)(\theta):=l(\theta-\theta_0)$ ($l\in LG$). When we make $S^1$ act on loop groups, we write $S^1$ as $\bb{T}_{rot}$.
So we can define the new Lie group $LG\rtimes \mathbb{T}_{rot}$.

The Lie algebra of $LG$ is the loop algebra $L\mathfrak{g}:=C^\infty(S^1,\mathfrak{g})$. Its addition, scaler multiplication and Lie bracket are induced by ones of $\mathfrak{g}$. We write the Lie algebra of $LG\rtimes \mathbb{T}_{rot}$ as $L\mathfrak{g}\oplus i\mathbb{R}_{rot}$. If $d$ is an infinitesimal generator of the action of $\mathbb{T}_{rot}$, $[d,\beta]=\frac{d\beta}{d\theta}$ for $\beta\in L\mathfrak{g}$. 

We need the following definition in order to define positeve energy representations of loop groups, because positive energy representations must be projective unless it is trivial.

\begin{dfn}[\cite{FHT2}]\label{admissible centaral extensions}
A central extension 
$$1\to U(1)\xrightarrow{i} LG^\tau\xrightarrow{p} LG\to1$$
is admissible if 

(1) The action of $\mathbb{T}_{rot}$ lifts to $LG^\tau$. Therefore, we can define the new Lie group $LG^\tau\rtimes \bb{T}_{rot}$.

(2) $LG^\tau\rtimes \mathbb{T}_{rot}$ acts on $L\mathfrak{g}^\tau\oplus i\mathbb{R}_{rot}$ via adjoint action. There exists an $LG^\tau\rtimes\mathbb{T}_{rot}$-invariant symmetric bilinear form $\Innpro{\cdot}{\cdot}{\tau}$ on $L\mathfrak{g}^\tau\oplus i\mathbb{R}_{rot}$ such that $\Innpro{K}{d}{\tau}=-1$, where $K$ is the infinitesimal generator of $i(U(1))$. If the restriction of $\Innpro{\cdot}{\cdot}{\tau}$ to $L\mathfrak{g}$ is positive definite, $\tau$ is called a positive central extension.
\end{dfn}
\begin{rmk}
Bilinear form $\Innpro{\cdot}{\cdot}{\tau}$ determines a splitting $L\fra{g}\to L\fra{g}^\tau\oplus i\bb{R}_{rot}$. So we can ``restrict'' $\Innpro{\cdot}{\cdot}{\tau}$ to $L\fra{g}$ (\cite{FHT2} Lemma 2.18).
\end{rmk}

\begin{dfn}\label{definition of pull back of central extension of loop groups}
Let $f:H\to G$ be a smooth group homomorphism, then the homomorphism $Lf:LH\to LG$ is defined by $l\mapsto f\circ l$. We can pull back the central extension $\tau$ of $LG$ along $Lf$ and we write it as $f^*\tau$. That is, the following commutative diagram holds.
$$\xymatrix{
U(1) \ar@{=}[d] \ar[r]^i & LG^\tau \ar[r]^p & LG \\
U(1) \ar[r]^i & LH^{f^*\tau} \ar[r]^p \ar[u]^{\widetilde{Lf}} & LH \ar[u]^{Lf}
}$$
\end{dfn}

\begin{lem}\label{admissibility of f^*tau}
If $\tau$ is admissible, so is $f^*\tau$.
\end{lem}
\begin{proof}
Let us notice that $\widetilde{Lf}$ can be extended to $LH^{f^*\tau}\rtimes \bb{T}_{rot}$ by $\widetilde{Lf}|_{\bb{T}_{rot}}:=id_{\bb{T}_{rot}}$. We use the same character for this extension.

(1) is clear from the definition.

(2) is verified by defining the bilinear form on $L\mathfrak{g}^\tau\oplus i\mathbb{R}_{rot}$ explicitly as follows. $d\widetilde{Lf}$ is the tangent map of $\widetilde{Lf}$.
$$\Innpro{v}{w}{f^*\tau}:=\Innpro{d\widetilde{Lf}(v)}{d\widetilde{Lf}(w)}{\tau}.$$
Since $\widetilde{Lf}|_{U(1)}$ is the identity, $d\widetilde{Lf}(K)=K$. Moreover,  $d\widetilde{Lf}(d)=d$ by the definition. Therefore, the above condition is satisfied.
\end{proof}
%from this remark, we can verify the following immediately
%\begin{cor}
%If $\tau$ is positive, so is $f^*\tau$.
%\end{cor}
\begin{dfn}
Let $\tau_1$ and $\tau_2$ be admissible central extensions of $LG$. Then we can define the tensor product between them by
$$LG^{\tau_1+\tau_2}:=LG^{\tau_1}\otimes LG^{\tau_2}=LG^{\tau_1}\times LG^{\tau_2}/LG.$$
\end{dfn}
\begin{lem}[\cite{FHT2} Section 2.2]
If $\tau_1$ and $\tau_2$ are admissible central extensions of $LG$, so is $\tau_1+\tau_2$.
Invariant bilinear form can be defined by $$\Innpro{\cdot}{\cdot}{\tau_1+\tau_2}=\Innpro{\cdot}{\cdot}{\tau_1}+\Innpro{\cdot}{\cdot}{\tau_2}.$$
\end{lem}
\begin{rmk}\label{relation with affine Lie algebra}
Lie algebra $L\mathfrak{g}^\tau\oplus i\mathbb{R}_{rot}$ for simple $\mathfrak{g}$ and the ``universal central extension'' $\tau$ can be thought as a ``completion'' of the affine algebra $\widehat{\mathfrak{g}}$.
From this view point, bilinear form $\Innpro{\cdot}{\cdot}{\tau}$ can be thought as a lift of Killing form of $\mathfrak{g}$ to $\widehat{\fra{g}}$.
\end{rmk}

When $G$ is a torus $T$, admissibility implies the following lemma. However, we need some terminologies about tori or their representations.

$\Pi_T:=\ker(\exp:\mathfrak{t}\to T)\cong\pi_1(T)\cong H_1(T)$ is a lattice in $\mathfrak{t}$.

$\Lambda_T:={\rm Hom}(\Pi_T,\mathbb{Z})\cong {\rm Hom}(T,U(1))\cong H^1(T)$ is the character group. We can regard it as the set of irreducible representations of $T$ from Schur's lemma (any irreducible representation of commutative group is $1$ dimensional).

\begin{lem}[\cite{FHT2} Proposition 2.27]\label{factorization of LT}
$LT$ has a canonical decomposition $LT\cong T\times \Pi_T\times U$, where $T$ is the set of initial values of loops, $\Pi_T$ is the set of ``rotation numbers''(naturally isomorphic to $\pi_1(T)\cong\pi_0(LT)$) and 
$$U:=\exp\biggl\{\beta :S^1\to\mathfrak{t}\bigg|\int_{S^1}\beta(s)ds=0\biggr\}$$
is the set of derivatives of contractible loops whose initial values are $0$.

Let $\tau$ be an admissible central extension of $LT$, then the above decomposition is inherited partially, that is $LT^{\tau}\cong (T\times\Pi)^{\tau}\otimes U^{\tau}$.
\end{lem}

This lemma allows us to define $\kappa^\tau$ and describe some formulae. $\kappa^\tau$ is the tangent map of $\widetilde{\kappa^\tau}$ defined below.

\begin{dfn}[%homomorphism $\kappa^\tau$, 
\cite{FHT2} Proposition 2.27]\label{morphism kappa}

A homomorphism $\widetilde{\kappa^\tau}:\Pi_T\to{\rm Hom}(T,U(1))$ is defined by
$$\widetilde{\kappa^\tau}(X)(t):=\widetilde{\phi}_X\widetilde{t}\widetilde{\phi}_X^{-1}\widetilde{t}^{-1}\in i(U(1)),$$
where $\widetilde{t}$ is a chosen lift of $t\in T$, and $\widetilde{\phi}_X$ is a chosen lift of $\phi_X\in LT$ which is a geodesic loop corresponding to $X\in\Pi_T\cong\pi_1(T)$ whose initial value is $0$. This definition is independent of the choices of lifts because they are determined up to $i(U(1))$ which is in the center of $LT^\tau$.
\end{dfn}
This definition means that $\kappa^\tau$ represents how non-commutativity is caused by the central extension.

\begin{lem}[\cite{FHT2} Proposition 2.27]\label{property of kappa^tau}
$\widetilde{\kappa^\tau}$ determines a symmetric bilinear form $\innpro{n}{m}{\tau}=-\frac{1}{\sqrt{-1}}\dot{\widetilde{\kappa^\tau}}(n)(m)$ on $\Pi_T$ for $n,m\in\Pi_T$.

Moreover, when we extend $\innpro{\cdot}{\cdot}{\tau}$ to $\mathfrak{t}$, $\innpro{v}{w}{\tau}= \Innpro{v}{w}{\tau}$.
\end{lem}
%Def? Prop?
\begin{pro}\label{positivity of central extension}
If a central extension $\tau$ of $LT$ is positive, so is $\innpro{\cdot}{\cdot}{\tau}$.
\end{pro}
%We can extend this definition for any $G$.
%\begin{dfn}
%A central extension $\tau$ of $LG$ is non-degenerate/positive if the restriction of $\tau$ to %$LT$ is non-degenerate/positive, where $T$ is a chosen maximal torus of $G$.
%\end{dfn}

We can verify the following formulae.

Let $f:T'\to T$ be a local injection, that is, the tangent map $df$ is injective. Then we have the associated homomorphism $Lf:LT'\to LT$ and $\widetilde{Lf}:LT'^{f^*\tau}\to LT^{\tau}$.
\begin{lem}\label{pull back formula of kappa}
$$\kappa^{f^*\tau}={}^tdf\circ\kappa^{\tau}\circ df$$
\end{lem}
\begin{proof}
Let $X\in\Pi_{T'}$, $t\in T'$. Take lifts $\widetilde{\phi_X}\in LT'^{f^*\tau}$ and $\widetilde{t}\in LT'^{f^*\tau}$ of $\phi_X\in LT'$ and $t$. %, where $\phi_X$ is a geodesic loop corresponding to $X$
Let us notice that lifts of $\phi_{df(X)}$ and $f(t)$ can be chosen by $\widetilde{Lf}(\phi_X)$ and $\widetilde{Lf}(\widetilde{t})$ respectively.
 By the definition, $\widetilde{Lf}|_{i(U(1))}$ is the identity. Therefore, we can compute $\widetilde{\kappa^{f^*\tau}}_X(t)$ using $\widetilde{\kappa^\tau}$ as follows.
$$\widetilde{\kappa^{f^*\tau}}_X(t)=\widetilde{Lf}(\widetilde{\kappa^{f^*\tau}}_X(t))$$
$$=\widetilde{Lf}(\widetilde{\phi_X}\widetilde{t}\widetilde{\phi_X}^{-1}\widetilde{t}^{-1})=
\widetilde{Lf}(\widetilde{\phi_X})\widetilde{Lf}(\widetilde{t})\widetilde{Lf}(\widetilde{\phi_X}^{-1})\widetilde{Lf}(\widetilde{t}^{-1})$$
$$=\widetilde{\phi_{df(X)}}\widetilde{f(t)}\widetilde{\phi_{df(X)}}^{-1}\widetilde{f(t)}^{-1}
=\widetilde{\kappa^\tau}_{df(X)}(f(t))$$
By taking the tangent map, we obtain the conclusion.

\end{proof}

From this lemma, we can verify the following one.

Let $T_1$ and $T_2$ be tori, $T:=T_1\times T_2$, $i_j:T_j\to T$ be the natural inclusion into the $j$'th factor, $p_j:T\to T_j$ be the projection onto the $j$'th factor and $\tau_j$ be an admissible central extension of $LT_j$ ($j=1,2$). Then $\tau=p_1^*\tau_1+p_2^*\tau_2$ is an admissible central extension of $LT$.

In this case we have canonical isomorphisms $LT\cong LT_1\times LT_2$ and $LT^\tau\cong LT_1^{\tau_1}\times LT_2^{\tau_2}/U(1)$.
Therefore, for any $l_1,l_1'\in LT_1^{\tau_1}$, $l_2, l_2'\in LT_2^{\tau_2}$, $l_1\otimes l_2\cdot l_1'\otimes l_2'=l_1l_1'\otimes l_2l_2'$.
Moreover, $LT^\tau\cong U_1^{\tau_1}\otimes U_2^{\tau_2}\otimes(T_1\times\Pi_{T_1})^{\tau_1}\otimes(T_2\times\Pi_{T_2})^{\tau_2}$.

\begin{lem}%[product formula for $\kappa^\tau$]
\label{product formula for kappa}
$\kappa^\tau:\Pi_{T_1\times T_2}\to\Lambda_{T_1\times T_2}$ is the composition of 
$$\Pi_{T_1\times T_2}\xrightarrow{dp_1\oplus dp_2}\Pi_{T_1}\oplus\Pi_{T_2}\xrightarrow{\kappa^{\tau_1}
\oplus\kappa^{\tau_2}}\Lambda_{T_1}\oplus\Lambda_{T_2}\xrightarrow{{}^tdp_1\oplus {}^tdp_2}\Lambda_{T_1\times T_2}.$$

Formally
$$\kappa^\tau=\begin{pmatrix}
\kappa^{\tau_1} & 0 \\
0 & \kappa^{\tau_2}
\end{pmatrix}
$$
\end{lem}
%\begin{proof}

%$X\in\Pi, t\in T$. Let $\phi_X\in LT$ be the corresponding geodesic loop. 

%$X_i:=dp_i(X)\in\Pi_i$, $t_i:=p_i(t)\in T_i (i=1,2)$. Choose lifts $\widetilde{\phi}_{X_i}\in LT^{\tau_i}$ of $\phi_{X_i}$ and $\widetilde{t}_i\in T_i^{\tau_i}\subseteq LT_i^{\tau_i}$ (i=1,2). We can take the associated lifts of $\phi_{X_1+X_2}$ and $t_1t_2$ by $\widetilde{\phi}_{X_1}\widetilde{\phi}_{X_2}$ and $\widetilde{t}_1\widetilde{t}_2$.

%From the above commutativity,
%$$\kappa^\tau(X_1+X_2)(t_1t_2):=\widetilde{\phi}_{X_1+X_2}\widetilde{t_1t_2}\widetilde{\phi}_{X_1+X_2}^{-1}\widetilde{t_1t_2}^{-1}$$
%$$=\widetilde{\phi}_{X_1}\widetilde{\phi}_{X_2}\widetilde{t_1}\widetilde{t_2}\widetilde{\phi}_{X_2}^{-1}\widetilde{\phi}_{X_1}^{-1}\widetilde{t_2}^{-1}\widetilde{t_1}^{-1}$$
%$$=\widetilde{\phi}_{X_1}\widetilde{t_1}\widetilde{\phi}_{X_1}^{-1}\widetilde{t_1}^{-1}\widetilde{\phi}_{X_2}\widetilde{t_2}\widetilde{\phi}_{X_2}^{-1}\widetilde{t_2}^{-1}$$
%$$=\kappa^{\tau_1}(X_1)(t_1)\kappa^{\tau_2}(X_2)(t_2).$$

%\end{proof}

We often use these formulae in this paper.
%%正エネルギー表現とその分類定理
\subsubsection{Positive energy representation}\label{positive energy representation}
We can define positive energy representations, which is one of the main objects in this paper.

Let $U(1)\xrightarrow{i} LG^\tau\xrightarrow{p} LG$ be an admissible central extension.

\begin{dfn}[%Positive energy representation 
\cite{PS}\cite{FHT2}]\label{def. of PER}
A continuous homomorphism $\rho:LG^\tau\to U(V)_{c.o.}$ is called a positive energy representation at level $\tau$ if

(1) $\rho\circ i(e^{2\pi\sqrt{-1}})=e^{2\pi\sqrt{-1}}id_V$ (the definition of $\tau$-twisted representations).

(2) $\rho$ lifts to a continuous homomorphism $\widehat{\rho}:LG^\tau\rtimes\mathbb{T}_{rot}\to U(V)_{c.o.}$.

(3) $E:=\frac{1}{\sqrt{-1}}\dot{\widehat{\rho}}(d)$ which is called an energy operator is self-adjoint with discrete spectrum bounded below.

$V$ is a separable complex Hilbert space as a representation space, $U(V)_{c.o.}$ is topologized by the compact open topology (or the strong topology) and $d$ is the infinitesimal generator of $\mathbb{T}_{rot}$.
\end{dfn}
A positive energy representation is said to be finitely reducible if it is a finite sum of irreducible representations.
%\begin{pro}[\cite{PS}]
%Positive energy representations are completely reducible.
%\end{pro}

Isomorphism classes of finitely reducible positive energy representations at level $\tau$ form an abelian semigroup under the direct sum.
\begin{dfn}
$R^\tau(LG)$ is the Grothendieck completion of this semigroup. In other words, $R^\tau(LG)$ is the $\mathbb{Z}$ free module generated by the set of isomorphism classes of irreducible positive energy representations. We call it a representation group of $LG$ at level $\tau$.
\end{dfn}
\begin{rmk}
%(1) The above proposition implies that $R^\tau(LG)$ is a $\bb{Z}$ free module generated by the set of irreducible positive energy representations at level $\tau$.

In fact, it has a ring structure and it is called the ``Verlinde ring''. The product is given by the fusion product of conformal field theory. However, we do not deal with the product in this paper.
\end{rmk}
Positive energy representations at level $\tau$ have been classified (\cite{PS}).

Let $G$ be a connected compact Lie group with torsion-free $\pi_1$ and $\tau$ be a positive central extension of $LG$. We fix a maximal torus $T$ of $G$.

We need to introduce some terminologies to describe the following theorems.

We can define central extensions $G^\tau$ and $T^\tau$ by the restrictions of $LG^\tau$ to $G$ and $T$. That is, the commutative following diagram holds.

$$\xymatrix{
1 \ar[r] \ar@{=}[d] & U(1) \ar[r]^i \ar@{=}[d] & T^\tau \ar[r]^p \ar[d] & T \ar[r] \ar[d] & 1 \ar@{=}[d] \\
1 \ar[r] \ar@{=}[d] & U(1) \ar[r]^i \ar@{=}[d] & G^\tau \ar[r]^p \ar[d] & G \ar[r] \ar[d] & 1 \ar@{=}[d] \\
1 \ar[r] & U(1) \ar[r]^i  &LG^\tau \ar[r]^p  & LG \ar[r]  & 1 }$$

We can verify that $T^\tau$ is a maximal torus of $G^\tau$.

So we can define a $\tau$-twisted representation of $T^\tau$. A $\tau$-twisted representation is a homomorphism $\rho:T^\tau\to U(V)$ such that $\rho\circ i(e^{\sqrt{-1}\theta})=e^{\sqrt{-1}\theta}id_V$, where $V$ is a representation space.
 $\Lambda_T^\tau$ is the set of irreducible $\tau$-twisted representations. It can be thought as the set of $\lambda:\Pi_{T^\tau}\to\mathbb{Z}$ such that the composition of $\mathbb{Z}=\Pi_{U(1)}\xrightarrow{di}\Pi_{T^\tau}\xrightarrow{\lambda}\mathbb{Z}$ is identity. We can define an action of $\Lambda_T$ on $\Lambda_T^\tau$ by the tensor product of the representations. We write it $\lambda+\lambda'$ ($\lambda\in\Lambda_T$, $\lambda'\in\Lambda_T^\tau$). The tensor product of the representations is defined by
$$(\lambda+\lambda')(t)=\lambda(p(t))\cdot\lambda'(t),$$
where $t\in T^\tau$. Since $p\circ i(z)=1$, $\lambda+\lambda'$ is also a $\tau$-twisted representation.

The Weyl group of $G$ and $G^\tau$ are naturally isomorphic. We write them $W(G)$. Since $W(G)$ acts on $\Pi_T$, we can define the semi direct product 
$$W^e_{\mathit{aff}}(G):=\Pi_T\rtimes W(G),$$
we call it the extended affine Weyl group. It is the Weyl group of $LG^\tau$, whose maximal torus is $T^\tau$. Since $W(G)$ acts on $\Pi_{T^\tau}$, $W(G)$ acts on  $\Lambda_T^\tau$.
And $\Pi_T$ acts on $\Lambda_T^\tau$ through the homomorphism $\kappa^\tau$ and $\Lambda_T\curvearrowright\Lambda_T^\tau$. Therefore, $W^e_{\mathit{aff}}(G)$ acts on $\Lambda_T^\tau$. We write the orbit space $\Lambda_T^\tau/\kappa^\tau(W^e_{\mathit{aff}}(G))$.

%The homomorphism $\kappa^\tau:\Pi_T\to\Lambda_T$ is defined in Definition \ref{morphism kappa}. It represents how the central extension causes non-commutativity.

$\mathbb{Z}[(\Lambda_T^\tau/\kappa^\tau(W^e_{\mathit{aff}}(G)))_\mathit{reg}]$ is the free $\mathbb{Z}$ module generated by the set of $W^e_{\mathit{aff}}(G)$-regular orbits $\Lambda_T^\tau/\kappa^\tau(W^e_{\mathit{aff}}(G))$. A regular orbit is an orbit whose stabilizer is trivial.

Let $\sigma$ be the spin extension of $LG$ defined in \cite{FHT2}. It is defined by the irreducible Clifford module of $L\mathfrak{g}^*$ and the coadjoint action of $LG$ on $L\fra{g}^*$. Let us suppose that $\tau-\sigma$ is positive.
\begin{lem}[\cite{Adams}]\label{rho-shift}
The map
$$\rho-\mathit{shift}:\Lambda_T^{\tau-\sigma}/\kappa^{\tau-\sigma}(W^e_{\mathit{aff}}(G))
\xrightarrow{\cong}(\Lambda_T^{\tau}/\kappa^{\tau}(W^e_{\mathit{aff}}(G)))_{\mathit{reg}}$$ defined by $[\lambda]\mapsto[\lambda+\rho]$ is bijective, where $\rho$ is the half of the sum of the positive roots of $G$.
\end{lem}
\begin{rmk}
When $G$ is a torus $T$, $W^e_{\mathit{aff}}(T)=\Pi_T$ and $\rho=0$. Since any orbit is regular, the above lemma is trivial.
\end{rmk}

\begin{thm}[\cite{FHT2}]\label{classifying theorem of PER in Section 1}
Isomorphism classes of positive energy $(\tau-\sigma)$-twisted representations of $LG^{\tau-\sigma}$ are parametrized by the $W^e_{\mathit{aff}}(G)$-orbits in $\Lambda_T^{\tau-\sigma}$ by taking the lowest weight.
From Lemma \ref{rho-shift}, we have the following isomorphisms
$$R^{\tau-\sigma}(LG)\xrightarrow{L.W._G}\mathbb{Z}[\Lambda_T^{\tau-\sigma}/W^e_{\mathit{aff}}(G)]\xrightarrow{\rho-\mathit{shift}}\mathbb{Z}[(\Lambda_T^{\tau}/W^e_{\mathit{aff}}(G))_{\mathit{reg}}].$$
$L.W._G$ is defined as follows. If $V$ is an irreducible positive energy representation of $LG$ at level $\tau-\sigma$, $L.W._G(V)$ is the lowest weight of $V$.
\end{thm}
\begin{dfn}
$$l.w._G:=\rho-\mathit{shift}\circ L.W._G$$
\end{dfn}
Since $W^e_{\mathit{aff}}(G)=\Pi_T\rtimes W(G)$ is the Weyl group of $LG$ or $LG^{\tau-\sigma}$, this theorem is an analogue of the Cartan Weyl's highest weight theory for compact Lie groups.

%同変ねじれK理論
\subsection{Twisted equivariant $K$-theory}\label{Twisted equivariant K-theory1}
\subsubsection{The definitions and the classification of twistings}\label{definitions and the classification of twistings}
%Twisted $K$-theory was introduced by Donovan and Karoubi to extend Thom isomorphism for non $K$-orientable vector bundles. $K$-orientable means that the rank is even and has a $Spin^c$ structure. It was rekindled by its appearance in the late 1990's in string theory.

Let $X$ be a manifold, $P$ be a projective bundle. We can define twisted $K$-theory associated with it following \cite{AS}. We do not deal with non-trivial gradings, that is, projective bundle $P$ has two disjoint sub projective bundles $P_1$ and $P_2$ such that $\mathcal{H}|_x=\mathcal{H}_{1}|_{x}\oplus \mathcal{H}_{2}|_x$, where $\mathcal{H}_j|_x$ is a chosen lift of $P_j|_x$ ($j=1,2$, $x\in X$). This assumption is automatically satisfied for the space whose first cohomology group with $\mathbb{Z}_2$ coefficient is zero. An isomorphism class of projective bundles is called a twisting. We define the classifying bundle $Fred^{(0)}(P)$ for twisted $K$-theory. Firstly, we define the space of Fredholm operators of a projective space.

\begin{dfn}\label{def. of Fred(P)}
$Fred^{(0)}(\bb{P}(\mathcal{H}))$ is the space of odd skew-adjoint Fredholm operators $A:\mathcal{H}\to\mathcal{H}$, for which $A^2+1$ is a compact operator. It is topologized so that $A\mapsto(A, A^2+1)\in B(\mathcal{H})_{c.o.}\times K(\mathcal{H})_{norm}$ is continuous.
\end{dfn}

\begin{rmk}\label{well-def. of Fred(P)}
(1) If we take an another lift $\mathcal{K}$ of $\bb{P}(\ca{H})$, we can take the unique isomorphism $\phi:\mathcal{H}\to\mathcal{K}$ up to $U(1)$ action. $\phi$ makes the following diagram commute.
$$\xymatrix{
\mathcal{H} \ar[rr]^\phi \ar[dr] & & \mathcal{K} \ar[dl] \\
& P(\mathcal{H}) & 
}$$

So if $A:\mathcal{H}\to\mathcal{H}$ is a linear map, $\phi\circ A \circ\phi^{-1}:\mathcal{K}\to\mathcal{K}$ is independent of the choice of $\phi$, so the above is well-defined.

(2) For the same reason we can take tensor products $P_1\widehat{\otimes} P_2$, $P\widehat{\otimes}\mathcal{L}$ and so on, where $\mathcal{L}$ is a Hilbert bundle (can be finite rank).
That is, $P_1\widehat{\otimes}P_2|_x:=\bb{P}(\ca{H}_1|_x\widehat{\otimes}\ca{H}_2|_x)$ and  $P\widehat{\otimes}\ca{L}|_x:=\bb{P}(\ca{H}|_x\widehat{\otimes}\ca{L}|_x)$, where $\ca{H}|_x$, $\ca{H}_1|_x$ and $\ca{H}_2|_x$ are chosen lifts of $P$ , $P_1$ and $P_2$ respectively.
$\bb{P}(V)$ is the projectification of the vector space $V$.
\end{rmk}

Let us define twisted $K$-theory. Let $X$ be a manifold, $X'$ be a closed subset of $X$.
\begin{dfn}%[twisted $K$-theory]
\label{def. of twisted K-theoy}
Let $P$ be a projective bundle. The classifying bundle is defined by
$$Fred^{(0)}(P):=\bigcup_{x\in X} Fred^{(0)}(P_x\otimes l^2).$$
We write the space of sections as
$$\Gamma_{c}(X,X',Fred^{(0)}(P))$$
$$:=\{s:X\to Fred^{(0)}(P) | \text{continuous section with compact support and }supp(s)\cap X'=\emptyset\},$$
where $supp(s):=\{x\in X | s_x \text{ is not invertible}\}=\{x\in X|ker(s_x)\neq0\}$ is called the support of $s$.
We define twisted $K$-theory as a generalized cohomology theory.
$$K^0_{P}(X,X'):=\pi_0(\Gamma_{c}(X,X',Fred^{(0)}(P))).$$
For $n\neq 0$,
$$K^n_P(X,X'):=\begin{cases}
K_P^0(X\times [0,1]^{-n},X\times\partial[0,1]^{-n}\cup X'\times[0,1]^{-n}) & n<0 \\
K^{-n}_P(X,X') & n>0
\end{cases}$$
\end{dfn}

\begin{rmk}
The functor $K^n_P$ has the classifying bundle whose fiber is the space of Fredholm operators satisfying some conditions of a projective space (\cite{FHT1}). Therefore, we regard an element of $K^n_P(X)$ as a homotopy class of a family of Fredholm operators.
\end{rmk}

Let $X$ and $Y$ be smooth manifolds, $X'$ and $Y'$ be closed subset of $X$ and $Y$ respectively, $F:Y\to X$ be a smooth proper map such that $F(X')\subseteq Y'$, and $P$ be a projective bundle over $X$.
\begin{dfn}%[pull back]
\label{def. of pull back}
$$F^{*}:K^{k}_{P}(X,X')\to K^k_{F^*P}(Y,Y')$$
is defined by $F^*([s]):=[s\circ F]$, where $[\cdot]$ means a homotopy class of sections.
\end{dfn}
\begin{thm}[\cite{AS}]
Twisted $K$-theory is a two periodic generalized cohomology theory and. That is, $K^{k}_P(X)\cong K^{k+2}_P(X)$ (Bott periodicity).
\end{thm}

%\begin{dfn}[\cite{AS}]
%An isomorphism class of projective bundles is called a twisting.
%\end{dfn}
Twistings are classified in \cite{AS}
\begin{pro}[\cite{AS} Proposition 2.1, 2.2]\label{classification theorem of unequivariant twistings}

(1) Isomorphism classes of twistings are classified by $H^3(X,\mathbb{Z})$ by taking Dixmier-Douady class.

(2) The group of connected components of the group of automorphisms of $P$ is isomorphic to $H^2(X,\bb{Z})$ which classifies isomorphism classes of complex line bundles. 

(3) For $l\in H^2(X,\bb{Z})$, the action of $l$ on $K_P^k(X)$ is given by the tensor product with $L$, where $L$ is a complex line bundle whose first Chern class is $l$ and we regard it as an element of $K(X)$.

%(3) If $l:P_1\xrightarrow{\cong}P_2$ is an isomorphism between two bundles we can define induced isomorphism between two twisted $K$-theory $K_{P_1}^n(X)$ and $K_{P_2}^n(X)$ by tensor product with corresponding line bundle $L$ whose first Chern class is $l$.

\end{pro}
\begin{rmk}
Since $PU(\ca{H})$ is an Eilenberg MacLane space $K(\bb{Z},2)$, the classifying space of it is homotopic to $K(\bb{Z},3)$. It is the classifying space of the functor $H^3$. Therefore, (1) holds. Let $\alpha: P\to P$ be an isomorphism, then $\alpha:P_x\to P_x$ has a lift
$\widetilde{\alpha_x}:\ca{H}_x\to\ca{H}_x$ up to $U(1)$. Therefore, $\alpha$ determines a $U(1)$ principal bundle. Therefore, (2) holds. (3) tells us that $H^2(X,\bb{Z})$ acts on $K^n_P(X)$ non-trivially.
\end{rmk}

\begin{ex}\label{example local spinor}
Let $X$ be a compact manifold. Let $p:E\to X$ be an oriented real vector bundle whose rank is $k$ over $X$. Let us take a local trivialization $\{E|_{U_\alpha}\cong U_\alpha\times\bb{R}^k\}_{\alpha}$.
Since any oriented vector space $V$ has the unique irreducible Clifford module $\Delta_V$ (\cite{Furuta}), we can define a ``local Spinor bundle''. That is, %when we trivialize $E$ locally, 
 $S|_{U_\alpha}:=U_\alpha\times \Delta_{\bb{R}^k}$.
From Schur's lemma, we can take a transition function $\phi_{\alpha\beta}:U_\alpha\cap U_\beta \to GL(\Delta_{\bb{R}^k})$ up to $\bb{C}^*$. Therefore, the family $\{\bb{P}(S|_{U_\alpha})\}_\alpha$ defines a projective bundle whose fiber is $\bb{P}(\Delta_{\bb{R}^k})$. Dixmier-Douady class of this bundle is the image of $w_2(E)$ under the Bockstein operator $\beta:H^2(X,\bb{Z}_2)\to H^3(X,\bb{Z})$. This is the obstruction for existence of a $Spin^c$ structure. $\beta(w_2(E))$-twisted $K$-theory is studied in \cite{Karoubi} and \cite{Furuta}. $K^{\beta(w_2(E))}(X)$ is written by $K^E(X)$ in \cite{Karoubi}, $K(X,Cl(E))$ in \cite{Furuta}.
Thom isomorphism states that a natural homomorphism $K(X)\to K^{\beta(w_2(p^*E))}(E)$ is isomorphic.
\end{ex}

Twisted $K$-theory has an equivariant version just like ordinary case. Let $G$ be a compact Lie group and act on $X$ smoothly. We explain only the changes or additional points from the non-equivariant cases. Let $X'$ be a closed $G$-invariant subset.

\begin{dfn}%[twisted equivariant $K$-theory]
\label{def. of twisted equivariant K-theoy}
Let $P$ be a $G$-equivariant projective bundle. An isomorphism class of $G$-equivariant projective bundles is called a $G$-equivariant twisting. Then the classifying bundle of twisted equivariant $K$-theory is defined by
$$Fred^{(0)}_G(P):=\bigcup_{x\in X} Fred^{(0)}(P_x\otimes L^2(G) \otimes l^2).$$
Since $G$ acts on $L^2(G)$ via the left regular representation and on $l^2$ trivially, we can define the space of $G$-equivariant sections.
$$\Gamma_{c,G}(X,X',Fred^{(0)}_G(P)) $$
$$:=\{s:X\to Fred^{(0)}_G(P) | \text{continuous, equivariant, compactly supported and }supp(s)\cap X'=\emptyset\}.$$

We define twisted equivariant $K$-theory as a generalized cohomology theory for $G$-spaces.
$$K^0_{G,P}(X,X'):=\pi_0(\Gamma_{c,G}(X,X',Fred^{(0)}_G(P)))$$
$K_{G,P}^n(X,X')$ is defined just like the non-equivariant cases for $n\neq 0$. ``Compactly supported'' means that the quotient space $supp(s)/G$ is compact.

Bott periodicity remains valid.
\end{dfn}
\begin{rmk}
Any finite dimensional representation space of $G$ can be embedded into $L^2(G)\otimes l^2$ from Peter-Weyl theorem.
\end{rmk}
Let $X$ be a smooth manifold and $G$ act on $X$ smoothly. Let $f:H\to G$ be a smooth group homomorphism. Let $P$ be a $G$-equivariant projective bundle over $X$, which is also $H$-equivariant through $f$. When we regard $P$ as an $H$-equivariant one, we write it $f^{\natural}P$.
\begin{dfn}%[restriction of group action]
\label{def. of restr. of group action}
$$f^{\natural}:K^k_{G,P}(X)\to K^k_{H,f^\natural P}(X)$$
is defined by regarding $G$-equivariant class $[s]$ as an $H$-equivariant one through $f$.
\end{dfn}

%\begin{dfn}
%An isomorphism class of $G$-equivariant projective bundles is called a $G$-equivariant twisting.
%\end{dfn}

Classification of $G$-equivariant twistings has been known.
\begin{pro}[\cite{AS} Proposition 6.3]\label{classification of twistings}

(1) Isomorphism classes of twistings are classified by $H^3_G(X,\mathbb{Z})$ by taking Dixmier-Douady class.

(2) The group of connected components of the group of automorphisms of $P$ is isomorphic to $H_G^2(X,\bb{Z})$ which classifies isomorphism classes of $G$-equivariant complex line bundles. 

(3) For $l\in H_G^2(X,\bb{Z})$, the action of $l$ on $K_{G,P}^k(X)$ is given by the tensor product with $L$, where $L$ is a complex line bundle whose first equivariant Chern class is $l$ and we regard $L$ as an element of $K_G(X)$.

%(3) If $l:P_1\xrightarrow{\cong}P_2$ is an isomorphism between two bundles we can define induced isomorphism between two twisted $K$-theory $K_{G,P_1}^n(X)$ and $K_{G,P_2}^n(X)$ by tensor product with corresponding line bundle $L$ whose first equivariant Chern class is $l$.

\end{pro}
Just like non-equivariant cases, $H^2_G(X)$ acts on $K^{k}_{G,P}(X)$ non-trivially.
%The above proposition means that twisted (equivariant) $K$-theory $K_{(G)}^{\tau+n}$ is determined as an isomorphism class of groups if we take only cohomology class $\tau$. However $G$-equivariant projective bundles we deal with in this paper are obtained from the following lemma. 
%So we regard projective bundle to be defined naturally, we write twisted equivariant $K$-group $K^{\tau+n}_G(X,X')$ simply.

$G$-equivariant twisting over $G$ we deal with in this paper can be obtained by the following lemma, where $G$ acts on itself by conjugation.

From now on, we write $K^k_{G,P}(X)$ as $K^{\tau+k}_G(X)$, where $\tau$ is Dixmier-Douady class of a projective bundle $P$.
\begin{lem}[\cite{AS}]\label{twisting induced by central extension}
If a central extension $LG^\tau$ has a $\tau$-twisted representation, $\tau$ gives a $G$-equivariant twisting over $G$.
\end{lem}
\begin{proof}
Firstly, we explain the relationship between projective representations of $LG$ and central extensions of $LG$.

Let us assume that $LG$ has a projective representation $\rho:LG\to PU(V)$. From the central extension $U(1)\to U(V) \to PU(V)$, we can define a central extension $LG^\tau=LG\times_{PU(V)}U(V)$ and a unitary representation $\widehat{\rho}:LG^\tau\to U(V)$ such that $\widehat{\rho}\circ i(e^{\sqrt{-1}\theta})=e^{\sqrt{-1}\theta}id_V$ as follows.
%$$\rho^*U(V)=U(V)\times_{PU(V)}LG\to LG$$

%$$\widehat{\rho}:\rho^*U(V)\to U(V).$$
%That is, the following diagram commutes.
$$\xymatrix{
U(1) \ar[r] & U(V) \ar[r] & PU(V) \\
U(1) \ar@{=}[u] \ar[r]^i & LG^\tau \ar[u]^{\widehat{\rho}} \ar[r]^p & LG \ar[u]^\rho
}
$$

In this situation, $(\widehat{\rho},V)$ is called a $\tau$-twisted representation.

If a central extension $U(1)\xrightarrow{i} LG^\tau\xrightarrow{p}LG$ is given, a representation 
$$\widehat{\rho}:LG^\tau\to U(V)$$
 is also called a $\tau$-twisted representation if $\widehat{\rho}\circ i(e^{\sqrt{-1}\theta})=e^{\sqrt{-1}\theta}id_V$. In this case, we can define a projective representation $\rho:LG\to PU(V)$.

Two definitions coincide.

Let $\rho:LG^\tau\to U(V)$ be a $\tau$-twisted representation. Then $LG$ acts on $\bb{P}(V\otimes l^2 \otimes L^2(G))$.

We have a $G$-equivariant $LG$ principal bundle $LG\to PG \to G$, where $G$ acts on $PG$ by left multiplication, and on itself by conjugation.
$PG$ is the path space of $G$ defined by
$$PG:=\{p:\bb{R}\to G|p(\theta+2\pi)p(\theta)^{-1}\in G \text{ is independent of }\theta\}.$$
Then we can define a $G$-equivariant projective bundle 
$$PG\times_{LG}\mathbb{P}(V\otimes l^2 \otimes L^2(G)).$$
\end{proof}

Lastly, let us study the push-forward maps. When $Y$ is a compact manifold and $X$ is a manifold, any smooth map $F:Y\to X$ can be written as a composition of closed embedding $i:Y\hookrightarrow X\times S^{2N}$ and the trivial $S^{2N}$ fibration $p:X\times S^{2N}\twoheadrightarrow X$ for sufficient large $N$. Let $\tau$ be a twisting over $X$.  Since we do not deal with non-trivial gradings, we assume that the normal bundle of $i(Y)$ in $X\times S^{2N}$ is orientable. Push-forward map 
$$i_!:K^{F^*\tau+W_3(F)+k+dimX-dimY}(Y)\to K^{p^*\tau+k}(X\times S^{2N})$$
 is defined by ``zero extension'' and 
$$p_!:K^{p^*\tau+k}(X\times S^{2N})\to K^{\tau+k}(X)$$
 is defined by ``family index theorem'', where $W_3(F)$ is the image of $w_2(X)-F^*(w_2(Y)) \in H^2(X,\mathbb{Z}_2)$ under the Bockstein homomorphism $\beta:H^2(X,\bb{Z}_2)\to H^3(X,\bb{Z}).$
\begin{dfn}%[push-forward]
\label{def of push-forward}
$$F_!:=p_!\circ i_!$$
\end{dfn}
$F_!$ is independent of choices of $i$ and $N$. One can find more information in \cite{CW}.

\begin{rmk}We can generalize for non-compact $Y$ because we deal with compactly supported $K$-theory.
\end{rmk}
Push-forward map has functoriality. Let $X$, $Y$ and $Z$ be manifolds, $\tau$ be a twisting over $Z$, and $X\xrightarrow{f}Y\xrightarrow{g}Z$ be smooth maps satisfying the above assumption.
\begin{lem}[\cite{CW}]\label{functoriality of push-forward}
$$g_!\circ f_!=(g\circ f)_!$$
\end{lem}
Let $p:Y\to X$ be a fiber bundle whose fiber is discrete and $P$ be a projective bundle over $X$. Then, we can describe $p_!:K^k_{p^*P}(Y)\to K^k_P(X)$ as follows.
\begin{lem}\label{explicit formula of push-forward with discrete fiber}
If $s$ is a section of $p^*Fred(P)$, the homotopy class $[s]$ determines an element of $K^{p^*\tau}(Y)$. Then, we have an explicit formula

$$p_!([s])=[\{\bigoplus_{y\in p^{-1}(x)}s_y\}_{x\in X}],$$
where we use an isomorphism $\oplus(P\otimes l^2)\cong P\otimes(\oplus l^2)\cong P\otimes l^2$ and Kuiper's theorem.
\end{lem}

\subsubsection{Computation of $K_G^{\tau+k}(G)$}
Freed, Hopkins and Teleman computed $K_G^{\tau+k}(G)$ by use of spectral sequences or a Mackey decomposition in \cite{FHT1} or \cite{FHT3}.
\begin{thm}[\cite{FHT1} \cite{FHT3}]
Let $\tau$ be a $G$-equivariant twisting over $G$ coming from a positive central extension of $LG$. Then the following isomorphism holds.
$$
K^{\tau+k}_G(G)=
\begin{cases}
0&(k=rank(G)+1 \mod 2)\\
\mathbb{Z}[(\Lambda^\tau/\kappa^\tau(W^e_{\mathit{aff}}(G)))_\mathit{reg}]&(k=rank(G) \mod 2)
\end{cases}
$$
\end{thm}

%3つの函手%この章の目的は言葉を用意することにしよう
\section{Three contravariant quasi functors and FHT isomorphisms}\label{3contravariant functors}
As a preliminary, we introduce some categories and quasi functors. Then we rewrite FHT isomorphisms using our terminologies and verify Theorem \ref{main corollary} under the assumption that Theorem \ref{main theorem for K for tori} and \ref{main theorem for R for tori} hold.

Firstly, we define ``quasi functors''.
\begin{dfn}%[Contravariant quasi functor]
Let $\mathcal{C}_1$, $\mathcal{C}_2$ be categories. $F:\mathcal{C}_1\to \mathcal{C}_2$ is a contravariant quasi functor if the following data are given.

(1) For any $O\in Obj(\mathcal{C}_1)$, $F(O)\in Obj(\mathcal{C}_2)$.

(2) For any $f\in Mor_{\mathcal{C}_1}(O_1,O_2)$, $F(f)\in Mor_{\mathcal{C}_2}(F(O_2),F(O_1))$ satisfying that $F(id)=id$.

That is, a quasi functor is a ``functor'' which allows that $F(f_1\circ f_2)\neq F(f_2)\circ F(f_1)$.
\end{dfn}

The classical category of Lie groups is too large to construct all quasi functors at the same time. For example, when $G=SO(3)$, $t.e.K(G,\tau)$ defined below does not sometimes have all information of $G$-equivaritant twisted $K$-theory of $G$ (\cite{FHT4}).

%議論する圏%
\subsection{Categories}\label{categories}

\begin{dfn}%[Category $\mathcal{C}_0^{twist}$ and $T-\mathcal{C}^{twist}$]
\label{2categories}

An object of $\mathcal{C}_0^{twist}$ and $T-\mathcal{C}^{twist}$ is a pair of a compact connected Lie group $G$ with torsion-free $\pi_1$ and a positive central extension $\tau$ of $LG$. 
$\tau$ determines a $G$-equivariant twisting over $G$, for which we use the same character (Lemma \ref{twisting induced by central extension}).

$f:(H,\tau')\to(G,\tau)$ is a morphism in $\mathcal{C}_0^{twist}$ if $G=H$, $\tau'=\tau$ and $f=id$. That is, $\mathcal{C}_0^{twist}$ is a discrete category.

$f:(H,\tau')\to(G,\tau)$ is a morphism in $T-\mathcal{C}^{twist}$ if it is a morphism in $\mathcal{C}_0^{twist}$ or $H$ and $G$ are tori, $\tau'=f^*\tau$, and $f$ has the injective tangent map.

%$\rho:(H,\sigma)\to(G,\tau)$ is a morphism in $\widetilde{\mathcal{C}}^{twist}$ if $\rho $is a Lie group homomorphism equipped with an isomorphism of twistings $\rho^*\tau\cong\sigma$ satisfying global condition and regurality condition defined below.
\end{dfn}

 %%linklinklink
We introduce the following definition to extend Theorem \ref{main theorem for K for tori} to $f:H\to G$ satisfying the decomposable condition described below.

Let $(G,\tau)$ be an object of $\ca{C}^{twist}_0$ and $H$ be a compact connected Lie group with torsion-free $\pi_1$. Take maximal tori $S$ and $T$ of $H$ and $G$ respectively. We can assume $f(S)\subseteq T$. That is, the following diagram commutes.
$$\begin{CD}
H @>f>> G \\
@AiAA @AkAA \\
S @>f|_S >> T
\end{CD}$$

\begin{dfn}\label{decomposable condition}
Let $f:H\to G$ be a smooth group homomorphism. The decomposable condition means the followings.

(1) $f^*\tau$ is also positive.

(2) $H/\ker(f)$ has torsion-free $\pi_1$.

(3) $[ df(\mathfrak{h}),\mathfrak{s}^\perp ] =0$, where $\fra{s}^\perp$ is the orthogonal complement of $df(\fra{s})$ in $\fra{g}$ under the bilinear form $\innpro{\cdot}{\cdot}{\tau}$ defined in Lemma \ref{property of kappa^tau}. We call this condition the ``local condition''.
%$\tau$ induces a symmetric bilinear form on $\mathfrak{t}$ and the restriction to $\mathfrak{s}$ of this bilinear form coincides with induced one from $f^*\tau$. Then we can take the orthogonal complement $\mathfrak{s}^\perp:=\mathfrak{t}\ominus df(\mathfrak{s})$. $[ df(\mathfrak{h}),\mathfrak{s}^\perp ] =0$. We call this condition ``local condition''.
\end{dfn}
\begin{rmk}
Condition (1) and (2) guarantee that $(H,f^*\tau)$ and $(H/\ker(f),f^*\tau)$ are objects of $\ca{C}^{twist}_0$.

Condition (3) guarantees that $H/\ker(f)\times S^\perp$ is also a group, where $S^\perp\subseteq T$ is the torus whose Lie algebra is $\fra{s}^\perp$. Well-definedness of of $S^\perp$ is verified in Corollary \ref{well-def. of orthogonal completion of tori}.

If $f$ satisfies this condition, $df$ is automatically injective. It follows from Lemma \ref{pull back formula of kappa}.
\end{rmk}
The following examples satisfy the decomposable condition.
\begin{ex}
(a) Any injection $f$ when $rank(H)=rank(G)$.

(b) Any local injection $f$ satisfying (2) when $rank(H)=rank(G)$.

(c) Any local injection $f$ when $H$ is a torus.

(d) $U(n)\hookrightarrow U(n+N)$ defined by 
$$A\mapsto \begin{pmatrix}A & 0 \\ 0 & 1\end{pmatrix}.$$
\end{ex}

\begin{lem}\label{well-def. of the map between Weyl group}
If $f$ satisfies the local condition.
The normalizer of $S$ in $H$ is mapped to the one of $T$ in $G$.
\end{lem}

\begin{proof}
Since $H$ is compact and connected, $\exp:\mathfrak{h}\to H$ is surjective, that is, for any $h\in H$, there exists $v\in\mathfrak{h}$ such that $h=\exp(v)$. With the same way, for $t\in S^\perp$, there exists $\widetilde{t}\in\mathfrak{s}^\perp$ such that $\exp(\widetilde{t})=t$. From Campbell-Hausdorff's formula 
$$f(h)tf(h)^{-1}=f(\exp(v))\exp(\widetilde{t})f(\exp(-v))$$
$$=\exp(df(v)+\widetilde{t}+\frac{1}{2}[df(v),\widetilde{t}]+\cdots)f(\exp(-\widetilde{t}))=\exp(df(v)+\widetilde{t})\exp(-df(v))$$
$$=\exp(df(v)+\widetilde{t}-df(v)+\frac{1}{2}[df(v)+\widetilde{t},-df(v)]+\cdots)$$
$$=\exp(\widetilde{t})=t.$$
Therefore, $f(h)$ and $t$ commute.
Moreover, for any $x\in T$, there exit $s\in S$ and $t\in T^\perp$ such that $x=f(s)\cdot t$. We prove it at Proposition \ref{decomposition of injection for tori}.

Therefore, for any $n\in N(S)$ and $x\in T$, the following holds.
$$f(n)xf(n^{-1})=f(n)(f(s)t)f(n^{-1})=f(n)f(s)f(n^{-1})t$$
$$=f(nsn^{-1})t\in T.$$
\end{proof}

From this lemma, we can define a homomorphism $f_*:W(H)\to W(G)$, where $W(H)$ and $W(G)$ are the Weyl groups of $H$ and $G$ respectively. 

\begin{lem}
$f_*$ is an injection.
\end{lem}
\begin{proof}
Suppose that $n\in N(S)$ and $[n]\in \ker(f_*)$. Then, $Ad(f(n))$ defines an automorphism of $\fra{t}$. From the assumption that $[n]\in \ker(f_*)$, $Ad(f(n))=id_{\fra{t}}$. 
Since $Ad(f(n))|_{\fra{s}}=Ad(n)$, $n\in S$. We obtain the conclusion.
\end{proof}

%Local condition is a sufficient condition to define $f^\#$, $f^!$ and 
%global condition is a sufficient condition to define 
%$char(f)$ defined below.

In this section, we define three quasi functors $t.e.K$, $char$ and $RL$ $:\mathcal{C}_0^{twist}\to\mathcal{A}b$. Later, we extend them as widely as possible.

Throughout this paper, $\ca{A}b$ is the category of abelian groups.

%ねじれ指標函手%
\subsection{The quasi functor $char$}\label{char}
Let us start from the easiest quasi functor $char$. FHT isomorphism was verified by use of this quasi functor essentially. 

Let us define a quasi functor
$$char:T-\ca{C}^{twist}\to\ca{A}b.$$
Firstly, we define it from $\ca{C}^{twist}_0$, later we extend it to $T-\ca{C}^{twist}$.

Let $(G,\tau)$ be an object of $\mathcal{C}_0^{twist}$. Let $T$ be a maximal torus of $G$.  Let us recall that 
$$1\to U(1) \xrightarrow{i}T^\tau \xrightarrow{p} T \to 1 $$
is the central extension of $T$ induced from $LT^\tau$. $\Lambda^\tau_T$ is the set of irreducible $\tau$-twisted representations of $T^\tau$.
That is,
$$\Lambda_T^\tau:=\{\lambda:T^\tau \to U(1)|\text{ homomorphism such that } \lambda\circ i=id\}$$
$$\cong \{d\lambda:\Pi_{T^\tau}\to\bb{Z}| \text{ homomorphism such that }d\lambda\circ di=id\}.$$

As we explain in Section 1, $W^{e}_{\mathit{aff}}(G)=\Pi_T\rtimes W(G)$ acts on $\Lambda^\tau_T$, we can define regular orbits. An orbit is regular if the stabilizer is trivial.
\begin{dfn}\label{def. of regular orbit}
The quasi functor 
$$char:\mathcal{C}_0^{twist}\to\mathcal{A}b$$
 is defined by the $\mathbb{Z}$ free module generated by the set of regular orbits $(\Lambda^\tau/W^e_{\mathit{aff}}(G))_\mathit{reg}$.

We write it as $\mathbb{Z}[(\Lambda^\tau/W^e_{\mathit{aff}}(G))_\mathit{reg}]$ also.

\end{dfn}

%ねじれK理論%
\subsection{The quasi functor $t.e.K$}\label{Twisted equivariant K-theory}

The following theorem is one of the main theorem of \cite{FHT1}. We have already stated it for tori.
\begin{thm}[\cite{FHT1}]\label{isomorphism between t.e.K and char}
Let $(G,\tau)$ be an object of $\ca{C}^{twist}_0$.
% compact connected Lie group with $\pi_1$ torsion-free and $\tau$ be a $G$-twisting.

$$
K^{\tau+k}_G(G)=
\begin{cases}
0&(k=rank(G)+1 \mod 2)\\
\mathbb{Z}[(\Lambda^\tau/\kappa^\tau(W^e_{\mathit{aff}}(G)))_\mathit{reg}]&(k=rank(G) \mod 2)
\end{cases}
$$
\end{thm}

This theorem implies the following immediately.
\begin{cor}\label{zeroness of classical pull back}
If $rank (G)-rank (H)$ is odd, the classical induced homomorphism $f^*\circ f^\natural:K^{\tau+k}_G(G)\to K^{f^*\tau+k}_H(H)$ is zero.
\end{cor}

Therefore, $f^*\circ f^\natural:K^{\tau+k}_G(G)\to K^{f^*\tau+k}_H(H)$ must not have a compatibility with the quasi functor $char$ generally.

Motivated by this observation, we define the quasi functor
$$t.e.K:T-\ca{C}^{twist}\to \ca{A}b$$
so that Theorem \ref{main theorem for K for tori} holds.

%So we will construct the quasi functor $t.e.K$, modification of twisted equivariant $K$-theory. At first we define it at $\mathcal{C}^{twist}_0$ and extend to $\mathcal{C}^{twist}$, that is we construct the ``induced homomorphism'' $f^\#:=t.e.K(f)$ so that it has a compatibility with $char(f)$ at Section \ref{Modified twisted equivariant K-functor}. This is one of the main construction in this paper.

\begin{dfn}%[The modified twisted equivariant $K$-theory functor $t.e.K$]
\label{def. of t.e.K}
A quasi functor
$$t.e.K:\mathcal{C}^{twist}_0\to\mathcal{A}b$$
is defined by $t.e.K(G,\tau):=K^{\tau+rank(G)}_G(G)$.
\end{dfn}
Theorem \ref{isomorphism between t.e.K and char} implies that there is an isomorphism $$M.d._G:t.e.K(G,\tau)\to char(G,\tau).$$
We define it in Section \ref{Modified twisted equivariant K-functor} by use of a Mackey decomposition and a push-forward map.

%ループ群の表現論函手%
\subsection{The quasi functor $RL$}\label{Positive energy representation group functor}
We have already stated the following isomorphism (Theorem \ref{classifying theorem of PER in Section 1}).
\begin{thm}\label{Computation of representation group, again}
$$R^{\tau-\sigma}(LG)\cong\mathbb{Z}[(\Lambda^{\tau}/W^e_{\mathit{aff}}(G))_{\mathit{reg}}]$$
The isomorphism is given by the composition of taking the lowest weight and $\rho-\mathit{shift}$.
\end{thm}

Let $f:H\to G$ be a smooth homomorphism, then we can define $Lf:LH\to LG$ %as follows. Let $l\in LH$, that is $l:S^1\to H$. We define 
by $Lf(l):=f\circ l$ for $l\in LH$.

At first sight, the homomorphism $Lf$ seems to induce the pull back of representations $(Lf)^*:R^\tau(LG)\to R^{f^*\tau}(LH)$ just like the cases of compact Lie groups.

Recall that all irreducible representations of compact Lie groups are finite dimensional and all representations of compact Lie groups are completely reducible, therefore, we can identify finitely reducible representations as finite dimensional representations. It is independent of the groups that a representation is finite dimensional.

However, in the case of loop groups, ``finitely reducibility'' is not a property of spaces, but one of group actions unlike the cases of compact Lie groups. For example, if $i_1:T\to T\times T'$ is the natural inclusion into the first factor, and $(V,\rho)$ is an irreducible positive energy representation of $L(T\times T')$, $(V,(Li_1)^*\rho)$ is not finitely reducible (Corollary \ref{infinitely reducibility in Section 5}).

Motivated by this observation, we define a quasi functor
$$RL:T-\ca{C}^{twist}\to \ca{A}b.$$
Firstly, we define it at $\ca{C}^{twist}_0$, later we extend to $T-\ca{C}^{twist}$.
%So we define the quasi functor $RL$ associated with the representation group of $LG$ so that it has a compatibility with $char$. At first we define it at $\mathcal{C}^{twist}_0$ and extend to $T-\mathcal{C}^{twist}$.

\begin{dfn}\label{def. of RL on C_0}
The contravariant quasi functor $RL:\mathcal{C}^{twist}_0\to\mathcal{A}b$ is defined by
$$RL(G,\tau):=R^{\tau-\sigma}(LG).$$
\end{dfn}
Let us recall that the isomorphism
$$l.w._G:RL(G,\tau)\to char(G,\tau)$$
has been defined by taking the lowest weight. We describe it more explicitly in Section 5 for tori.

We define the ``induced homomorphism'' $f^!=RL(f)$ for $f$ which is a morphism in $T-\mathcal{C}^{twist}$ so that it has a compatibility with $char(f)$.

%%FHTisomorphisms

\subsection{Rewriting of FHT isomorphism}\label{FHT isom.}

Freed, Hopkins and Teleman constructed isomorphism between twisted equivariant $K$-theory and positive energy representation group by use of a family of ``Dirac operators'' parametrized by the set of connections over the trivial bundle $S^1\times G$ in \cite{FHT2}. We write FHT isomorphism for $G$ as $FHT_G$.

We describe the essence of their proof for tori in our terminology.
\begin{thm}[\cite{FHT2} Proposition 4.8]\label{explicit description of FHT}
Let $T$ be a torus and $\tau$ be a positive central extension of $LT$. The following  commutative diagram holds.
$$\xymatrix{
RL(T,\tau) \ar[rr]^{FHT_T} \ar[dr]^{l.w._T} & & t.e.K(T,\tau) \ar[dl]^{M.d._T} \\
& char(T,\tau) & 
}$$
\end{thm}
%\begin{thm}[for simply connected]
%Let $G$ be a simly connected Lie group, $i:T\to G$ be a chosen maximal torus and $\tau$ be a non-degenerate twisting.
%$\sigma$ is the spin extension of $LG$. Following diagram is commutative.
%$$\xymatrix{
%RL(G,\tau-\sigma) \ar[r]^{FHT_G} \ar[d]^{l.w.} & t.e.K(G,\tau) \ar[r]^{i^\#} & t.e.K(T,i^*\tau) \ar[d]^{M.d._T} \\
%\mathbb{Z}[\Lambda^{\tau-\sigma}/W^e_\mathit{aff}(G)] \ar[r]^{\rho -\mathit{shift}} & char(G,\tau) \ar[r]^{char(i)} & char(T,i^*\tau)
%}$$
%\end{thm}
%They verified that $FHT_G$ is isomorphic using the above theorem (it implies that $FHT_G$ is injective) and computation of image of $RL(G,\tau-\sigma)$ under $M.d._T\circ i^\# \circ FHT_G$ and $char(i) \circ (\rho-\mathit{shift}) \circ l.w._G$.

%For general compact Lie group $G$ with $\pi_1$ torsion-free, they verified this theorem by taking finite cover $1\to A\to \widetilde{G}\to G\to 1$ such that $\widetilde{G}=G'\times Z$ is the product of a torus $Z$ and a simply connected Lie group $G'$.
\subsection{Proof of Theorem \ref{main corollary}}\label{proof of main corollary}
Let us verify our main theorem under the assumption that Theorem \ref{main theorem for K for tori} and \ref{main theorem for R for tori} hold.

The following commutative diagram holds.
$$\xymatrix{
R^\tau(LT) \ar[rrr]^{f^!} \ar[rd]^{l.w._{T}} \ar[dd]_{FHT_{T}} \ar@{}[rrrd]|{(1)} \ar@{}[drd]|{(2)} & & &
R^{f^*\tau}(LT') \ar[dl]_{l.w._{T'}} \ar[dd]^{FHT_{T'}} \ar@{}[ddl]|{(3)} \\
 & char(T,\tau) \ar[r]^{char(f)} \ar@{}[rd]|{(4)} & char(T',f^*\tau) & \\
K^{\tau+dim(T)}_{T}(T) \ar[ru]_{M.d._{T}} \ar[rrr]_{f^\#} & & & K^{f^*\tau+dim(T')}_{T'}(T') \ar[lu]^{M.d._{T'}}
}$$

Commutativity of (1) is Theorem \ref{main theorem for R for tori}, and commutativity of (4) is Theorem \ref{main theorem for K for tori}. Commutativity of (2) and (3) are the above theorem. Since $M.d._T$ and $M.d._{T'}$ are isomorphisms, we obtain the conclusion.

%%%ねじれ指標函手

\section{The quasi functor $char$}\label{Section of the quasi functor char}
In this section, we study the easiest quasi functor $char$. We use Lemma \ref{pull back formula of kappa} and \ref{product formula for kappa}.

%トーラス
\subsection{For tori}\label{the quasi functor for tori}
In this section, we study in the case of tori. This is essential because we extend the quasi functor $char$ to general Lie groups by use of reduction to maximal tori.

%Definition of char(f)
\subsubsection{Definition of induced homomorphism $char(f)$}\label{subsubsection of def. of induced homomorphism of char(f)}
Let $f:T'\to T$ be a local injection and $\tau$ be a positive central extension of $LT$. We define $char(f)([\lambda]_{T})$ for $\lambda\in\Lambda_{T}^\tau$, where $[\lambda]_{T}$ means the $\Pi_{T}$-orbit including $\lambda$. Let us consider the set 
$\{{}^tdf(\lambda+\kappa^\tau(n))|n\in\Pi_{T}\}$ and decompose it as $\Pi_{T'}$-orbit. That is,
$$\{{}^tdf(\lambda+\kappa^\tau(n))|n\in\Pi_{T}\}=\coprod_{i=1}^N\{\mu_i+\kappa^{f^*\tau}(n)|n\in\Pi_T\}.$$
$N$ is a finite number because the set $\Lambda^{f^*\tau}_{T'}/\kappa^{f^*\tau}(\Pi_{T'})$ is a finite set.
\begin{dfn}\label{def. of char(f)}
$$char(f)([\lambda]_{T}):=\sum_{i=1}^N[\mu_i]_{T'}$$
\end{dfn}
\begin{lem}
This is well-defined.
\end{lem}
\begin{proof}
It is sufficient to verify that $\{{}^tdf(\lambda+\kappa^\tau(n))|n\in \Pi_{T}\}$ is $\Pi_{T'}$-invariant. It follows from that
$${}^tdf(\lambda+\kappa^\tau(n))+\kappa^{f^*\tau}(l)={}^tdf(\lambda+\kappa^\tau(n)+\kappa^\tau(df(l)))\in\{{}^tdf(\lambda+\kappa^\tau(n))|n\in \Pi_{T}\},$$
where the equality follows from Lemma \ref{pull back formula of kappa}.
\end{proof}
We can compute the above homomorphism for some cases. We verify the followings in Section 5.

\begin{lem}%[direct product]
\label{char formula for direct product}
Let $i_1:T_1\to T_1\times T_2$ be the natural inclusion into the first factor and 
$\tau=p_1^*\tau_1+p_2^*\tau_2$, where $\tau_j$ is a positive central extension of $LT_j$ and $p_j:T_1\times T_2\to T_j$ is the natural projection onto the $j$'th factor. Then,
$$char(i_1)([\lambda]_{T_1\times T_2})=[{}^tdi_1(\lambda)]_{T_1}.$$
\end{lem}
\begin{lem}%[finite cover]
\label{char formula for finite cover}
Let $q:T'\to T$ be a finite covering. Then, 
$$char(q)([\lambda]_{T})=\sum_{m\in\Pi_{T}/df(\Pi_{T'})}[{}^tdf(\lambda+\kappa^\tau(m))]_{T'}.$$

\end{lem}

\begin{rmk}
$m\in\Pi_T/df(\Pi_{T'})$ is a chosen representative element. Let us notice that the orbit $[{}^tdf(\lambda+\kappa^\tau(m))]_{T'}$ is independent of the choice of the representative element.
\end{rmk}

%%関手でないための反例
\subsubsection{A counterexample to that $char$ is a functor}
The quasi functor $char$ is not a functor. We give a counterexample to that $char$ is a functor.

Let us consider the following commutative diagram
$$\xymatrix{
\mathbb{T} \ar[rr]^h \ar[rd]^f & & \mathbb{T}^2_2 \\
& \mathbb{T}^2_1\ar[ru]^g &
}$$
where $f(z):=(z,z^{-1})$, $g(z_1,z_2):=(z_1z_2,z_1z_2^{-1})$ and $h(z):=(1,z^2)$. Let us notice that $1$ is the unit element of $\bb{T}$.

We regard $\Pi_\bb{T}=\bb{Z}=\Lambda^{h^*\tau}_\bb{T}$, $\Pi_{\bb{T}^2_1}=\bb{Z}^2=\Lambda^{g^*\tau}_{\bb{T}^2_1}$ and $\Pi_{\bb{T}^2_2}=\bb{Z}^2=\Lambda^\tau_{\bb{T}^2_2}$ by fixing the base of $\Pi$ and dual base.

Let us consider a central extension $\tau$ of $L\bb{T}^2_2$ so that 
$$\kappa^\tau=\begin{pmatrix} -1 & 0 \\ 0 & -1 \end{pmatrix}.$$
Therefore, $char(\mathbb{T}^2_2,\tau)=\mathbb{Z}\begin{bmatrix} 0 \\ 0 \end{bmatrix}_{\bb{T}^2_2}$. $\begin{bmatrix} 0 \\ 0 \end{bmatrix}_{\bb{T}^2_2}$ means the orbit generated by $\begin{pmatrix} 0 \\ 0 \end{pmatrix}$, that is, the whole $\Lambda^\tau_{\bb{T}^2_2}$.

We have the following data about the central extensions from $df=\begin{pmatrix} 1 \\ -1 \end{pmatrix}$, $dg=\begin{pmatrix} 1 & 1 \\ 1 & -1 \end{pmatrix}$, $dh=\begin{pmatrix} 0 \\ 2 \end{pmatrix}$ and Lemma \ref{pull back formula of kappa}.
$$\kappa^{g^*\tau}=
\begin{pmatrix} -2 & 0 \\ 0 & -2 \end{pmatrix} ,
\kappa^{h^*\tau}=
\begin{pmatrix} -4 \end{pmatrix}.$$
Therefore, 
$$char(\mathbb{T}^2_1,g^*\tau)=\mathbb{Z}\begin{bmatrix} 0 \\ 0 \end{bmatrix}_{\bb{T}^2_1}
\oplus \mathbb{Z}\begin{bmatrix} 0 \\ 1 \end{bmatrix}_{\bb{T}^2_1}\oplus \mathbb{Z}\begin{bmatrix} 1 \\ 0 \end{bmatrix}_{\bb{T}^2_1} \oplus \mathbb{Z}\begin{bmatrix} 1 \\ 1 \end{bmatrix}_{\bb{T}^2_1}$$
$$char(\mathbb{T},h^*\tau)=\mathbb{Z}[0]_{\bb{T}}\oplus\mathbb{Z}[1]_{\bb{T}}\oplus\mathbb{Z}[2]_{\bb{T}}\oplus\mathbb{Z}[3]_{\bb{T}}$$

Let us compute $char(f)$, $char(g)$ and $char(h)$. Since

$${}^tdh\begin{pmatrix}n \\ m \end{pmatrix} = 2m,$$
the image of $\begin{bmatrix}0 \\ 0 \end{bmatrix}_{\bb{T}^2_2}$ under ${}^tdh$ is the set of the whole even numbers. Therefore,
 $$char(h)\begin{bmatrix}0 \\ 0 \end{bmatrix}_{\bb{T}^2_2} = [0]_{\bb{T}}+[2]_{\bb{T}}.$$

With the same way, we have the followings
$$char(g) \begin{bmatrix}0 \\ 0 \end{bmatrix}_{\bb{T}^2_2}=\begin{bmatrix}0 \\ 0 \end{bmatrix}_{\bb{T}^2_1}+\begin{bmatrix}1 \\ 1 \end{bmatrix}_{\bb{T}^2_1}$$
$$char(f)\begin{bmatrix}0 \\ 0 \end{bmatrix}_{\bb{T}^2_1}=[0]_{\bb{T}}+[2]_{\bb{T}}$$
$$char(f)\begin{bmatrix}1 \\ 1 \end{bmatrix}_{\bb{T}^2_1}=[0]_{\bb{T}}+[2]_{\bb{T}}.$$

Therefore, $$char(f)\circ char(g)\begin{bmatrix} 0 \\ 0 \end{bmatrix}_{\bb{T}^2_2} = 2 char(h) \begin{bmatrix} 0 \\ 0 \end{bmatrix}_{\bb{T}^2_2}.$$

% the decomposition of a local injection
\subsubsection{The decomposition of a local injection associated with a positive central extension}
While $char$ is not a functor, it holds ``functoriality'' we need.
To state it, we need to rewrite a local injection as a composition of finite coverings and an injection into the first factor of a direct product canonically.
%\begin{thm}
%$$
%\begin{CD}
%t.e.K(G,\tau) @>>> t.e.K(H,f^*\tau) \\
%@VVV @VVV \\
%char (G,\tau) @>>> char (H,f^*\tau)
%\end{CD}
%$$commutes.
%\end{thm}

Let $T$ and $T'$ be tori, $f:T'\to T$ be a local injection and $\tau$ be a positive central extension of $LT$. At first sight, our formulae Lemma \ref{char formula for direct product}, \ref{char formula for finite cover} are very special. However, we can rewrite any local injection $f$ as a composition of finite coverings and an inclusion into the first factor of a direct product canonically. The following proposition is clear. Firstly, we state the following proposition, which is clear.

\begin{pro}%[decomposition of a local injection]
\label{decomposition of local injection for tori}
We can canonically decompose a local injection $f$ as follows
$$T'\xrightarrow{q} T'/\ker f \xrightarrow{i} T$$
where $q$ is a finite covering, and $i$ is an injection.
\end{pro}

The above proposition tells us that we can assume that $f$ is an injection.

Since we assume $f^*\tau$ is positive, $df(\mathfrak{t}')\cap df(\mathfrak{t}')^\perp=0$, where $df(\mathfrak{t}')^\perp$ is the orthogonal completion under $\innpro{\cdot}{\cdot}{\tau}$ in $\fra{t}$.% That is $df(\mathfrak{t}')\oplus df(\mathfrak{t}')^\perp=\mathfrak{t}$. 

\begin{lem}
$df(\mathfrak{t}')^\perp\cap\Pi_T$ is a $\mathbb{Z}$ free module of rank $n-n'$, where $n=\dim(T)$ and $n'=\dim(T')$.
\end{lem}

\begin{proof}
Since $\Pi_T$ is a free module, $df(\mathfrak{t}')\cap\Pi_T\subseteq\Pi_T$ is also free.
%We use linear algebra over $\mathbb{Q}$.
Let us fix a $\bb{Z}$ besis $\{v_1,v_2,\cdots,v_{n'}\}$ of $\Pi_{T'}$. Then, 
$$(df(\Pi_{T'})\otimes\mathbb{Q})^\perp=(df(\fra{t}')^\perp\cap\Pi_T)\otimes\bb{Q}=
\bigcap_{i\in\{1,2,\cdots,n'\}}\ker(\kappa^\tau(v_i)\otimes\mathbb{Q}).$$
Since $\kappa^\tau$ is injective, $\dim((df(\Pi_{T'})\otimes\bb{Q})^\perp)=n-n'$.
%Dimension of this vector space is $n-n'$ 
 So we can take a $\bb{Q}$ basis $\{w_j\}_{j=1}^{n-n'}$ such that $w_j\in\Pi_T$. %because for any $v\in \Pi_T\otimes\mathbb{Q}$ there exists $N\in\mathbb{Z}$ such that $Nv\in\Pi$.
Since $\{w_j\}$ are linearly independent over $\bb{Z}$, $rank(df(\mathfrak{t}')^\perp\cap\Pi_T)\geq n-n'$. If $r:=rank(df(\mathfrak{t}')^\perp\cap\Pi)>n-n'$, we can take a $\bb{Z}$ basis $\{w_j'\}_{j=1}^r$ of $df(\mathfrak{t}')^\perp\cap\Pi_T$, linearly independent in $(df(\mathfrak{t}')^\perp\cap\Pi_T)\otimes\mathbb{Q}$ over $\bb{Q}$. The dimension of the right hand side is $n-n'$, which is a contradiction.
\end{proof}

This lemma implies the followings.

\begin{cor}\label{well-def. of orthogonal completion of tori}
$$T'^\perp:= df(\mathfrak{t}')^\perp/(df(\mathfrak{t}')^\perp\cap\Pi_T)\subseteq T$$
is a torus. We call it the ``orthogonal completion torus''. Let $j:T'^\perp\hookrightarrow T$ be the natural inclusion.
\end{cor}

\begin{pro}\label{decomposition of injection for tori}
If $f:T'\to T$ is an injection, we can define the new torus $T'\times T'^\perp$ and decompose $f$ as follows.
$$T'\xrightarrow{i_1}T'\times T'^\perp \xrightarrow{f\cdot j} T.$$
Moreover, $f\cdot j$ is a finite covering, where $(f\cdot j)(t_1,t_2)$ is defined by $f(t_1)\cdot j(t_2)$.
\end{pro}

Combining Proposition \ref{decomposition of local injection for tori} with Proposition \ref{decomposition of injection for tori}, we obtain the following theorem.

\begin{thm}%[decomposition of a local injection]
\label{decomposition of local injection}
We can decompose $f$ as follows
$$T'\xrightarrow{q} T'/\ker(f) \xrightarrow{i_1} T'/\ker(f)\times (T'/\ker(f))^\perp \xrightarrow{f\cdot j} T,$$
where $(T'/\ker(f))^\perp$ is the orthogonal completion torus associated with $f$ and $\tau$, $q$ is the natural finite covering, $i_1$ is the natural inclusion into the first factor and $j$ is the natural inclusion.
\end{thm}

Formula $\kappa^{(f\cdot j)^*\tau}={}^td(f\cdot j)\circ\kappa^\tau\circ d(f\cdot j)$
and that $df(\mathfrak{t}')\perp df(\mathfrak{t}')^\perp$ under the bilinear form $\innpro{\cdot}{\cdot}{\tau}$ imply the following lemma.

\begin{lem}\label{decompositin formula for kappa}
$$(f\cdot j)^*\tau=p_1^*i_1^*(f\cdot j)^*\tau+p_2^*i_2^*(f\cdot j)^*\tau=p_1^*\tau_1+p_2^*\tau_2.$$
\end{lem}

That is, the pair $((T'/\ker(f))\times (T'/\ker(f))^\perp,(f\cdot j)^*\tau)$ satisfies the assumption of Lemma \ref{char formula for direct product}.

\subsubsection{The functoriality we need}\label{The compatibility required for our theorems}
We use the same notations in the previous section.

The following is the functoriality we need mentioned above.

\begin{thm}\label{partial functoriality}
$$char(f)=char(q)\circ char(i_1) \circ char(f\cdot j)$$
\end{thm}

Firstly, we verify that we can assume that $f$ is injective.

\begin{lem}\label{reduction to injection}
If we have the following commutative diagram
$$\xymatrix{
T' \ar[rr]^f \ar[rd]^q & & T \\
& S \ar[ur]^g & 
}$$
where $f$ and $g$ are local injections and $q$ is a finite covering, the equality
$$char(f)=char(q)\circ char(g)$$ holds.
\end{lem}

\begin{proof}
Let us take an orbit $[\lambda]_{T}\in char(T,\tau)$ and compute $char(f)([\lambda]_{T})$ and $char(q)\circ char(g)([\lambda]_{T})$.

Since ${}^tdf={}^tdq\circ {}^tdg$, the image of the orbit $[\lambda]_T$ under ${}^tdf$ coincides with the one under ${}^tdq\circ{}^tdg$. Therefore, when $char(f)([\lambda]_{T})=\sum_{i=1}^N[\mu_i]_{T'}$,
$$char(q)\circ char(g)([\lambda]_{T})=\sum_{i=1}^Na_i[\mu_i]_{T'}$$
for some $a_i\in\mathbb{Z}_{>0}.$
It is sufficient to verify that $a_i=1$ for any $i$.

If $char(g)([\lambda]_{T})=\sum_{j=1}^M[\nu_j]_S$, 
$$char(q)\circ char(g)([\lambda]_{T})=\sum_{j=1}^Mchar(q)([\nu_j]_S)=\sum_{i=1}^Na_i[\mu_i]_{T'},$$ where 
$a_i=\#\{j\in\{1,2,\cdots,M\}|{}^tdq(\nu_j+\kappa^{g^*\tau}(\Pi_S))\supseteq \mu_i+\kappa^{f^*\tau}(\Pi_{T'})\}.$

However, if $[\nu]_S\neq[\nu']_S$, the images of orbits $\{\nu+\kappa^{g^*\tau}(n)|n\in\Pi_{S}\}$ and $\{\nu'+\kappa^{g^*\tau}(n)|n\in\Pi_{S}\}$ under ${}^tdq$ are disjoint because ${}^tdq$ is injective. That is, for any $j_1\neq j_2$, $char(q)([\nu_{j_1}]_S)\neq char(q)([\nu_{j_2}]_S)$.
\end{proof}
Let us verify Theorem \ref{partial functoriality}. Let $f:T'\to T$ be an injection and $\tau$ be a positive central extension of $LT$.\\\\
{\it Proof of Theorem \ref{partial functoriality}.}
The following commutative diagram holds.
$$\xymatrix{
T' \ar[rr]^f \ar[rd]^{i_1}& & T \\
& T'\times T'^\perp \ar[ur]^{f\cdot j} &
}$$
Let $[\lambda]_{T}\in char(T,\tau)$ be an orbit, $char(f)([\lambda]_{T})=\sum_{i=1}^N[\mu_i]_{T'}$, and $char(i_1)\circ char(f\cdot j)([\lambda]_{T})=\sum_{i=1}^Na_i[\mu_i]_{T'}$. We verify that $a_i=1$ for any $i$. Since
%We write $T'':=T\times T^\perp$ for simplicity. 
$${}^td(f\cdot j)(\{\lambda+\kappa^\tau(n)|n\in\Pi_T\})
=\{{}^td(f\cdot j)(\lambda)+{}^td(f\cdot j)(\kappa^\tau(n))|n\in\Pi_T\}$$
$$=\coprod_{m\in\Pi_T/d(f\cdot j)(\Pi_{T'\times T'^\perp})}\{{}^td(f\cdot j)(\lambda)+{}^td(f\cdot j)(\kappa^\tau(m))+\kappa^{(f\cdot j)^*\tau}(n')|n'\in\Pi_{T'\times T'^\perp}\}$$
$$=\coprod_{m\in\Pi_T/d(f\cdot j)(\Pi_{T'\times T'^\perp})}[{}^td(f\cdot j)(\lambda+\kappa^\tau(m))]_{T'\times T'^\perp},$$
that $a_i=1$ for any $i$ is equivalent to that 
$$char(i_1)([{}^td(f\cdot j)(\lambda+\kappa^\tau(m))]_{T'\times T'^\perp})\neq char(i_1)([{}^td(f\cdot j)(\lambda+\kappa^\tau(m'))]_{T'\times T'^\perp})$$
 if $m\neq m'$. From liniearity of these maps, it is sufficient to verify it under the assumption that $m'\in d(f\cdot j)(\Pi_{T'\times T'^\perp}).$

%That is, $m\in \Pi_{T\times T^\perp}$ if and only if 
%$[{}^tdi_1({}^tdr(\lambda+\kappa^\tau(m)))]=[{}^tdi_1({}^tdr(\lambda))]$.

Let us assume that $[{}^tdi_1({}^td(f\cdot j)(\lambda+\kappa^\tau(m)))]_{T'}=[{}^tdi_1({}^td(f\cdot j)(\lambda))]_{T'}$. We verify that $m\in d(f\cdot j)(\Pi_{T'\times T'^\perp})$.

The assumption is equivalent to that
$${}^tdi_1\circ{}^td(f\cdot j)\circ\kappa^\tau(m)={}^tdf(\kappa^{\tau}(m))\in\kappa^{f^*\tau}(\Pi_{T'}).$$
Therefore, there exists $k\in\Pi_{T'}$ such that 
$${}^tdi_1\circ{}^td(f\cdot j)\circ\kappa^\tau(m)={}^tdi_1\circ{}^td(f\cdot j)\circ\kappa^\tau\circ df(k).$$
So, ${}^td(f\cdot j)(\kappa^\tau(m)-\kappa^\tau(df(k)))\in\ker({}^tdi_1)$.

Since $df(\fra{t}')\perp df(\fra{t}')^\perp$ under $\innpro{\cdot}{\cdot}{\tau}$, 
$$\ker({}^tdi_1)%=\kappa^{(f\cdot j)^*\tau}(di_1(\Pi_{T'}\otimes\bb{R})^\perp)
=\kappa^{(f\cdot j)^*\tau}(di_1(\fra{t}')^\perp).$$ 
Therefore, there exists $v\in di_1(\fra{t}')^\perp=\fra{t}'^\perp$ such that 
$${}^td(f\cdot j)\circ(\kappa^\tau(m)-\kappa^\tau(df(k)))
=\kappa^{(f\cdot j)^*\tau}(v)
={}^td(f\cdot j)\circ\kappa^\tau\circ d(f\cdot j)(v).$$
Since ${}^td(f\cdot j)\circ\kappa^\tau$ is injective, $m-df(k)=d(f\cdot j)(v)$. Moreover, from that $m$ and $df(k)$ are elements of $\Pi_{T}$, $m-df(k)=d(f\cdot j)(v)\in\Pi_{T'}\cap df(\fra{t}')^\perp$.

By the definition of the orthogonal completion torus, $j=(f\cdot j)|_{T^\perp}$ is injective and
$$dj|_{\Pi_{T'^{\perp}}}=d(f\cdot j)|_{\Pi_{T'^\perp}}:\Pi_{T'^\perp}\to \Pi_{T}\cap df(\fra{t}')^\perp$$
 is bijective. Therefore, $v\in \Pi_{T'^\perp}\subseteq\Pi_{T'\times T'^\perp}$.
Therefore, $$m=d(f\cdot j)(di_1(k)+v)\in d(f\cdot j)(\Pi_{T'\times T'^\perp}).$$

\begin{flushright}$\Box$
\end{flushright}

%一般のLie群
\subsection{For general Lie group}\label{quasi functor for general Lie group}
We extend the above construction to more general cases. We use reduction to maximal tori.

%極大トーラスへのreduction
\subsubsection{Reduction to a maximal torus}\label{subsubsection of reduction to a max. torus of the quasi functor char}
Let $G$ be a compact connected Lie group with torsion-free $\pi_1$ and $T$ be a maximal torus of $G$. Let $i:T\to G$ be the natural inclusion.

Let us suppose that $\lambda\in\Lambda_T^\tau$ determines a regular orbit in $\Lambda_T^\tau$, that is, the stabilizer of $\lambda$ in $W^e_{\mathit{aff}}(G)$ is trivial. Then the $W^e_{\mathit{aff}}(G)$-orbit $W^e_{\mathit{aff}}(G).\lambda$ can be regarded as the union of $\Pi_T$-orbits $\coprod_{w\in W(G)}(w.\lambda+\kappa^\tau(\Pi_T))$. So we define $char(i)([\lambda]_G)$ as follows.
\begin{dfn}\label{def. of char(i)}
$$char(i)([\lambda]_G):=\sum_{w\in W(G)}[w.\lambda]_T$$
\end{dfn}
The following lemma is clear from the definition.
\begin{lem}
$char(i)$ is an injection.
\end{lem}

%%一般の場合のchar(f)の定義
\subsubsection{The definition of the induced homomorphism of $char$}\label{subsubsection of def. of char(f)}
Let us consider the following commutative diagram
$$\begin{CD}
H @>f>> G \\
@AiAA @AkAA \\
S @>f|_S>> T
\end{CD}$$
where $H$ and $G$ are compact connected Lie groups with torsion-free $\pi_1$, $S$ and $T$ are maximal tori of $H$ and $G$ respectively satisfying that $f(S)\subseteq T$. Let $\tau$ be a positive central extension of $LG$. Let us suppose that $f$ satisfies the decomposable condition.

%$G$-equivariant twisting over $G$. We assume that $f$ is a morphism in $\mathcal{C}^{\mathit{twist}}$.
\begin{dfn}\label{def. of char(f)}
$char(f):char(G,\tau) \to char(H,f^*\tau)$ is defined so that the following commutative diagram holds.
$$\begin{CD}
char(G,\tau) @>char(f)>> char(H,f^*\tau) \\
@Vchar(k)VV @Vchar(i)VV \\
char(T,k^*\tau) @>char(f|_S)>> char(S,i^*f^*\tau).
\end{CD}$$
\end{dfn}
We have to verify well-definedness of this definition.

Firstly, we verify the regularity of orbits. For any $[\lambda]_G\in char(G,\tau)$, $char(f|_S)\circ char(k)([\lambda]_G)$ determines a subset in $\Lambda^{i^*f^*\tau}_S$. We verify that the subset is a union of $W^e_{\mathit{aff}}(H)$-regular orbits. Since this subset coincides with the image of $W^e_{\mathit{aff}}(G).\lambda$ under ${}^tdf$, it is sufficient to verify the following lemma.
\begin{lem}\label{regularity of the image of a regular orbit}
If $f$ satisfies the local condition, the image of a regular orbit under ${}^tdf$ is $W^e_{\mathit{aff}}(H)$-invariant and has the trivial stabilizer.
\end{lem}

\begin{proof}
Let $\lambda\in\Lambda_T^\tau$ determine a $W^e_{\mathit{aff}}(G)$-regular orbit. Let us consider the $W^e_{\mathit{aff}}(G)$-orbit $W^e_{\mathit{aff}}(G).\lambda$ and the image of it under ${}^tdf$. Firstly, we verify that ${}^tdf(W^e_{\mathit{aff}}(G).\lambda)$ is $W^e_{\mathit{aff}}(H)$-invariant, then we verify that it has the trivial stabilizer.

For any $\mu\in {}^tdf(W^e_{\mathit{aff}}(G).\lambda)$, there exists $\nu\in W^e_{\mathit{aff}}(G).\lambda$ such that $\mu={}^tdf(\nu)$.
%If $\mu\in{}^tdf(W^e_{\mathit{aff}}(G)\cdot\lambda)$ we can take the natural preimage $\widetilde{\mu}\in W^e_{\mathit{aff}}(G)\cdot\lambda$ using bilinear form induced $\kappa^\tau$, that is $\widetilde{\mu}\in\kappa^\tau(df(\mathfrak{s}))$. 
$W^e_{\mathit{aff}}(H)$-invariance of ${}^tdf(W^e_{\mathit{aff}}(G).\lambda)$ follows from that
$$w.\mu=w.{}^tdf(\nu)={}^tdf(f_*(w).\nu)\in{}^tdf(W^e_{\mathit{aff}}(G).\lambda)$$
$$\mu+\kappa^{f^*\tau}(n)={}^tdf(\nu+\kappa^\tau(df(n)))\in{}^tdf(W^e_{\mathit{aff}}(G).\lambda)$$
where $w\in W(H)$ and $n\in \Pi_S$. These imply the invariance of ${}^tdf(W^e_{\mathit{aff}}(G).\lambda)$.

Now, let us verify the triviality of the stabilizer of ${}^tdf(W^e_{\mathit{aff}}(G).\lambda)$. Suppose that $(w,n)\in W_{\mathit{aff}}^e(H)$ stabilizes $\mu\in{}^tdf(W^e_{\mathit{aff}}(G).\lambda)$. Then
$$w.\mu+\kappa^{f^*\tau}(n)=w.{}^tdf(\nu)+{}^tdf\circ\kappa^\tau(df(n))$$
$$={}^tdf(f_*(w).\nu+\kappa^\tau(df(n)))=\mu={}^tdf(\nu).$$
That is,
$$f_*(w).\nu+\kappa^\tau(df(n))-\nu\in \ker({}^tdf)=\kappa^\tau(df(\fra{s})^\perp).$$
On the other hand, that
$$f_*(w).\nu+\kappa^\tau(df(n))-\nu\in\kappa^\tau(df(\fra{s}))$$
holds. It can be verified as follows.

Let us notice that $\nu$ can be orthogonally decomposed as $\nu=\nu_1+\nu_2$, where $\nu_1\in \kappa^\tau(df(\fra{s}))$ and $\nu_2\in \kappa^\tau(df(\fra{s}))^\perp$. The local condition implies that $f_*(W(H))$ preserves $df(\fra{s})$ and acts on $\fra{s}^\perp$ trivially. Therefore, $f_*(w).\nu-\nu=f_*(w).\nu_1-\nu_1\in \kappa^\tau(df(\fra{s}))$.

Since $df(\fra{s})\cap df(\fra{s})^\perp=0$,
$$f_*(w).\nu+\kappa^\tau(df(n))-\nu=0.$$
Since $\nu$ has the trivial stabilizer, $(f_*(w),\kappa^\tau(df(n)))=(e_{W(G)},0)$. Since $f_*$, $df$ and $\kappa^\tau$ are injective, we obtain the conclusion.

\end{proof}
Let us verify the well-definedness  of Definition \ref{def. of char(f)} by the following three steps.
\begin{lem}%[the case of finite cover]
\label{lemma for char in the case of finite cover}
If the restriction of $f$ to $S$ is a finite covering, the above definition is well-defined.
\end{lem}
\begin{proof}
We can write $char(f|_S)$ explicitly as 
$$char(f|_S)([\lambda]_T)=\sum_{m\in \Pi_T/df(\Pi_S)}[{}^tdf(\lambda+\kappa^\tau(m))]_S.$$
Since ${}^tdf$ is injective and the orbit $W^e_{\mathit{aff}}(G).\lambda$ is regular, for any $w,w'\in W(H)$ and $m,m'\in\Pi_T$ such that $(w,n)\neq(w',n')$, the following holds.
$$(w.{}^tdf(\lambda+\kappa^\tau(m))+\kappa^{i^*f^*\tau}(\Pi_S))\cap (w'.{}^tdf(\lambda+\kappa^\tau(m'))+\kappa^{i^*f^*\tau}(\Pi_S))=\emptyset,$$
 that is, $[{}^tdf(w.\lambda+\kappa^\tau(m))]_S\neq [{}^tdf(w'.\lambda+\kappa^\tau(m'))]_S$.

$$char(f|_S)(\sum_{w\in W(G)}[w.\lambda]_T)=\sum_{w\in W(G)}\sum_{m\in\Pi_T/df(\Pi_S)}[{}^tdf(w.\lambda+\kappa^\tau(m))]_S.$$
Since $\Pi_T$ is $W(G)$-invariant and $\kappa^\tau$ is $W(G)$-equivariant
($\innpro{\cdot}{\cdot}{\tau}$ is $W(G)$-invariant), for any $w\in W(G)$ and $m\in \Pi_T$, $w^{-1}.m\in\Pi_T$ and
$$w.\lambda+\kappa^\tau(m)=w.(\lambda+\kappa^\tau(w^{-1}.m)).$$

Let us consider the quotient space $W(G)/f_*(W(H))=\{[W_1],[W_2],\cdots,[W_k]\}$. That is, for any $w$, there uniquely exist $w'\in W(H)$ and $l\in\{1,2,\cdots,k\}$ such that $w=f_*(w')W_l$, where $W_l$ is a fixed representative element of $[W_l]$. Since ${}^tdf\circ f_*(w')=w'.{}^tdf$,
$$\sum_{m\in\Pi_T/df(\Pi_S)}\sum_{w\in W(G)}
[{}^tdf(w.\lambda+\kappa^\tau(m))]_S=
\sum_{l=1}^k\sum_{m\in\Pi_T/df(\Pi_S)}\sum_{w'\in W(H)}[w'.{}^tdf(W_l.\lambda+\kappa^\tau(m))]_S.$$

%Therefore,
%$$char(f|_S)(\sum_{w\in W(G)}[w.\lambda]_T)=\sum_{[w]\in W(G)/W(H)}\sum_{m\in\Pi_T/\Pi_S}(\sum_{w'\in W(H)}[w'.({}^tdf([w].\lambda+\kappa^\tau(m))]_S.$$
%The right hand side is an element of $char(H,f^*\tau)$.

%Where $[w]\in W(G)/W(H)$ means a chosen representative element.
By summing over $W(H)$, ambiguity of the choice of representatives $W_1,W_2,\cdots,W_k$ does not appear. $W(H)$-equivariance of $\kappa^\tau$ implies that for any $m'\in\Pi_T/df(\Pi_S)$, there exists $m\in\Pi_T/df(\Pi_S)$ such that $w'.\kappa^\tau(m)=\kappa^\tau(m')$. Since  $$\sum_{w'\in W(H)}[w'.{}^tdf(W_l.\lambda+\kappa^\tau(m))]_S$$
 is an element of $char(H,f^*\tau)$, so is the right hand side.
\end{proof}
\begin{lem}%[the case of direct product]
\label{lemma for char in the case of direct product}
If $T=S\times S^\perp$, the restriction of $f$ to $S$ is the inclusion into the first factor and $k^*\tau=p_1^*i^*f^*\tau+p_2^*j^*k^*\tau$, where $p_1:S\times S^\perp\to S$ and $p_2:S\times S^\perp\to S^\perp$ are the natural projection, the above definition is well-defined.

\end{lem}
\begin{proof}
We can write $char(f|_S)$ explicitly as $char(f|_S)([\lambda]_T)=[{}^tdf(\lambda)]_S$. So

$$char(f|_S)(\sum_{w\in W(G)}[w.\lambda]_T)=\sum_{w\in W(G)}[{}^tdf(w.\lambda)]_S$$
$$=\sum_{[w]\in W(G)/W(H)}\sum_{w'\in W(H)}[w'.{}^tdf([w].\lambda)]_S$$

Since $\sum_{w'\in W(H)}[w'.{}^tdf(W_l.\lambda)]_S$ is an element of $char(H,f^*\tau)$, so is the right hand side.
\end{proof}

\begin{thm}\label{well-def. of char(f) for general Lie group}
The definition is well-defined for $f:H\to G$ satisfying the decomposable condition.
\end{thm}

\begin{proof}
Let us notice that we can decompose $f$ so that the following commutative diagram holds.
$$\begin{CD}
H @>q>> H/\ker(f) @>i_1>> H/\ker(f)\times S^\perp @>\bar{f}\cdot j>> G \\
@AiAA @A\bar{i}AA @A\bar{i}\times idAA @AkAA \\
S @>q|_S>> S/\ker(f|_S) @>i_1|_S>> S/\ker(f|_S)\times S^\perp @>\overline{f|_S}\cdot j>> T
\end{CD}$$
where $j:S^\perp\hookrightarrow T\subseteq G$ is the orthogonal completion torus defined in Corollary \ref{well-def. of orthogonal completion of tori}. 
Firstly, we assume that $f$ is injective for simplicity.

Lemma \ref{lemma for char in the case of finite cover} and \ref{lemma for char in the case of direct product} imply that the following commutative diagram holds.
$$\begin{CD}
char(G,\tau) @>char(f\cdot j)>> char(H\times S^\perp,(f\cdot j)^*\tau) @>char(i_1)>> char(H,f^*\tau) \\
@Vchar(k)VV @Vchar(i\times id)VV @Vchar(i)VV \\
char(T,k^*\tau) @>char(f|_S \cdot j)>> char(S\times S^\perp,(f|_S\cdot j)^*\tau) @>char(i_1|_S)>> char(S,i^*f^*\tau)
\end{CD}$$

Theorem \ref{partial functoriality} tells us that $char(f|_S)=char(i_1|_S)\circ char(f|_S \cdot j)$.

If we define $char(f)$ by $char(i_1)\circ char(f\cdot j)$, $char(f)$ is compatible with $char(f|_S)$, that is, the following commutative diagram holds.
$$\begin{CD}
char(G,\tau) @>char(f)>> char(H,f^*\tau)\\
@Vchar(k)VV @Vchar(i)VV \\
char(T,k^*\tau) @>char(f|_S)>> char(S,i^*f^*\tau)
\end{CD}$$
If $f$ is not injective, we can verify the well-definedness with the same way.

\end{proof}

%安直にやってはいけないための反例
\subsubsection{An example which shows that we should not easily define $char(f)$ for general group}

If one easily define $char(f)$, just like the case in tori, $char(f)([\lambda]_G)$ may be the sum of $W^e_{\mathit{aff}}(H)$-orbits ${}^tdf(W(G).(\lambda+\kappa^\tau(\Pi_T))$. However, this is not compatible with the case of tori. We give an example.

Let us define $f:\mathbb{T}\to U(3)$ as
$$f(z):=\begin{pmatrix} z & 0 & 0 \\ 0 & 1 & 0 \\ 0 & 0 & 1\end{pmatrix}$$
and $\tau$ be a positive central extension of $LU(3)$ such that induced homomorphism is represented as
$$\kappa^\tau=\begin{pmatrix} -3 & 0 & 0 \\ 0 & -3 & 0 \\ 0 & 0 & -3 \end{pmatrix}$$
where the maximal torus of $U(3)$ is the diagonal matrices. Let $i:\mathbb{T}^3\to U(3)$ be the natural inclusion. We write $f$ as $i_1$ if we regard $f$ as $\mathbb{T} \to \mathbb{T}^3$. That is, the following commutative diagram holds.
$$\xymatrix{
\bb{T} \ar@{=}[d] \ar[r]^f & U(3)  \\
\bb{T} \ar[r]^{i_1} & \bb{T}^3 \ar[u]^i
}$$

Since $W(U(3))\cong\mathfrak{S}_3$, $W^e_{\mathit{aff}}(U(3))$-regular orbit is 
$W^e_{\mathit{aff}}(U(3))\cdot\begin{pmatrix}0 \\ 1 \\ 2 \end{pmatrix}.$ Therefore, $K_{U(3)}^{\tau+3}(U(3))\cong\mathbb{Z}$.

Since $\kappa^{f^*\tau}=(-3)$, $K^{f^*\tau+1}_\bb{T}(\bb{T})\cong \bb{Z}^3$.

We can compute $char(i)$ and $char(i_1)$ as follows.
$$char(i)(\begin{bmatrix}0 \\ 1 \\ 2 \end{bmatrix}_{U(3)})=\begin{bmatrix}0 \\ 1 \\ 2 \end{bmatrix}_{\mathbb{T}^3}+\begin{bmatrix}0 \\ 2 \\ 1 \end{bmatrix}_{\mathbb{T}^3}+\begin{bmatrix}1 \\ 0 \\ 2 \end{bmatrix}_{\mathbb{T}^3}+\begin{bmatrix}1 \\ 2 \\ 0 \end{bmatrix}_{\mathbb{T}^3}+\begin{bmatrix}2 \\ 0 \\ 1 \end{bmatrix}_{\mathbb{T}^3}+\begin{bmatrix}2 \\ 1 \\ 0 \end{bmatrix}_{\mathbb{T}^3}.$$
$$char(i_1)(\begin{bmatrix}a \\ b \\ c \end{bmatrix}_{\bb{T}^3})=[a]_\bb{T}$$

Therefore, $char(i_1)\circ char(i)(\begin{bmatrix}0 \\ 1 \\ 2 \end{bmatrix}_{U(3)})=2([0]_\bb{T}+[1]_\bb{T}+[2]_\bb{T}).$

However, ${}^tdf(W^e_{\mathit{aff}}(U(3))\cdot\begin{pmatrix}0 \\ 1 \\ 2 \end{pmatrix})$ is entire of $\Lambda^{f^*\tau}_\mathbb{T}$, that is $char(f)(\begin{bmatrix} 0 \\ 1 \\ 2\end{bmatrix}_{U(3)})$ seems to be $[0]_\bb{T}+[1]_\bb{T}+[2]_\bb{T}$. This example tells us that we should not define just like the case of tori using the correspondence of orbits.

%ねじれK理論%このたびの結果
\section{The quasi functor $t.e.K$}\label{Modified twisted equivariant K-functor}

%一般に使える構成
\subsection{General constructions}
In this section, we describe general constructions, K\"unneth isomorphisms, family index maps and Mackey decompositions.

K\"unneth isomorphisms will be verified by a canonical way, that is, Mayer Vietoris argument just like ordinary $K$-theory \cite{Karoubi}. By use of it, we define family index maps.

A Mackey decomposition is an analogue of the isomorphism in un-twisted case $K_G(X)\cong K(X)\otimes R(G)$ for trivial $G$ space $X$.

\subsubsection{K\"unneth isomorphisms and family index maps}\label{Kunneth isomorphisms and family index map}
In this section, we verify K\"unneth isomorphisms and define family index maps for twisted equivariant $K$-theory in special cases. This is very important to construct $i_1^\#$, where $i_1$ is the natural inclusion into the first factor of the direct product.

Let us start from more general construction, the tensor product.
\begin{dfn}%[tensor products]
\label{tensor products}
We can define the graded commutative and associative tensor product 
$$K_G^{\tau_1+n_1}(X)\otimes K_G^{\tau_2+n_2}(X)\to K_G^{\tau_1+\tau_2+n_1+n_2}(X)$$
by well known formula
$$[\{\mathcal{F}_x\}_{x\in X}]\otimes[\{\mathcal{G}_x\}_{x\in X}]\mapsto[\{\mathcal{F}_x\otimes id+\epsilon\otimes\mathcal{G}_x\}_{x\in X}],$$
where $\epsilon$ is the grading involution $id\oplus(-id)$. Associativity follows from Kuiper's theorem.
\end{dfn}
\begin{rmk}\label{rmk about Kunneth formula}
(1) When $G$ acts on $X$ trivially, we have a natural map $K^{\tau_2+n_2}(X)\to K_G^{\tau_2+n_2}(X)$ by regarding a family of Fredholm operators equivariant one via the trivial action. Through this map, we can define the tensor product
$$K_G^{\tau_1+n_1}(X)\otimes K^{\tau_2+n_2}(X)\to K_G^{\tau_1+\tau_2+n_1+n_2}(X).$$
We need it for K\"unneth isomorphism in our cases.

(2) We can define a ``formal $K$-theory class'' $\{\mathcal{F}_x\}_{x\in X}$ as a family of Fredholm operators not compactly supported. Moreover, we can define as
$$[\{\ca{G}_x\}_{x\in X}]\otimes\{\ca{F}_x\}_{x\in X}:=[\{\ca{G}_x\otimes id+\epsilon\otimes \ca{F}_x\}_{x\in X}]$$
 for a $K$-theory class $[\{\ca{G}_x\}_{x\in X}]\in K^{\tau+n}_G(X)$ just like de-Rham comlex $\Omega^n(X)\otimes\Omega^m_c(X)\to\Omega^{n+m}_c(X)$. We need this construction for a Mackey decomposition in Section \ref{section about Mackey decomposition}.
\end{rmk}

The following is clear from the definition.
\begin{pro}\label{ring hom}
$F^*$ and $f^\natural$ preserves $\otimes$.
\end{pro}
Let us verify K\"unneth isomorphisms.

\begin{thm}%[K\"unneth isomorphism]
\label{Kunneth isomorphism}
Let $M$ be a smooth compact manifold which $G$ acts on trivially, $\tau$ be a $G$-equivariant twisting over $M$ and $Z$ be an $n$ dimensional compact manifold. 
Let us suppose that $K^m(Z)$ is free abelian groups for any $m$. Let us consider the direct product $M\times Z$ and the projections $p_1:M\times Z\to M$, $p_2:M\times Z\to Z$. $G$ acts on $M\times Z$ trivially. Then the homomorphism
$$p_1^*\otimes p_2^*:K_G^{\tau+*}(M)\otimes K^*(Z)\to K^{p_1^*\tau+*}_G(M\times Z)$$
is an isomorphism. Therefore, we can define the K\"unneth isomorphism
$$\Phi_{M\times Z,G,\tau}:K^{p_1^*\tau+*}_G(M\times Z)\xrightarrow{\cong} K_G^{\tau+*}(M)\otimes K^*(Z)$$
by $(p_1^*\otimes p_2^*)^{-1}$, where $K^*$ means $K^0\oplus K^1$.

%Moreover K\unneth isomorphism is compatible with Mackey decomposition when $G$ is a torus.
\end{thm}
\begin{proof}
We write the restriction of a twisting $\tau$ to $U$ which is a subset of $M$ as $\tau|_U$ for simplicity.

We can take a finite compact contractible cover $\{U_i\}_{i=1}^{i=N}$ of $M$ such that $\bigcap_{i\in I}U_i$ is contractible for any $I\subseteq\{1,2,\cdots,N\}$, for example triangulation.

Since $U_i$ is compact and contractible, we can assume that $U_i$ is a point $\{pt\}$ and $\tau$ determines the central extension $G^\tau$ of $G$. Firstly, we verify that 
$$p_1^*\otimes p_2^*:K_G^{\tau|_{\{pt\}}+*}(\{pt\})\otimes K^*(Z)\xrightarrow{\cong}K_G^{p_1^*(\tau|_{\{pt\}})+*}(\{pt\}\times Z)$$
is an isomorphism.

Since $Z$ is a compact manifold, it has a finite compact contractible cover $\{W_j\}_{j=1}^{i=N'}$ such that $\bigcap_{j\in J}W_j$ is contractible for any $J\subseteq\{1,2,\cdots,N'\}$, for example triangulation. Since $W_j$ is compact and contractible, we can assume that $W_j$ is also a point.
$$K_G^{\tau|_{\{pt\}}+*}(\{pt\})\otimes K^*(\{pt\}) \cong K^*_{G^\tau}(\{pt\})(1)\otimes K^*(\{pt\})$$
$$\cong K^*_{G^\tau}(\{pt\}\times\{pt\})(1)\cong K_G^{p_1^*(\tau|_{\{pt\}})+*}(\{pt\}\times\{pt\}),$$
where $(1)$ means that the weight of $i(U(1))\subseteq G^\tau$ is $1$, that is, $K^*_{G^\tau}(\{pt\})(1)$ is the Grothendieck completion of the semigroup of $G^\tau$-equivariant vector spaces which $i(U(1))\subseteq G^\tau$ acts on by scaler multiplication. The addition of this semigroup is defined by the direct sum. The first and the last isomorphisms are verified in \cite{FHT1} Proposition 3.5 i). The middle isomorphism is obtained by $p_1^*\otimes p_2^*$.
%
%$$K^{p_1^*\tau+*}_G(\{pt\}\times \{pt\})=K^*_{G^\tau}(\{pt\}\times \{pt\})(1)\cong R(G^\tau)(1)\otimes K^*(\{pt\})$$
%$$=R^\tau(G)\otimes K^*(\{pt\})\cong K_G^{\tau+*}(\{pt\})\otimes K^*(\{pt\}).$$
%Therefore, $p_1^*\otimes p_2^*$ gives an isomorphism for compact contractible $W_j$.

These isomorphisms and Mayer Vietoris argument imply that
$$p_1^*\otimes p_2^*:K_G^{\tau|_{\{pt\}}+*}(\{pt\})\otimes K^*(W_1\cup W_2)\xrightarrow{\cong}K_G^{p_1^*(\tau|_{\{pt\}})+*}(\{pt\}\times (W_1\cup W_2))$$
is an isomorphism. By the induction on the number of $k$ in $\bigcup_{i=1}^kW_j$, we have the conclusion.

With the same way, we can verify that
$$p_1^*\otimes p_2^*:K_G^{\tau|_{U_1\cup U_2}+*}(U_1\cup U_2)\otimes K^*(Z)\xrightarrow{\cong}K_G^{p_1^*(\tau|_{U_1\cup U_2})+*}((U_1\cup U_2)\times Z)$$
is an isomorphism. By the induction on the number of $k$ in $\bigcup_{i=1}^kU_i$, we have the conclusion.

\end{proof}

Let us define family index maps for our cases.

When the tangent bundle $TZ$ is trivial, and we fix a trivialization of $TZ$ and an orientation of $Z$, we can define isomorphisms
$$K^n(Z)\cong K^{TZ}(Z)\cong K(Z,Cl(TZ))$$
by fixing the trivial Spinor bundle $Z\times \Delta_{\bb{R}^n}$, where $n$ is the dimension of $Z$ (Example \ref{example local spinor}). 
Let us recall that an orientation of vector space $V$ determines the Spinor $\Delta_V$ (\cite{Furuta} Theorem 2.15). %The second symbol is defined in \cite{Karoubi}, the 
%third one is defined in \cite{Furuta}.

The following lemma is one of the consequence of Atiyah-Singer index theorem when $n$ is even.

\begin{lem}[\cite{CW}]\label{def. of ind.}
We can define the index map $ind_Z:K^n(Z)\to \mathbb{Z}$. It depends on the choice of the isomorphism $K^n(Z)\cong K(Z,Cl(TZ))$, that is, it depends on the choice of the trivialization of $TZ$ and the orientation of $Z$. $ind_Z$ is given by the push-forward along the unique map $Z\to\{pt\}$.
\end{lem}
For simplicity, we assume that $K_G^{\tau}(M)$ or $K_G^{\tau+1}(M)$ is equal to $0$. Then we can define a family index map. We take $n_0=0$ or $1$ so that $K^{\tau+n_0+1}_G(M)=0$.
\begin{dfn}%[family index map]
\label{family index map}
If we fix a trivialization of $TZ$ and an orientation of $Z$, 
$$ind_{M\times Z\to M}:K^{p_1^*\tau+n_0+n}_G(M\times Z)\to K_G^{\tau+n_0}(M)$$
 is defined by the composition of the following sequence
$$K^{p_1^*\tau+n_0+n}_G(M\times Z)\xrightarrow{\Phi_{M\times Z,G,\tau}}K^{\tau+n_0}_G(M)\otimes K^n(Z)
\xrightarrow{id\otimes ind_Z}K^{\tau+n_0}_G(M).$$
\end{dfn}
%We write $ind_{M\times Z\to M}$ simply $p_{1!}$ too.

%With the same way, we can prove following theorem.

%\begin{rmk}\label{dependence on orientation}
%Constructed map $ind$ depends on the choice of orientation of $Z$.
%\end{rmk}

We can extend these constructions to non-trivial $G$-manifold $M$.

\begin{thm}%[generalization of K\"unneth formula]
\label{generalization of Kunneth formula}
Theorem \ref{Kunneth isomorphism} 
is valid for non-trivial $G$-manifold $M$.
\end{thm}

\begin{proof}
We have the ``slice theorem''. That is, each $x\in M$ has a closed $G$-neighborhood of the form $G\times_{G_x}S_x$, where the slice $S_x$ is equivariantly contractible under the stabilizer $G_x$.
Therefore, we have an isomorphism 
$$K_G^{p_1^*(\tau|_{G\times_{G_x}S_x})+*}((G\times_{G_x}S_x)\times Z)\cong K^{p_1^*(\tau|_{G\times_{G_x}S_x})+*}_G(G\times_{G_x}(S_x\times Z))$$
$$\cong K_{G_x}^{p_1^*(\tau|_{S_x})+*}(S_x\times Z),$$
where the last isomorphism can be defined by the pull back along the natural inclusion $$S_x\times Z\ni(s,z)\mapsto [(e_G,s,z)]\in G\times_{G_x}(S_x\times Z).$$
Since $S_x$ is equivariantly contractible under the stabilizer $G_x$, we have an isomorphism 
$$K_{G_x}^{p_1^*(\tau|_{S_x})+*}(S_x\times Z)
\cong K_{G_x}^{p_1^*(\tau|_{\{x\}})+*}(\{x\}\times Z).$$
 $K_{G_x}^{p_1^*(\tau|_{\{x\}})+*}(\{x\}\times Z)$ is isomorphic to $K_{G_x}^{\tau|_{\{x\}}+*}(\{x\})\otimes K^*(Z)$ by Theorem \ref{Kunneth isomorphism}. From the isomorphism $K_{G_x}^{\tau|_{\{pt\}}+*}(\{x\})\cong K^{\tau|_{G\times_{G_x}S_x}+*}_G(G\times_{G_x}S_x)$, the natural isomorphism 
$$K_G^{p_1^*(\tau|_{G\times_{G_x}S_x})+*}(G\times_{G_x}(S_x\times Z))\cong K_G^{\tau|_{G\times_{G_x}S_x}+*}(G\times_{G_x}S_x)\otimes K^*(Z)$$
holds. Mayer Vietoris argument implies the theorem just like Theorem \ref{Kunneth isomorphism}.
\end{proof}

When $K^{\tau+n_0+1}_G(M)=0$ and $Z$ has the trivial tangent bundle, we can extend a family index map to non-trivial $G$-manifold with the same way in the case of trivial action.

\begin{dfn}%[generalization of the family index map]
\label{generalization of family index map}
A family index map
$$ind_{M\times Z\to M}:K_G^{p_1^*\tau+n_0+n}(M\times Z)\to K_G^{p_1^*\tau+n_0}(M)$$
is defined just like Definition \ref{family index map}.
\end{dfn}
The following lemmas about functorialities of K\"unneth isomorphisms are clear from Proposition \ref{ring hom}.

Let $X$ and $Y$ be smooth manifolds, $G$ be a compact Lie group and $G$ act on $X$ and $Y$ smoothly. Let a smooth proper map $F:Y\to X$ be $G$-equivariant, $f:H\to G$ be a smooth group homomorphism and $\tau$ be a $G$-equivariant twisting over $X$, which is also $H$-equivariant through $f$. When we regard $\tau$ as $H$-equivariant one, we write it as $f^{\natural}\tau$.

\begin{lem}%[naturality of K\"unneth formula for $F^*$]
\label{naturality of Kunneth formula for F^* for G}
The following commutative diagram holds.
$$
\begin{CD}
K_G^{p_1^*\tau +*}(M\times Z) @>\Phi_{M\times Z,G,\tau}>> K_G^{\tau+*}(M)\otimes K^*(Z) \\
@V(F\times id)^* VV @VF^*\otimes id VV \\
K_G^{p_1^*F^*\tau+*}(N\times Z) @>\Phi_{N\times Z,G,F^*\tau}>> K_G^{F^*\tau+*}(N)\otimes K^*(Z)
\end{CD}
$$
\end{lem}
%\begin{proof}
%This is clear from that $F^*$ preserves $\otimes$.
%\end{proof}
\begin{lem}%[naturality of K\"unneth formula for $f^\natural$]
\label{naturality of Kunneth formula for j^natural for G}
The following commutative diagram holds.
$$\begin{CD}
K_G^{p_1^*\tau+*}(M\times Z) @>\Phi_{M\times Z,G,\tau}>> K_G^{\tau+*}(M)\otimes K^*(Z) \\
@Vf^\natural VV @Vf^\natural\otimes id VV \\
K_H^{p_1^* f^\natural \tau +*}(M\times Z) @>\Phi_{M\times Z,H,f^\natural\tau}>> K_H^{f^\natural\tau+*}(M)\otimes K^*(Z)
\end{CD}$$
\end{lem}
%\begin{proof}
%This is clear from that $f^\natural$ preserves $\otimes$.% and 
%$$K_G^{p_1^*\tau+n_0+n}(M\times Z)\supseteq K_G^{p_1^*\tau+n_0}(M\times Z)\otimes K_G^{n}(M\times Z)$$
%we have a commutative diagram
%$$
%\begin{CD}
%K_G^{p_1^*\tau+n_0+n}(M\times Z) @<p_1^*\otimes p_2^*<< K_G^{p_1^*\tau +n_0}(M)\otimes K^n_G(Z) @<id\otimes(\otimes[1])<< K_G^{p_1^*\tau +n_0}(M)\otimes K^n(Z) \\
%@Vj^\natural VV @Vj^\natural\otimes j^\natural VV @Vj^\natural\otimes idVV \\
%K_H^{p_1^*j^\natural\tau+n_0+n}(M\times Z) @<p_1^*\otimes p_2^*<< K_H^{p_1^*j^\natural\tau +n_0}(M)\otimes K^n_H(Z) @<id\otimes(\otimes[1])<< K_H^{p_1^*j^\natural\tau +n_0}(M)\otimes K^n(Z)
%\end{CD}
%$$
%The compositions of horizontal arrows are the inverse maps of the K\unneth isomorphisms.
%\end{proof}

These lemmas imply the followings.% $ind_{M\times Z\to M}$.
\begin{pro}%[naturality of index map for $F^*$]
\label{naturality of ind. for F^*}
The following commutative diagram holds.
$$
\begin{CD}
K^{p_1^*\tau+n_0+n}_G(M\times Z) @>ind_{M\times Z\to M}>> K^{\tau+n_0}_G(M) \\
@V(F\times id)^*VV @VF^*VV \\
K^{p_1^*F^*\tau+n_0+n}_G(N\times Z) @>ind_{N\times Z\to N}>> K^{F^*\tau+n_0}_G(N) 
\end{CD}
$$
\end{pro}
\begin{proof}
Lemma \ref{naturality of Kunneth formula for F^* for G} implies the commutativity of the left side of the following diagram.
$$
\begin{CD}
K_G^{p_1^*\tau +n_0+n}(M\times Z) @>\Phi_{M\times Z,G,\tau}>> K_G^{\tau+n_0}(M)\otimes K^n(Z) @>id\otimes ind_Z>> K_G^{\tau+n_0}(M) \\
@V(F\times id)^*VV @VF^*\otimes idVV @VF^*VV \\
K_G^{p_1^*F^*\tau +n_0+n}(N\times Z) @>\Phi_{N\times Z,G,\tau}>> K_G^{F^*\tau+n_0}(N)\otimes K^n(Z) @>id\otimes ind_Z>> K_G^{\tau+n_0}(N)
\end{CD}
$$
The compositions of horizontal arrows are $ind_{M\times Z\to M}$ and $ind_{N\times Z\to N}$ respectively. The commutativity of the right side is clear.

\end{proof}
\begin{pro}%[naturality for $f^\natural$]
\label{naturality of ind. for j^natural}
The following commutative diagram holds.
$$
\begin{CD}
K^{p_1^*\tau+n_0+n}_G(M\times Z) @>ind_{M\times Z\to M}>> K^{\tau+n_0}_G(M) \\
@Vf^\natural VV @Vf^\natural VV \\
K^{p_1^*f^\natural\tau+n_0+n}_H(M\times Z) @>ind_{M\times Z\to M}>> K^{f^\natural\tau+n_0}_H(M)
\end{CD}
$$
\end{pro}
\begin{proof}
We can verify it with the same way in the proof of the above proposition from Lemma \ref{naturality of Kunneth formula for j^natural for G}.
\end{proof}

%Mackey Lemma%
\subsubsection{A Mackey decomposition}\label{section about Mackey decomposition}
We use a Mackey decomposition to compute $K^{\tau+dim(T)}_T(T)$ for a torus, which is an analogue of the isomorphism in un-twisted equivariant $K$-theory $K_G(X)\cong K(X)\otimes R(G)$ for trivial $G$-space $X$. We state it for trivial $T$-spaces. We omit the proof. See \cite{FHT3} for detail.

Let $T$ be a torus, $X$ be a smooth manifold, $T$ act on $X$ trivially, $P$ be a $T$-equivariant projective bundle over $X$, and $\tau$ be the twisting represented by $P$. Covering space $Y$ and the associated twisting $\tau'$ over $Y$ are defined by the following.

(i) a $T$-equivariant family, parametrized by $X$, of central extension $T^\tau$ of $T$ by $U(1)$;

(ii) a covering space $p:Y\to X$, whose fibers label the isomorphism classes of irreducible, $\tau$-twisted representations $\Lambda_T^\tau$;

(iii) a tautological projective bundle $\mathbb{P}R\to Y$ whose fiber $\mathbb{P}R_\lambda$ at $\lambda\in Y$ is the projectification of the $\tau$-twisted representation of $T$ labelled by $\lambda$;

(iv) a formal $\mathbb{P}R$-twisted $K$-theory class $[R]$, represented by $R$;

(v) a non-equivariant twisting $\tau'$ over $Y$ and an isomorphism as $T$-equivariant twistings $\tau'\cong p^*\tau-\mathbb{P}R$.

\begin{thm}[\cite{FHT3}]\label{theorem Mackey decomposition}
There is an isomorphism
$$Mackey_{X,T,\tau}:K_T^{\tau+k}(X)\to K^{\tau'+k}(Y)$$
associated with an isomorphism $\tau'\cong p^*\tau-\bb{P}R$ in (v).

%The twisted $K$-theories $K^{\tau'}(Y)$ and $K_T^\tau(X)$ are naturally isomorphic. We write this isomorphism as $Mackey_{X,T,\tau}$.
\end{thm}
%\begin{rmk}
%This isomorphism is natural with respect to $F:X_1\to X_2$ in the sense that the isomorphism between two twistings can be pulled back along $\widetilde{F}:Y_1\to Y_2$, where $Y_j$ is the covering space of $X_j$ defined above ($j=1,2$).
%\end{rmk}
We verify naturality with respect to pull back along a continuous map between two spaces and pull back of group action along a local injection between two tori. Before that, we verify Theorem \ref{isomorphism between t.e.K and char}.
\begin{cor}\label{computation of K using a Mackey decomposition}%[$M.d._T$]
Let $T$ be a torus and $\tau$ be a positive central extension of $LT$ and the associated $T$-equivariant twisting over $T$. Fix an orientation of $T$.

Then we have an isomorphism $M.d._T:K^{\tau+dim(T)}_T(T)\to char(T,\tau)$. Moreover, $K^{\tau+dim(T)+1}_T(T)\cong 0$.
\end{cor}
\begin{proof}
Associated covering space $Y$ is $\Lambda_T^\tau\times_{\Pi_T}\mathfrak{t}$, $\Pi_T$ acts on $\Lambda_T^\tau$ through $\kappa^\tau$ (Definition \ref{morphism kappa}) and on $\fra{t}$ via the translation. $\kappa^\tau(n)$ represents how the character changes when one travels in $T$ along the geodesic loop $n$. Since $\tau$ is positive, $\kappa^\tau$ is injective. Therefore, $\Lambda_T^\tau\times_{\Pi_T}\mathfrak{t}$ has a structure of a trivial vector bundle over $\Lambda_T^\tau/\kappa^\tau(\Pi_T)$ whose fiber is $\mathfrak{t}$. We write this vector bundle as 
$$\pi:\Lambda_T^\tau\times_{\Pi_T}\mathfrak{t}\to \Lambda_T^\tau/\kappa^\tau(\Pi_T).$$ Moreover, since any connected component of $\Lambda_T^\tau\times_{\Pi_T}\mathfrak{t}$ is contractible, a twisting over $\Lambda_T^\tau\times_{\Pi_T}\mathfrak{t}$ is automatically trivial. So we have the following series of isomorphisms
$$K^{\tau+dim(T)}_T(T)\xrightarrow{Mackey_{T,T,\tau}} K^{dim(T)}(\Lambda_T^\tau\times_{\Pi_T}\mathfrak{t})\xrightarrow{\pi_!} K( \Lambda_T^\tau/\kappa^\tau(\Pi_T))\cong char(T,\tau).$$

The first isomorphism is the Mackey decomposition and $\pi_!$ is Thom isomorphism, so this isomorphism depends on the orientation of $T$.

Since $K^1(\Lambda_T^\tau/\kappa^\tau(\Pi_T))\cong0$, $K^{\tau+dim(T)+1}_T(T)\cong 0$.
\end{proof}
We omit the proof of the following theorem. You can find the proof in \cite{FHT3}.
\begin{thm}%[$M.d._G$]
Let $k:T\hookrightarrow G$ be a chosen maximal torus. We can define the isomorphism $M.d._G:K^{\tau+rank(G)}_G(G)\to char(G,\tau)$ so that the following diagram commutes.
$$\begin{CD}
t.e.K(G,\tau) @>k^*\circ k^\natural>> t.e.K(T,k^*\tau) \\
@VM.d._GVV @VM.d._TVV \\
char(G,\tau) @>char(k)>> char(T,k^*\tau)
\end{CD}$$
\end{thm}
Let us construct a Mackey decomposition, and verify naturality with respect to $F^*$ and $f^\natural$, where $F$ is a proper smooth map between two manifolds and $f$ is a locally injective group homomorphism.\\\\
\textit{Construction}: 
The isomorphism stated above is the inverse of the composition of the following sequence
$$K^{\tau'}(Y)\to K^{\tau'}_T(Y)\cong K^{p^*\tau-\mathbb{P}R}_T(Y)\xrightarrow{\otimes[R]}
K_T^{p^*\tau}(Y)\xrightarrow{p_!}K^\tau_T(X),
$$
where the first map is defined by regarding a family of Fredholm operators as an equivariant one via the trivial action, the second isomorphism is defined by the isomorphism of twistings, the third map is defined by the tensor product with the formal $K$-theory class $[R]$, and the last map is the push-forward along $p:Y\to X$.

%\begin{lem}\label{M.d.'s independence of isom. of twisting}
%Composition of the above series is indepondent of the choice of the isomorphism $\tau'\cong p^*\tau-\mathbb{P}R$.
%\end{lem}

Since the central extension of $T$ splits and $T$ is abelian group, $T^\tau$ is also an abelian group (\cite{FHT1} Lemma 4.1). So $\tau$-twisted irreducible representations of $T$ are $1$ dimensional. Therefore, the fiber $\mathbb{P}R_\lambda$ of $\mathbb{P}R$ at $\lambda$ is a point, that is, $\mathbb{P}R=Y$. In other words, the twisting $\mathbb{P}R\cong 0$.
\begin{rmk}
If $T$-equivariant twistings $\tau$ and $\sigma$ are isomorphic, the isomorphism $K^\tau_T(X)\to K^\sigma_T(X)$ depends on the choice of the ``homotopy class'' of isomorphisms between $\tau$ and $\sigma$ (Proposition \ref{classification of twistings}). ``Homotopy classes'' of isomorphisms are classified by the second equivariant cohomology group $H^2_T(X)$, that is, isomorphism classes of equivariant complex line bundles, just as $H^2(X)$ acts on $K^*(X)$ via the tensor product with the corresponding complex line bundle, where we regard $K^*(X)$ as not a ring but a group.

So the isomorphism $K^{\tau'}_T(Y)\cong K^{p^*\tau-\mathbb{P}R}_T(Y)$ is determined not automatically but by taking an isomorphism $\tau'\cong p^*\tau-\mathbb{P}R$.
\end{rmk}
Let us verify naturality of a Mackey decomposition.

\begin{thm}%[naturality for $F^*$]
\label{naturality for F^*}
Let $X$ and $X'$ be manifolds, $T$ be a torus acting on $X$ and $X'$ trivially, $\tau$ be a $T$-equivariant twisting over $X$, and $F:X'\to X$ be a smooth proper map. $p:Y_T\to X$ and $p':Y_T'\to X'$ are the covering spaces constructed above.

The following commutative diagram holds.
$$\begin{CD}
K^{\tau'+k}(Y_T) @>\widetilde{F}^*>> K^{\widetilde{F}^*\tau'+k}(Y_T') \\
@AMackey_{X,T,\tau}AA @AAMackey_{X',T,F^*\tau}A \\
K_T^{\tau+k}(X) @>F^*>> K^{F^*\tau+k}_T(X') \\
\end{CD}$$
where $\widetilde{F}$ is defined in the proof.
\end{thm}
\begin{proof}
Firstly, we verify naturality of the covering space.
\begin{lem}%[naturality of covering spaces]
\label{naturality of covering spaces}
The covering space $Y'_T$ over $X'$ associated with $F^*\tau$ is homeomorphic to $F^*Y_T$. Therefore, we can define a continuous map $\widetilde{F}:Y'_T\to Y_T$ and the following commutative diagram holds.
$$
\begin{CD}
Y'_T\cong F^* Y_T @>\widetilde{F}>> Y_T \\
@Vp' VV @Vp VV \\
X' @>F>> X
\end{CD}
$$
\end{lem}
\begin{proof}
Firstly, we verify naturality of $T^\tau$ family with $F^*$, then we verify the lemma.

Let $P$ be a projective bundle which represents the twisting $\tau$ and $\rho$ be the action of $T$ on $P$. $T^\tau_x$ is defined by the pull back of the central extension $U(1)\to U(\mathcal{H}_x)\to PU(\mathcal{H}_x)$ along the homomorphism $\rho_x:T\to U(P_x)=PU(\mathcal{H}_x)$, that is, $T^\tau_x=\rho_x^*(U(\mathcal{H}_x))$, where $\ca{H}_x$ is a chosen lift of $P_x$.

Since the action of $T$ on $F^*P$ is defined by $\rho_{x'}=\rho_{F(x')}$, $T^{F^*\tau}_{x'}=\rho_{x'}^*(U(\mathcal{H}_{F(x')}))=\rho_{F(x')}^*(U(\mathcal{H}_{F(x')}))=T^\tau_{F(x')}$.

Let us verify the lemma. Let us recall the topology of $Y_T$. Let $\mathcal{U}$ be a sufficient small open neighborhood of $x_0\in X$ in $X$. Then we can trivialize $P$ on $\mathcal{U}$, that is, $P|_\mathcal{U}\cong \mathcal{U}\times\mathbb{P}(\mathcal{H})$. By use of this homeomorphism, we can trivialize the $T^\tau$ family over $\mathcal{U}$, that is, $\{T_x^\tau\}_{x\in\mathcal{U}}\cong\mathcal{U}\times T^\tau_{x_0}$. So we can trivialize as $Y_T|_\mathcal{U}\cong\mathcal{U}\times\Lambda_{T,x_0}^\tau$. Since this trivialization depends only on one of $P$, we can define the associated trivialization of $Y_{T'}$ on $F^{-1}(\mathcal{U})$ by well known method.

Since $\Lambda^\tau_{T,F(x')}=\Lambda^{F^*\tau}_{T,x'}$, we can define a fiber map $\widetilde{F}:Y'_T\to Y_T$ by $\widetilde{F}(x',\lambda):=(F(x'),\lambda)$, where $x'\in X', \lambda\in\Lambda^{F^*\tau}_{T,x'}$. Let us notice that $p\circ \widetilde{F}=F\circ p'$ from the definition of $\widetilde{F}$. Continuity of this map follows from that the restriction of $\widetilde{F}$ to $F^{-1}(\mathcal{U})$ can be written as $F|_{F^{-1}(\ca{U})}\times id_{\Lambda_{T,x_0}^\tau}$ through the above trivialization. By the universality of the fiber product, there exists the unique continuous fiber map $F^*Y_T\to Y'_T$. Since this map is surjective and fiberwisely homeomorphic, it is a homeomorphism.
$$\begin{xymatrix}{
F^*Y_T=X'\times_XY_T \ar[rrd]^{p_2} \ar@{.>}[rd]|{\exists!} \ar[rdd]_{p_1} & & \\
& Y'_{T} \ar[r]_{\widetilde{F}} \ar[d]^{p'} & Y_T \ar[d]^p \\
& X' \ar[r]_F & X }\end{xymatrix}$$
\end{proof}

Let us consider the following diagram
$$
\begin{CD}
K^{\tau'+k}(Y_T) @>>> K^{\tau'+k}_T(Y_T) @>\cong>>K_T^{p^*\tau-\mathbb{P}R_T+k}(Y_T) @>\otimes[R_T]>>  \\
@V\widetilde{F}^*VV @V\widetilde{F}^*VV @V\widetilde{F}^* VV \\
K^{\widetilde{F}^*\tau'+k}(Y'_T) @>>> K^{\widetilde{F}^*\tau'+k}_T(Y'_T) @>\cong>>K_T^{\widetilde{F}^*p^*\tau-\mathbb{P}\widetilde{F}^*R_T+k}(Y'_T) @>\otimes[R_T']>> \\
\end{CD}
$$
$$
\begin{CD}
 @>\otimes[R_T]>> K_T^{p^*\tau+k}(Y_T) @>p_!>> K_T^{\tau+k}(X) \\
@. @V \widetilde{F}^* VV @VF^*VV\\
 @>\otimes[R_T']>>K_T^{\widetilde{F}^*p^*\tau+k}(Y'_T) @>p'_!>>K_T^{F^*\tau+k}(X').
\end{CD}
$$

The compositions of horizontal sequences are the inverse of $Mackey_{X,T,\tau}$ and $Mackey_{X',T,F^*\tau}$ respectively. If we verify the commutativity of four squares, we obtain the conclusion.

The first is clear.

The second is verified by pull back of isomorphism between two twistings along $\widetilde{F}$.

The third is clear from the following lemma and that $F^*$ preserves $\otimes$.
\begin{lem}%[naturality of the classes $R$]
\label{naturality of the classes R}
Let $R_T$ and $R'_T$ be the formal $K$-theory classes associated with $(X,\tau)$ and $(X',F^*\tau)$ respectively. Then $R_T'$ is canonically isomorphic to $\widetilde{F}^*R_T$.
\end{lem}
\begin{proof}
Let us define an isomorphism as $T$-equivariant vector bundles $\phi:R'_T\to F^*R_T$. Let us recall that
$$R_T=\bigcup_{(x,\lambda)\in Y_T}(\mathbb{C}_\lambda)_x, \text{ }R'_{T}=\bigcup_{(x',\lambda)\in Y'_T}(\mathbb{C}_\lambda)_{x'},$$
where $T_x^\tau$ acts on $\bb{C}_\lambda$ via  character $\lambda\in\Lambda^\tau_{T,x}$.

Let us define $\phi$ as
$$\phi(x',\lambda,z):=(F(x'),\lambda,z),$$
where $x'\in X'$, $\lambda\in\Lambda^{F^*\tau}_{x'}$ and $z\in\mathbb{C}_\lambda.$

Continuity is verified by taking a local trivialization. The isomorphism $R'_T\cong \widetilde{F}^*R_T$ as vector bundles is verified just like the canonical homeomorphism $F^*Y_T\cong Y'_T$. Equivariance is verified as follows.
$$t.\phi(x',\lambda,z):=t.(F(x'),\lambda,z)$$
$$=(F(x'),\lambda,\lambda(t)z)=\phi(x',\lambda,\lambda(t)z),$$
where $t\in T^{F^*\tau}_{x'}=T^\tau_{F(x')}$.
\end{proof}

The fourth follows from Lemma \ref{explicit formula of push-forward with discrete fiber}.
Let $[\{\ca{F}_{(x,\lambda)}\}_{(x,\lambda)\in Y_T}\}]\in K^{p^*\tau+k}_T(Y_T)$. We can compute $F^*\circ p_!([\{\ca{F}_{(x,\lambda)}\}_{(x,\lambda)\in Y_T}])$ and $p'_!\circ\widetilde{F}^*([\{\ca{F}_{(x,\lambda)}\}_{(x,\lambda)\in Y_T}])$ as follows.
$$F^*\circ p_!([\{\ca{F}_{(x,\lambda)}\}_{(x,\lambda)\in Y_T}])=F^*([\{\bigoplus_{\lambda\in\Lambda^\tau_{T,x}}\ca{F}_{(x,\lambda)}\}_{x\in X}])$$
$$=[\{\bigoplus_{\lambda\in\Lambda^{\tau}_{T,F(x')}}\ca{F}_{(F(x'),\lambda)}\}_{x'\in X'}]$$
$$p'_!\circ\widetilde{F}^*([\{\ca{F}_{(x,\lambda)}\}_{(x,\lambda)\in Y_T}])=
p'_!([\{\ca{F}_{(F(x'),\lambda)}\}_{(x',\lambda)\in Y'_T}])$$
$$=[\{\bigoplus_{\lambda\in\Lambda^{\tau}_{T,F(x')}}\ca{F}_{(F(x'),\lambda)}\}_{x'\in X'}].$$
Therefore, the both sides coincide.

\end{proof}

Let us verify naturality with respect to pull back of group action. However, the sense of ``naturality with respect to $f^\natural$'' is not clear, because $K^{\tau'+k}(Y_T)$ is a non-equivariant twisted $K$-group. In this paper, we regard the following theorem as naturality with respect to $f^\natural$.

\begin{thm}%[naturality for $f^\natural$ for our cases]
\label{naturality for f^natural}
Let $p:Y_T\to X$ and $p':Y_{T'}\to X$ be the covering spaces associated with $(T,\tau)$ and $(T',f^\natural\tau)$ respectively.
The following commutative diagram holds.
$$
\begin{CD}
K^{\tau'+k}(Y_T) @>\cong>> K^{\tau+k}_T(X) \\
@V\widetilde{{}^tdf}_! VV @Vf^\natural VV \\
K^{\tau''+k}(Y_{T'}) @>\cong>> K^{f^\natural\tau+k}_{T'}(X)
\end{CD}
$$ where $\tau'$, $\tau''$ and $\widetilde{{}^tdf}$ are defined in the proof.
\end{thm}
\begin{rmk}
The above theorem is an analogue of the following commutative diagram mentioned in Introduction.
$$\begin{CD}
K(\Lambda_T) @>\cong>> R(T)=K_T(\{pt\}) \\
@V({}^tdf)_!VV @Vf^\natural VV \\
K(\Lambda_{T'}) @>\cong>> R(T')=K_{T'}(\{pt\})
\end{CD}$$
Usually, $f^\natural$ is written as $f^*$ which is the pull back of representations along $f$.
\end{rmk}
\begin{proof}
Let us start from naturality of the covering spaces.
\begin{lem}
%Let $f:T'\to T$ be a smooth locally injective group homomorphism, $\tau$ be a $T$-equivariant twisting over $X$, $p:Y_T\to X$, $p':Y_{T'}\to X$ be the covering space associated to $\tau$, $f^\natural\tau$ respectively.

We can define a continuous fiber map $\widetilde{{}^tdf}:Y_T\to Y_{T'}$ such that thefollowing diagram holds.
$$
\begin{CD}
Y_T @>\widetilde{{}^tdf} >> Y_{T'} \\
@V p VV @V p' VV \\
X @= X
\end{CD}
$$
\end{lem}
\begin{proof}
The family $\{T'^{f^\natural\tau}_x\}_{x\in X}$ of central extensions of $T'$ can be obtained by the fiber product $T^{'f^\natural\tau}_x:=T'\times_TT_x^\tau$, that is, the following diagram commutes.
$$\begin{CD}
T^{'f^\natural \tau}_x @>\widetilde{f}_x>> T^\tau_x \\
@VVV @VVV \\
T' @>f>> T
\end{CD}$$
It can be verified with the same way in the proof of Lemma \ref{naturality of covering spaces}.

Since we have a homomorphism $\widetilde{f}_x:T_x'^{f^\natural\tau}\to T_x^\tau$ induced by $f$, we have
$${}^td\widetilde{f}_x:\Lambda_{T,x}^\tau\to\Lambda_{T',x}^{f^\natural\tau},$$
a pull back of representations. We can define $\widetilde{{}^tdf}:Y_T\to Y_{T'}$ by $\widetilde{{}^tdf}(x,\lambda):=(x,{}^td\widetilde{f}_x(\lambda))$. Continuity is verified by use of the trivialization defined in Theorem \ref{naturality for F^*}.
\end{proof}

Formal $K$-theory classes $[R_T]$ and $[R_{T'}]$ have the following functoriality.
\begin{lem}%[naturality of the classes $R$]
Let $[R_T]$ and $[R_{T'}]$ be the formal $K$-theory classes associated with $(T,\tau)$ and $(T',f^\natural \tau)$ respectively. 

We can define a homomorphism of vector bundles $\widehat{\widetilde{{}^tdf}}:R_{T'}\to R_T$, which makes the following diagram commute.
$$
\begin{CD}
R_{T}@>\widehat{\widetilde{{}^tdf}}>>R_{T'} \\
@VVV @VVV \\
Y_{T}@>\widetilde{{}^tdf}>>Y_{T'}
\end{CD}
$$
Through $f:T'\to T$, $T'$ acts on $R_T$. When we regard it as $T'$-equivariant bundle, $\widetilde{{}^tdf}$ is a homomorphism as $T'$-equivariant bundles.
\end{lem}
\begin{proof}
$\widehat{\widetilde{{}^tdf}}(x,\lambda',z):=(x,{}^tdf(\lambda'),z)$ is a homomorphism between two vector bundles. Equivariance is verified by the same method in Lemma \ref{naturality of the classes R}.
\end{proof}
Let us verify the theorem. Let us consider the following diagram.
{\small
$$\xymatrix{
K^{\tau'+k}(Y_T) \ar[r] \ar@{=}[d] \ar@{}[dr]|{(1)} & K^{\tau'+k}_T(Y_T) \ar[r]^{\cong} \ar[d]^{f^\natural} \ar@{}[dr]|{(2)} & K^{p^*\tau-\mathbb{P}R_T+k}_T(Y_T) \ar[r]^{\otimes[R_T]} \ar[d]^{f^\natural} \ar@{}[dr]^{(3)}& 
K^{p^*\tau+k}_T(Y_T) \ar[r]^{p_!} \ar[d]^{f^\natural} \ar@{}[ddr]|{(4)} & K^{\tau+k}_T(X) \ar[dd]^{f^\natural} 
\\
K^{\tau'+k}(Y_T) \ar[r] \ar[d]^{(\widetilde{{}^tdf})_!} \ar@{}[dr]|{(5)}& K^{f^\natural\tau'+k}_{T'}(Y_T) \ar[r]^{\cong} \ar[d]^{\widetilde{({}^tdf)}_!} \ar@{}[dr]|{(6)} & K^{f^\natural(p^*\tau-\mathbb{P}R_T)+k}_{T'}(Y_{T}) \ar[r] \ar@{}[ur]_{\otimes\widetilde{({}^tdf)}^*[R_{T'}]} \ar[d]^{\widetilde{({}^tdf)}_!} \ar@{}[dr]|{(7)} & 
K^{f^\natural p^*\tau+k}_{T'}(Y_T) \ar[d]^{\widetilde{({}^tdf)}_!} & 
\\
K^{\tau''+k}(Y_{T'}) \ar[r] & K^{\tau''+k}_{T'}(Y_{T'}) \ar[r]^{\cong} & K^{p'^*f^\natural\tau-\mathbb{P}R_{T'}+k}_{T'}(Y_{T'}) \ar[r]^{\otimes[R_{T'}]} & 
K^{p'^*f^\natural\tau+k}_{T'}(Y_{T'}) \ar[r]^{p'_!} & K^{f^\natural\tau+k}_{T'}(X) 
}$$}

If we can prove all of the commutativities from (1) to (7), we obtain the required commutativity. Let us notice that $f^\natural p^*\tau=\widetilde{{}^tdf}^*p'^*f^\natural\tau$ and $f^\natural(p^*\tau-\bb{P}R_T)=\widetilde{{}^tdf}^*(p'^*f^\natural\tau-\bb{P}R_{T'})$. Therefore, the above push-forward maps are well-defined.

(1), (5) are clear from the definitions of each maps.

(2), (6) follow from that we can pull back the isomorphism between two twistings. $\tau'$ and $\tau''$ are defined so that $\widetilde{{}^tdf}^*\tau''=f^\natural\tau'$.
%(2), (6) follow from that we can pull back the isomorphism between two twistings.
%We can define two isomorphisms (a) $f^\natural\tau'\cong f^\natural(p^*\tau-\bb{P}R_T)$ and (b) $\widetilde{{}^tdf}^*\tau'''\cong f^\natural(p^*\tau-\bb{P}R_T)$, where $\tau'''$ is a non-equivariant twisting over $Y_{T'}$. The ``difference'' of them can be written $l\in H^2_{T'}(Y_T)$  from Proposition \ref{classification of twistings}. Therefore, the following commutative diagram commutes.
%$$\begin{CD}
%K_{T'}^{\widetilde{{}^tdf}^*\tau'''+k}(Y_T) @>\otimes l>> K_{T'}^{f^\natural\tau+k}(Y_T) \\
%@V(\widetilde{{}^tdf})_!VV @V(\widetilde{{}^tdf})_!VV \\
%K_{T'}^{\tau'''+k}(Y_{T'}) @>\otimes\widetilde{{}^tdf}^*l>> K_{T'}^{\tau'''\otimes\widetilde{{}^tdf}^*l+k}(Y_{T'})
%\end{CD}$$
%By defining $\tau''$ as $\tau'''\otimes\widetilde{{}^tdf}^*l$, we obtain the conclusion.

(3), (4), (7) follow from the explicit description of these maps. Let us verify (3). Let us notice that $(\widetilde{{}^tdf})^*[R_{T'}]\cong f^\natural[R_T]$.

Let $[\{\ca{F}_{(x,\lambda)}\}_{(x,\lambda)\in Y_Y}]\in K^{p^*\tau-\bb{P}R_T+k}_T(Y_T)$. We can compute each compositions as follows.
$$f^\natural([\{\ca{F}_{(x,\lambda)}\}_{(x,\lambda)\in Y_Y}]\otimes[R_T])=
f^\natural([\{\ca{F}_{(x,\lambda)}\otimes id + \epsilon\otimes \bb{C}_\lambda\}_{(x,\lambda)\in Y_Y}])$$
$$=[\{f^\natural(\ca{F}_{(x,\lambda)})\otimes id + \epsilon\otimes \bb{C}_{{}^tdf(\lambda)}\}_{(x,\lambda)\in Y_Y}].$$
$$f^\natural([\{\ca{F}_{(x,\lambda)}\}_{(x,\lambda)\in Y_T}])\otimes(\widetilde{{}^tdf})^*[R_{T'}]=
([\{f^\natural(\ca{F}_{(x,\lambda)})\}_{(x,\lambda)\in Y_T}])\otimes(\widetilde{{}^tdf})^*[R_{T'}]$$
$$=[\{f^\natural(\ca{F}_{(x,\lambda)})\otimes id + \epsilon\otimes \bb{C}_{{}^tdf(\lambda)}\}_{(x,\lambda)\in Y_T}].$$
Therefore, the both sides coincide.

If we notice that 
$$(\widetilde{{}^tdf})_!([\{\ca{F}_{(x,\lambda)}\}_{(x,\lambda)\in Y_T}])=[\{\bigoplus_{{}^tdf(\lambda)=\mu}\ca{F}_{(x,\lambda)}\}_{(x,\mu)\in Y_{T'}}],$$
(4) and (7) follow from the same argument. 

\end{proof}

%%指数写像とMackey分解との関係
\subsubsection{A relationship between index maps and Mackey decomposition}

\begin{lem}%[naturality of K\"unneth isomorphisms with Mackey decompositions]
\label{naturality with Kunneth isomorphisms}
Let $M$, $Z$ be manifolds which $T$ acts on trivially and $\tau$ be a $T$-equivariant twisting over $M$. We suppose that $K^m(Z)$ is free for any $m\in\bb{Z}$. Let $p:Y_T\to X$ be the covering space of $X$ associated with $(T,\tau)$. The following commutative diagram holds.
$$
\begin{CD}
K^{p_1^*\tau+*}_T(M\times Z) @>\Phi_{M\times Z}>> K^{\tau+*}_T(M)\otimes K^*(Z) \\
@V{Mackey_{M\times Z,T,p^*\tau}}VV @V{Mackey_{M,T,\tau}\otimes id}VV \\
K^{p_1^*\tau'+*}(Y_T\times Z) @>{\Phi_{Y_T\times Z}}>> K^{\tau'+*}(Y_T)\otimes K^*(Z)
\end{CD}$$
\end{lem}
\begin{proof}
We verify the following commutative diagram.
$$\begin{CD}
K^{\tau'+*}(Y_T)\otimes K^*(Z) @>>> K_T^{\tau'+*}(Y_T)\otimes K^*(Z) @>\cong>> K_T^{p^*\tau-\bb{P}R+*}(Y_T)\otimes K^*(Z) @>({\otimes[R]})\otimes id>> \\
@V{p_1^*\otimes p_2^*}VV @V{p_1^*\otimes p_2^*}VV @V{p_1^*\otimes p_2^*}VV \\
K^{p_1^*\tau'+*}(Y_T\times Z) @>>> K_T^{p_1^*\tau'+*}(Y_T\times Z) @>\cong>> K_T^{p_1^*(p^*\tau-\bb{P}R)+*}(Y_T\times Z) @>{\otimes p_1^*[R]}>> 
\end{CD}$$
$$\begin{CD}
@>({\otimes[R]})\otimes id>> K_T^{p^*\tau+*}(Y_T)\otimes K^*(Z) @>{p_!\otimes id}>> K_T^{\tau+*}(M)\otimes K(Z) \\
@. @V{p_1^*\otimes p_2^*}VV @V{p_1^*\otimes p_2^*}VV \\
@>{\otimes p_1^*[R]}>> K_T^{p^*p_1^*\tau+*}(Y_T\times Z) @>{(p\times id)_!}>> K^{p_1^*\tau+*}_T(M\times Z) 
\end{CD}$$
where $({\otimes[R]})\otimes id(x\otimes y)=(x\otimes[R])\otimes y$.

Commutativity of the first square and the fourth one are trivial.

The second one follows by taking the isomorphism between $p_1^*\tau'$ and $p_1^*(p^*\tau-\bb{P}R)$ as pull back of the one between $\tau'$ and $p^*\tau-\bb{P}R$ along $p_1$. 

The third one is clear from that pull backs preserves $\otimes$.
\end{proof}

Let us verify the following theorem. Let us recall that we assume that $K_G^{\tau+n_0+1}(M)=0$, $dim(Z)=n$, $Z$ has the trivial tangent bundle and $K^m(Z)$ is free for any $m$ in order to define a family index map.

\begin{thm}%[naturality of index map with Mackey decomposition]
\label{naturality of index map with Mackey decomposition}
The following commutative diagram holds.
$$\begin{CD}
K_T^{p_1^*\tau+n+n_0}(M\times Z) @>ind_{M\times Z\to M}>> K_T^{\tau+n_0}(M) \\
@VMackey_{M\times Z,T,p_1^*\tau}VV @VMackey_{M,T\tau}VV \\
K^{p_1^*\tau'+n+n_0}(Y_T\times Z) @>ind_{Y_T\times Z\to Y_T}>> K^{\tau'+n_0}(Y_T)
\end{CD}$$
\end{thm}
\begin{proof}
The following commutative diagram holds.
$$\begin{CD}
K_T^{p_1^*\tau+n+n_0}(M\times Z) @>\Phi_{M\times Z}>> K^{\tau+n_0}_T(M)\otimes K^n(Z) @>id\otimes ind_Z>> K_T^{\tau+n_0}(M) \\
@VMackey_{M\times Z,T,p_1^*\tau}VV @VMackey_{M,T,\tau}\otimes idVV
@VMackey_{M,T,\tau}VV \\
K^{p_1^*\tau'+n+n_0}(Y_T\times Z) @>\Phi_{Y_T\times Z}>> K^{\tau'+n_0}(Y_T)\otimes K^{n}(Z)
@>id\otimes ind_Z>> K^{\tau'+n_0}(Y_T)
\end{CD}$$
The first square commutes from the above lemma.

The second one commutes clearly.

\end{proof}

%トーラス%
\subsection{For tori}\label{t.e.K for tori}
Let $T$ be a torus, and $\tau$ be a positive central extension of $LT$ and the associated $T$-equivariant twisting over $T$. In this section, we construct $f^\#$ for a local injection $f:S\to T$ so that it has a compatibility with $char(f)$ by use of the following lemma.

\begin{lem}\label{re-write of char(f)}%[re-write of $char(f)$]
We have a canonical isomorphism $char(T,\tau)\cong K(\Lambda_T^\tau/\kappa^\tau(\Pi_T))$.

Under this identification, if $T=S\times S'$, $f$ is the natural inclusion into the first factor, and $\tau=p_1^*\tau_1+p_2^*\tau_2$, $char(f)$ corresponds to the push-forward map $(({}^tdf)mod\Pi_T)_!$, where $\tau_1$ and $\tau_2$ are positive central extension of $LS$ and $LS^\perp$ respectively. $({}^tdf)mod\Pi_T$ is defined in Lemma \ref{realization of tdi_1 at covering space level}.
\end{lem}
\begin{proof}
The canonical isomorphism is defined by defining the values of the canonical generators as $char(T,\tau)\ni [\lambda]_T\mapsto \delta_{[\lambda ]_T}\in K(\Lambda_T^\tau/\kappa^\tau(\Pi_T))$, where $\delta_{[\lambda ]_T}\in K(\Lambda_T^\tau/\kappa^\tau(\Pi_T))$ is a trivial vector bundle whose rank is $1$ supported on the point $[\lambda]_T$.

As we explain in Lemma \ref{realization of tdi_1 at covering space level}, $\Lambda^\tau_T/\kappa^\tau(\Pi_T)\cong\Lambda^{\tau_1}_{S}/\kappa^{\tau_1}(\Pi_{S})\times\Lambda^{\tau_2}_{S'}/\kappa^{\tau_2}(\Pi_{S'})$ and $({}^tdf)mod\Pi_T$ is the natural projection onto the first factor.
By this correspondence and the definition of the push-forward map, the statement about induced homomorphism is clear.
\end{proof}
We use this identification and the functoriality of push-forward maps, pull backs, family index maps and Mackey decompositions in order to verify our theorem.

We will generalize the correspondence between $(({}^tdf)mod\Pi_T)_!$ and $char(f)$ in the following.

%直積%

\subsubsection{Direct product}\label{t.e.K for direct product of tori}
Let $T_1, T_2$ be tori, $T=T_1\times T_2$, $n_j$ be the dimension of $T_j$, $\tau_j$ be a positive central extension of $LT_j$, $i_j:T_j \hookrightarrow T$ be the natural inclusion into the $j$'th factor, $p_j:T \twoheadrightarrow T_j$ be the natural projection onto the $j$'th factor ($j=1,2$). We obtain a positive central extension $\tau=p_1^*\tau_1+p_2^*\tau_2$ of $LT$ from these data. Orientations of $T_1$ and $T_2$ naturally define the one of $T$.

We verify the following theorem.
\begin{thm}\label{compatibility of t.e.K(i_1) with char(i_1)}
If we define as
$$i_1^\#:=ind_{T_1\times T_2\to T_1}\circ i_1^\natural,$$
the following diagram commutes.
$$\begin{CD}
t.e.K(T,\tau) @>i_1^\#>> t.e.K(T_1,\tau_1) \\
@VM.d._TVV @VM.d._{T_1}VV \\
char(T,\tau) @>char(i_1)>> char(T_1,\tau_1)
\end{CD}$$
\end{thm}

Let us start from the lift of $({}^tdi_1)mod\Pi_T$ to the covering space.
\begin{lem}\label{realization of tdi_1 at covering space level}%[realization of ${}^tdi_1$ at covering space level]
\label{realization of tdi_1 at space level}
The following commutative diagram holds.
$$
\begin{CD}
\Lambda_T^\tau\times_{\Pi_T}\mathfrak{t} @>({}^tdi_1\times dp_1)mod\Pi_T>> \Lambda_{T_1}^{\tau_1}\times_{\Pi_{T_1}}\mathfrak{t}_1 \\
@V\pi VV @V\pi_1VV \\
\Lambda_T^\tau/\kappa^\tau(\Pi_T) @>({}^tdi_1)mod\Pi_T>> \Lambda_{T_1}^{\tau_1}/\kappa^{\tau_1}(\Pi_{T_1})
\end{CD}
$$
where the horizontal arrows are defined by
$$({}^tdi_1)mod\Pi_T([\lambda]):=[{}^tdi_1(\lambda)]$$
$$({}^tdi_1\times dp_1)mod\Pi_T([(\lambda,v)]):=[({}^tdi_1(\lambda),dp_1(v))]$$
and the vertical arrows
$$\pi:\Lambda^\tau_T\times _{\Pi_T}\fra{t}\to \Lambda^\tau_T/\kappa^\tau(\Pi_T)$$
$$\pi_1:\Lambda^{\tau_1}_{T_1}\times_{\Pi_{T_1}}\fra{t}_1\to \Lambda^{\tau_1}_{T_1}/\kappa^{\tau_1}(\Pi_{T_1})$$ are the trivial vector bundles mentioned in the proof of Corollary \ref{computation of K using a Mackey decomposition}.
\end{lem}
\begin{proof}
If we verify the well-definedness of $({}^tdi_1\times dp_1)mod\Pi_T$, commutativity is clear from the definition. Let $n\in\Pi_T$, $\lambda\in\Lambda_T^\tau$ and $v\in\mathfrak{t}$. We have to verify that ${}^tdi_1\times dp_1(\lambda+\kappa^\tau(n),v+n)$ is equivalent to ${}^tdi_1\times dp_1(\lambda,v)$ under the action of $\Pi_{T_1}$.
$${}^tdi_1\times dp_1(\lambda+\kappa^\tau(n),v+n)=({}^tdi_1(\lambda+\kappa^\tau(n)),dp_1(v+n))$$
$$=({}^tdi_1(\lambda+\kappa^\tau(di_1\circ dp_1(n)+di_2\circ dp_2(n))),dp_1(v)+dp_1(n))$$
$$=({}^tdi_1(\lambda)+{}^tdi_1\circ\kappa^\tau\circ di_1(dp_1(n))+{}^tdi_1\circ\kappa^\tau\circ di_2(dp_2(n)),dp_1(v)+dp_1(n))$$
$$=({}^tdi_1(\lambda)+\kappa^{\tau_1}(dp_1(n)),dp_1(v)+dp_1(n)),$$
where ${}^tdp_1\circ\kappa^\tau\circ di_2=0$ from Lemma \ref{product formula for kappa}.
\end{proof}
Lemma \ref{functoriality of push-forward} implies the following.
\begin{pro}
The following commutative diagram holds.
$$\begin{CD}
K^n(\Lambda_T^\tau\times_{\Pi_T}\mathfrak{t}) @>(({}^tdi_1\times dp_1)\mathit{mod}\Pi_T)_!>> K^{n_1}(\Lambda_{T_1}^{\tau_1}\times_{\Pi_{T_1}}\mathfrak{t}_1) \\
@V\pi_!VV @V\pi_{1!}VV \\
K(\Lambda_T^\tau/\kappa^\tau(\Pi_T)) @>(({}^tdi_1)mod\Pi_T)_!>> K(\Lambda_{T_1}^{\tau_1}/\kappa^{\tau_1}(\Pi_{T_1}))
\end{CD}$$

\end{pro}
The lift $({}^tdi_1\times dp_1)\mathit{mod}\Pi_T$ of $({}^tdi_1)mod\Pi_T$ can be written as the composition of the following sequence.
\begin{lem}\label{realization of tdi_1 times dp_1mod Pi}
$({}^tdi_1\times dp_1)mod\Pi_T$ can be written as the composition of the following sequence
$$\Lambda_T^\tau\times_{\Pi_T}\mathfrak{t}
\xrightarrow{({}^tdi_1\times id)mod\Pi_T}
\Lambda_{T_1}^{\tau_1}\times_{\Pi_T}\mathfrak{t}
\xrightarrow{(id\times dp_1)mod\Pi_T}
\Lambda_{T_1}^{\tau_1}\times_{\Pi_{T_1}}\mathfrak{t}_1,$$
where $\Pi_T$ acts on $\Lambda^{\tau_1}_{T_1}$ through $dp_1:\Pi_T\to \Pi_{T_1}$ and $\kappa^{\tau_1}$, and
$$({}^tdi_1\times id)mod\Pi_T([(\lambda,v)]):=[({}^tdi_1(\lambda),v)]$$
$$(id\times dp_1)mod\Pi_T([(\mu,v)]):=[(\mu,dp_1(v))].$$
\end{lem}
\begin{proof}
Well-definedness of these maps can be verified with the same way in the proof of Lemma \ref{realization of tdi_1 at covering space level}. The equality
$$({}^tdi_1 \times dp_1)mod\Pi_T=(id\times dp_1)mod\Pi_T\circ ({}^tdi_1\times id)mod\Pi_T$$
follows from that ${}^tdi_1 \times dp_1=(id\times dp_1)\circ({}^tdi_1\times id)$
\end{proof}

Through the diffeomorphism $\Lambda_{T_1}^{\tau_1}\times_{\Pi_T}\mathfrak{t}\cong
(\Lambda_{T_1}^{\tau_1}\times_{\Pi_{T_1}}\mathfrak{t}_1)\times T_2$, 
$(id\times dp_1)mod\Pi_T$ can be thought as the projection onto the first factor.

The following lemma is clear from Theorem \ref{naturality for f^natural}, \ref{naturality of index map with Mackey decomposition}. Let us recall that $K^{\tau_1+\dim(T_1)+1}_{T_1}(T_1)=0$ and the tangent bundle of $T_2$ is trivial. We can trivialize $TT_2$ via the left translation.
%\begin{lem}
%The push-forward map along ${}^tdi_1\times id\mathit{mod}\Pi$ corresponds to the restriction of group action, which follows from Theorem \ref{naturality for f^natural}.
%%\end{lem}
%\begin{lem}
%The second one corresponds to the family index, which follows from Theorem \ref{naturality of index map with Mackey decomposition}.
%\end{lem}

\begin{lem}\label{naturality of ind.}
The following commutative diagram holds.
$$\begin{CD}
K^{\tau+n}_{T}(T) @>i_1^\natural>> K^{p_1^*\tau_1+n}_{T_1}(T_1\times T_2)
@>ind_{T_1\times T_2\to T_1}>> K^{\tau_1+n_1}_{T_1}(T_1) \\
@VMackey_{T,T,\tau}VV @VMackey_{T_1\times T_2,T_1,p_1^*\tau}VV @VMackey_{T_1,T_1,\tau_1}VV \\
K^{n}(Y_{T}) @>(\widetilde{{}^tdi_1})_!>> K^{n}(Y_{T_1}\times T_2) @>ind_{Y_{T_1}\times T_2\to Y_{T_1}}>> K^{n_1}(Y_{T_1}) 
\end{CD}$$
where $n:=n_1+n_2$, $Y_T:=\Lambda^\tau_T\times_{\Pi_T}\fra{t}$ and $Y_{T_1}:=\Lambda^\tau_{T_1}\times_{\Pi_{T_1}}\fra{t}_1$.
\end{lem}

The decomposition of $({}^tdi_1\times dp_1)mod\Pi_T$ and Theorem \ref{explicit formula of push-forward with discrete fiber} imply the following lemma.
\begin{lem}\label{decomposition of char(i_1)}
The following commutative diagram holds.
$$\xymatrix{
K^n(Y_T) \ar[r]^{(\widetilde{{}^tdi_1})_!} \ar[d]^{\pi_!}  & K^{n}(Y_{T_1}\times T_2) \ar[rr]^{ind_{Y_{T_1}\times T_2\to Y_{T_1}}} && K^{n_1}(Y_{T_1}) \ar[d]^{\pi_{1!}} \\
K(\Lambda_T^\tau/\kappa^\tau(\Pi_T)) \ar[rrr]^{(({}^tdi_1)mod\Pi_T)_!} & && K(\Lambda_{T_1}^{\tau_1}/\kappa^{\tau_1}(\Pi_{T_1}))
}$$
\end{lem}

Let us notice that the upper sequence in the above lemma is the lower one in Lemma \ref{naturality of ind.}.

Let us verify the theorem.\\\\
{\it Proof of Theorem \ref{compatibility of t.e.K(i_1) with char(i_1)}.} 
Let us recall Lemma \ref{re-write of char(f)}, that is, the following commutative diagram holds.
$$\begin{CD}
K(\Lambda_T^\tau/\kappa^\tau(\Pi_T)) @>{(({}^tdi_1)mod\Pi_T)_!}>> K(\Lambda_{T_1}^{\tau_1}/\kappa^{\tau_1}(\Pi_{T_1}))\\
@V\cong VV @V\cong VV \\
char(T,\tau) @>char(i_1)>> char(T_1,\tau_1)
\end{CD}$$
Combining with Lemma \ref{decomposition of char(i_1)}, the following commutative diagram holds.
$$\xymatrix{
K^n(Y_T) \ar[r]^{(\widetilde{{}^tdi_1})_!} \ar[d]^{\pi_!}  & K^{n}(Y_{T_1}\times T_2) \ar[rr]^{ind_{Y_{T_1}\times T_2\to Y_{T_1}}} && K^{n_1}(Y_{T_1}) \ar[d]^{\pi_{1!}} \\
K(\Lambda_T^\tau/\kappa^\tau(\Pi_T)) \ar[rrr]^{(({}^tdi_1)mod\Pi_T)_!} \ar[d]^\cong & && K(\Lambda_{T_1}^{\tau_1}/\kappa^{\tau_1}(\Pi_{T_1})) \ar[d]^\cong \\
char(T,\tau) \ar[rrr]^{char(i_1)} &&& char(T_1,\tau_1)
}$$

Let us consider the following diagram.
$$\xymatrix{
K_T^{\tau+n}(T) \ar[rr]^{i_1^\natural} \ar[d]_{Mackey_{T,T,\tau}} \ar@{}[rrd]|{(1)} && K^{p_1^*\tau_1+n}_{T_1}(T) \ar[rr]^{ind_{T_1\times T_2\to T_1}} \ar[d]|{Mackey_{T,T_1,p_1^*\tau_1}} \ar@{}[rrd]|{(2)} & & K^{\tau_1+n_1}_{T_1}(T_1) \ar[d]^{Mackey_{T_1,T_1,\tau_1}} \\
K^n(Y_T) \ar[rr]^{(\widetilde{{}^tdi_1})_!} \ar[d]_{\pi_!} \ar@{}[rrrrd]|{(3)} && K^n(Y_{T_1}\times T_2) \ar[rr]^{ind_{Y_{T_1}\times T_2\to Y_{T_1}}} && K^{n_1}(Y_{T_1}) \ar[d]^{\pi_{1!}} \\
char(T,\tau) \ar[rrrr]^{char(i_1)} &&&& char(T_1,\tau_1)
}$$
Let us recall that $M.d._T=\pi_!\circ Mackey_{T,T,\tau}$ and $M.d._{T_1}=\pi_{1!}\circ Mackey_{T_1,T_1,\tau_1}$.

 (1) commutes from Theorem \ref{naturality for f^natural}. (2) commutes from Theorem \ref{naturality of index map with Mackey decomposition}. (3) commutes from the above commutative diagram.
\begin{flushright} $\Box$ \end{flushright}

%有限被覆%

\subsubsection{Finite covering}\label{t.e.K for finite cover of tori}

Let $q:T'\to T$ be a finite covering of an $n$ dimensional torus and $\tau$ be a positive central extension of $LT$ and the associated $T$-equivariant twisting over $T$. In this section, we verify the following theorem.
\begin{thm}\label{main theorem for K for finite cover}
If we define
$$q^\#:=q^*\circ q^\natural$$
the following commutative diagram holds.
$$\begin{CD}
t.e.K(T,\tau) @>q^\#>> t.e.K(T',q^*\tau) \\
@VM.d._TVV @VM.d._{T'}VV \\
char(T,\tau) @>char(q)>> char(T',q^*\tau)
\end{CD}$$
\end{thm}
\begin{rmk}
$q^\#$ is the composition of 
$$K^{\tau+\dim(T)}_T(T)\xrightarrow{q^\natural} K^{q^\natural\tau+\dim(T)}_{T'}(T) \xrightarrow{q^*} K^{q^*\tau+\dim(T')}_{T'}(T').$$
It is written as simply $q^*$ in \cite{FHT1}. In the terminology of groupoids, $q$ determines a functor $(q,q):T'//T' \to T//T$ and $q^*$ is the pull back along the functor $(q,q)$.
\end{rmk}
Let us compute $K^{q^\natural\tau+\dim(T)}_{T'}(T)$ by use of a Mackey decomposition.
\begin{thm}\label{computation of K^tau+k_T'(T')}
$$K^{q^\natural\tau+k}_{T'}(T)\cong \begin{cases}
\bb{Z}[\Lambda^{q^*\tau}_{T'}/{}^tdq(\kappa^\tau(\Pi_T))] & k=\dim(T) \mod2 \\
0 & k=\dim(T)+1 \mod2 \end{cases}$$
\end{thm}
\begin{proof}
From the definition of $q^\natural\tau$, if one travels on $T$ along $n\in\Pi_T$, the character $\lambda\in \Lambda^{q^\natural\tau}_{T'}$ changes to $\lambda+{}^tdq(\kappa^\tau(n))$. Therefore, the covering space constructed in Section \ref{section about Mackey decomposition} is 
$$p^{(')}:\Lambda^{q^*\tau}_{T'}\times_{\Pi_T}\fra{t}\to T,$$
 where $\Pi_T$ acts on $\Lambda^{q^*\tau}_{T'}$ via ${}^tdq\circ\kappa^\tau$ and $\Lambda_{T'}\curvearrowright\Lambda^{q^*\tau}_{T'}$. It has a structure of a trivial vector bundle 
$$\pi^{(')}:\Lambda^{q^*\tau}_{T'}\times_{\Pi_T}\fra{t}\to \Lambda^{q^*\tau}_{T'}/{}^tdq(\kappa^\tau(\Pi_T)).$$
Through Thom isomorphism $\pi^{(')}_!$, we obtain the conclusion.
\end{proof}
\begin{lem}\label{decomposition of char(q)}
Let us consider the maps 
$$\Lambda_T^{\tau}/\kappa^{\tau}(\Pi_T)
\xrightarrow{({}^tdq)mod\Pi_T}
\Lambda_{T'}^{q^*\tau}/{}^tdq(\kappa^\tau(\Pi_T))
\xleftarrow{r}
\Lambda_{T'}^{q^*\tau}/\kappa^{q^*\tau}(\Pi_{T'}),$$
where 
$$({}^tdq)mod{\Pi_T}([\lambda]_T):=[{}^tdq(\lambda)]$$
$$r([\mu]_{T'}):=[\mu].$$
$[\mu]$ is the $\Pi_T$-orbit of $\mu$ in $\Lambda^{q^*\tau}_{T'}$.

Then $char(q)$ corresponds to $r^*\circ(({}^tdq)mod\Pi_T)_!$. That is, the following commutative diagram holds.
$$\xymatrix{
K(\Lambda^\tau_T/\kappa^\tau(\Pi_T)) \ar[rr]^{(({}^tdq)mod\Pi_T)_!} \ar[d]^\cong && K(\Lambda_{T'}^{q^*\tau}/{}^tdq(\kappa^\tau(\Pi_T)) \ar[r]^{r^*} & K(\Lambda^{q^*\tau}_{T'}/\kappa^{q^*\tau}(\Pi_{T'})) \ar[d]^\cong \\
char(T,\tau) \ar[rrr]^{char(q)} & & & char(T',q^*\tau)}$$
\end{lem}
\begin{proof}
Since $({}^tdq)mod\Pi_T$ is an injection, for $\lambda\in\Lambda^\tau_T$,
$$(({}^tdq)mod\Pi_T)_!(\delta_{[\lambda]_T})=\delta_{[{}^tdq(\lambda)]}.$$

On the other hand, for $\lambda'\in\Lambda^{q^*\tau}_{T'}$,
$$r^*(\delta_{[\lambda']})=
\sum_{n\in\Pi_T/dq(\Pi_{T'})}\delta_{[\lambda'+{}^tdq(\kappa^\tau(n))]_{T'}}.$$
If we notice that 
$$r^{-1}(\{[\lambda']\})=\{[\lambda'+{}^tdq(\kappa^\tau(n))]_{T'}|n\in \Pi_T/dq(\Pi_{T'})\},$$
the above computation is clear from the definition of pull back of vector bundles.

As we compute at Lemma \ref{char formula for finite cover}, the above commutative diagram holds.
\end{proof}
%From Theorem \ref{naturality for f^natural}, $({}^tdq)_!$ corresponds to $q^\natural$. That is the following diagram commutes:
%$$\begin{CD}
%K_T^{\tau+dim(T)}(T) @>Mackey_{T,T,\tau}>> K^{\tau'+dim(T)}(Y_T) \\
%@V({}^tdq)_!VV @Vq^\natural VV \\
%K_{T'}^{q^\natural+dim(T)}(T) @>Mackey_{T,T',q^\natural \tau}>> K^{\tau''+dim(T)}(Y_{T'})
%\end{CD}$$

Let us lift $r$ to the covering spaces. We can verify the following lemma with the same way in Lemma \ref{realization of tdi_1 at space level}.
\begin{lem}\label{geometrical realization of r}
The following commutative diagram holds.
$$\begin{CD}
\Lambda_{T'}^{q^*\tau}\times_{\Pi_T}\mathfrak{t} @<(id\times dq)\mathit{mod}\Pi_{T'}<< \Lambda_{T'}^{q^*\tau}\times_{\Pi_{T'}}\mathfrak{t}' \\
@V\pi^{(')}VV @V\pi'VV \\
\Lambda_{T'}^{q^*\tau}/{}^tdq(\kappa^\tau(\Pi_T))  @<r<< \Lambda_{T'}^{q^*\tau}/\kappa^{q^*\tau}(\Pi_{T'})
\end{CD}$$
where $\pi':\Lambda_{T'}^{q^*\tau}\times_{\Pi_{T'}}\mathfrak{t}' \to\Lambda_{T'}^{q^*\tau}/\kappa^{q^*\tau}(\Pi_{T'})$ is the trivial vector bundle defined in the proof of Theorem \ref{computation of K^tau+k_T'(T')}, and 
$(id\times dq)mod\Pi_{T'}([\lambda,v]):=[(\lambda,dq(v))].$

From the above diagram, the vector bundle $\Lambda^{q^*\tau}_{T'}\times_{\Pi_{T'}}\fra{t}'$ is the pull back of $\Lambda^{q^*\tau}_{T'}\times_{\Pi_T}\fra{t}$ along $r$.
\end{lem}
%\begin{proof}
%What to verify is only $\Pi'$ equivariance of $id\times dq:\Lambda'^{q^*\tau}\times\mathfrak{t}'\to\Lambda'^{q^*\tau}\times\mathfrak{t}$. One can prove 
%\end{proof}

Let us recall that push-forward is given by the inverse of the tensor product with Thom class. Combining the above lemma with naturality of Thom class, we obtain the following.
\begin{pro}\label{geometrical realization of r^*}
The following commutative diagram holds.
$$\begin{CD}
K^n(\Lambda_{T'}^{q^*\tau}\times_{\Pi_T}\mathfrak{t}) @>((id\times dq)\mathit{mod}\Pi_{T'})^*>> K^n(\Lambda_{T'}^{q^*\tau}\times_{\Pi_{T'}}\mathfrak{t}') \\
 @V{\pi^{(')}_!}VV @V{\pi'_!}VV \\
K(\Lambda_{T'}^{q^*\tau}/{}^tdq(\kappa^\tau(\Pi_T))) @>r^*>> K(\Lambda_{T'}^{q^*\tau}/\kappa^{q^*\tau}(\Pi_{T'}))
\end{CD}$$
where $(id\times dq)mod\Pi_{T'}([(\mu,w)]):=[(\mu,dq(w))]$.
\end{pro}
Let us lift $({}^tdq)mod\Pi_T$ to the covering spaces.

The following lemma can be verified with the same way in Lemma \ref{realization of tdi_1 at space level}.
\begin{lem}\label{geometrical realization of tdqmod Pi}
The following diagram commutes.
$$\begin{CD}
\Lambda^\tau_T\times_{\Pi_T}\fra{t} @>{({}^tdq\times id)mod\Pi_T}>> \Lambda_{T'}^{q^*\tau}\times_{\Pi_T}\fra{t} \\
@V\pi VV @V\pi^{(')}VV \\
\Lambda^\tau_T/\kappa^\tau({\Pi_T}) @>{({}^tdq)mod\Pi_T}>> \Lambda_{T'}^{q^*\tau}/{}^tdq(\kappa^{\tau}(\Pi_T))
\end{CD}$$
\end{lem}
From Lemma \ref{functoriality of push-forward}, we can verify the following.
\begin{pro}\label{q^natural at covering space level}
The following commutative diagram holds.
$$\begin{CD}
K^n(\Lambda^\tau_T\times_{\Pi_T}\fra{t}) @>{(({}^tdq\times id)mod\Pi_T)_!}>> K^n(\Lambda_{T'}^{q^*\tau}\times_{\Pi_T}\fra{t}) \\
@V\pi_!VV @V\pi^{(')}_!VV \\
K(\Lambda^\tau_T/\kappa^\tau({\Pi_T})) @>{(({}^tdq)mod\Pi_T)_!}>> K(\Lambda_{T'}^{q^*\tau}/{}^tdq(\kappa^{\tau}(\Pi_T)))
\end{CD}$$
\end{pro}\label{q^* at covering space level}
Let us verify that $(id\times dq)mod\Pi_{T'}$ is the lift of $q$, that is, if we use the notation in Theorem \ref{naturality for F^*}, $(id\times dq)mod\Pi_{T'}$ can be written as $\widetilde{q}$.

\begin{lem}\label{geometrical realization of tdq}
The following commutative diagram holds.
$$\begin{CD}
\Lambda^{q^*\tau}_{T'}\times \fra{t}' @>id\times dq>> \Lambda^{q^*\tau}_{T'}\times \fra{t}\\
@V\overline{p'}VV @V\overline{p^{(')}}VV \\
\Lambda_{T'}^{q^*\tau}\times_{\Pi_{T'}}\mathfrak{t}' @>(id\times dq)mod\Pi_{T'}>> \Lambda_{T'}^{q^*\tau}\times_{\Pi_T}\mathfrak{t} \\
@Vp'VV @Vp^{(')}VV \\
T' @>q>> T
\end{CD}$$
where $\overline{p'}$ and $\overline{p^{(')}}$ are the natural projections.
\end{lem}

\begin{proof}
That $p^{(')}\circ\overline{p^{(')}}\circ(id\times dq)=q\circ p'\circ\overline{p'}$ follows from the property of the tangent map $dq$, and that $\bar{p^{(')}}\circ(id\times dq)=(id\times dq)mod\Pi_T\circ\overline{p'}$ follows from the definition of $(id\times dq)mod\Pi_{T'}$. Therefore, that $p^{(')}\circ (id\times dq)mod\Pi_{T'}=q\circ p'$ holds.

\end{proof}
This lemma and Theorem \ref{naturality for F^*} imply the following. Let $n$ be the dimension of $T$.
\begin{pro}\label{realization of ((id times dq) mod Pi')^*}
$((id\times dq)\mathit{mod}\Pi_{T'})^*$ corresponds to $q^*$ under the Mackey decomposition. That is, the following commutative diagram holds.
$$\begin{CD}
K^{q^\natural\tau+n}_{T'}(T) @>q^*>> K^{q^*\tau+n}_{T'}(T') \\
@VMackey_{T,T',q^\natural\tau}VV @VVMackey_{T',T',q^*\tau}V \\
K^n(\Lambda^{q^*\tau}_{T'}\times_{\Pi_T}\fra{t}) @>((id\times dq)\mathit{mod}\Pi_{T'})^*>> K^n(\Lambda^{q^*\tau}_{T'}\times_{\Pi_{T'}}\fra{t}')
\end{CD}$$
\end{pro}

Let us verify the theorem.\\\\
{\it Proof of Theorem \ref{main theorem for K for finite cover}.} Let us consider the following diagram.
$$\xymatrix{
K^{\tau+n}_T(T) \ar[rr]^{q^\natural} \ar[d]^{Mackey_{T,T,\tau}} \ar@{}[drr]|{(1)} &&
 K^{q^\natural\tau+n}_{T'}(T) \ar[rr]^{q^*} \ar[d]^{Mackey_{T,T',q^\natural\tau}} \ar@{}[drr]|{(2)}&&
 K^{q^*\tau+n}_{T'}(T') \ar[d]^{Mackey_{T',T',q^*\tau}} \\
K^n(\Lambda^{\tau}_{T}\times_{\Pi_T}\fra{t}) \ar[rr]^{(({}^tdq\times id)mod\Pi_T)_!} \ar[d]^{\pi_!} \ar@{}[drr]|{(3)}&&
K^n(\Lambda^{q^*\tau}_{T'}\times_{\Pi_T}\fra{t}) \ar[rr]^{((id\times dq)mod\Pi_{T'})^*} \ar[d]^{\pi^{(')}_!}  \ar@{}[drr]|{(4)}&&
K^n(\Lambda^{q^*\tau}_{T'}\times_{\Pi_{T'}}\fra{t}') \ar[d]^{\pi'_!} \\
K(\Lambda^\tau_T/\kappa^\tau(\Pi_T)) \ar[rr]^{(({}^tdq)mod\Pi_T)_!} &&
 K(\Lambda^{q^*\tau}_{T'}/{}^tdq(\kappa^\tau(\Pi_T))) \ar[rr]^{r^*} &&
 K(\Lambda^{q^*\tau}_{T'}/\kappa^{q^*\tau}(\Pi_{T'}))}$$

(1) commutes from Theorem \ref{naturality for f^natural}. (2) commutes from Lemma \ref{geometrical realization of tdq} and \ref{naturality for F^*}. (3) commutes from Proposition \ref{q^natural at covering space level}. (4) commutes from Proposition \ref{geometrical realization of r^*}.

From Lemma \ref{decomposition of char(q)}, we obtain the conclusion.

\begin{flushright} $\Box$ \end{flushright}

%一般の場合%
\subsubsection{Local injection}\label{t.e.K for local injection}
Let $S$ and $T$ be tori, $f:S\to T$ be a local injection and $\tau$ be a positive central extension. Let us recall that $t.e.K(T,\tau)$ is defined by $K_T^{\tau+dim(T)}(T)$.

From Theorem \ref{decomposition of local injection}, $f$ can be decomposed as follows.
$$S\xrightarrow{q} S/\ker(f) \xrightarrow{i_1} S/\ker(f)\times S^\perp \xrightarrow{f\cdot j} T,$$
where $q$ is the natural finite covering, $i_1$ is the natural inclusion into the first factor, and $j:S^\perp\to T$ is the natural inclusion.

Let us verify the following theorem.

\begin{thm}\label{main theorem for K for tori in Section 4}
If we define as
$$f^\#:=q^\#\circ i_1^\#\circ (f\cdot j)^\#,$$
the following commutative diagram holds.
$$\begin{CD}
t.e.K(T,\tau) @>f^\#>> t.e.K(S,f^*\tau) \\
@VM.d._TVV @VM.d._SVV \\
char(T,\tau) @>char(f)>> char(S,f^*\tau)
\end{CD}$$
\end{thm}
\begin{proof}
The following commutative diagram holds from Theorem \ref{compatibility of t.e.K(i_1) with char(i_1)} and Theorem \ref{main theorem for K for finite cover}.
$$\begin{CD}
t.e.K(T,\tau) @>(f\cdot j)^\#>> t.e.K(S/\ker(f)\times S^\perp,(f\cdot j)^*\tau) @>i_1^\#>> \\
@VM.d._TVV @VM.d._{S/\ker(f)\times S^\perp}VV \\
char(T,\tau) @>char(f\cdot j)>> char(S/\ker(f)\times S^\perp,(f\cdot j)^*\tau) @>char(i_1)>>
\end{CD}$$
$$\begin{CD}
@>i_1^\#>> t.e.K(S/\ker(f),i_1^*(f\cdot j)^*\tau) @>q^\#>> t.e.K(S,f^*\tau) \\
@. @VM.d._{S/\ker(f)}VV @VM.d._SVV \\
@>char(i_1)>> char(S/\ker(f),i_1^*(f\cdot j)^*\tau) @>char(q)>> char(S,f^*\tau) 
\end{CD}$$
Form Theorem \ref{partial functoriality}, 
$$char(f)=char(q) \circ char(i_1) \circ char(f\cdot j),$$
and by the definition,
$$f^\#=q^\#\circ i_1^\#\circ (f\cdot j)^\#.$$
Therefore, we obtain the conclusion.
\end{proof}

%一般のG%

\subsection{For compact connected $G$ with torsion-free $\pi_1$}\label{section of t.e.K for G}
Let us extend the above construction to compact connected Lie groups with torsion-free $\pi_1$.

Let $G$ and $H$ be compact connected Lie groups with torsion-free $\pi_1$, $\tau$ be a positive central extension of $LG$ and the associated $G$-equivariant twisting over $G$, $f$ satisfy the decomposable condition, and $S$ and $T$ be chosen maximal tori of $H$ and $G$ respectively such that $f(S)\subseteq T$. That is, the following commutative diagram holds.
$$
\begin{CD}
H @>f>> G \\
@AiAA @AkAA \\
S @>f>> T
\end{CD}
$$
In this section, we use the same character for the restriction of homomorphisms to subgroup or induced map to the quotient group. For example, $f:H\to G$ determines group homomorphisms $f:H/\ker(f)\to G$ and $f:S\to T$.

Let us recall the result in \cite{FHT1}.

\begin{thm}[\cite{FHT1} Theorem 4.27]\label{naturality by FHT}
If $f$ is an injection and $rank(G)=rank(H)$, the following commutative diagram holds.
$$
\begin{CD}
K^{\tau+rank(G)}_G(G) @>f^*\circ f^\natural>> K_H^{f^*\tau+rank(H)} \\
@V{M.d._G} VV @V{M.d._H}VV \\
char(G,\tau) @>char(f)>> char(H,f^*\tau)
\end{CD}
$$
\end{thm}
This theorem tells us that $f^\#$ has been defined and holds naturality with $char(f)$. Let us extend the result for more general cases. That is, we verify the following theorem.
\begin{thm}\label{main theorem for K for general Lie group}
We can define $f^\#$ for $f$ satisfying the decomposable condition and the following commutative diagram holds.
$$\begin{CD}
t.e.K(G,\tau) @>f^\#>> t.e.K(H,f^*\tau) \\
@VM.d.GVV @VM.d._HVV \\
char(G,\tau) @>char(f)>> char(H,f^*\tau) \\
\end{CD}$$
\end{thm}

%局所単射の分解
\subsubsection{The decomposition of a homomorphism satisfying the decomposable condition}\label{decomposition of the local injection in the general case}
As we explain in the proof of Theorem \ref{well-def. of char(f) for general Lie group}, $f:H\to G$ can be written as the composition of the following series
$$H\xrightarrow{q}H/\ker(f)\xrightarrow{i_1}H/\ker(f)\times S^\perp\xrightarrow{f\cdot j}G$$
if $f$ satisfies the decomposable condition.
By restricting this sequence to maximal tori, we obtain the following commutative diagram
$$\xymatrix{
H \ar[r]^q & H/\ker(f) \ar[r]^{i_1} & H/\ker(f)\times S^\perp \ar[rr]^{f\cdot j} && G & \\
S\ar[r]^q \ar[u]^i & S/\ker(f) \ar[u]^i \ar[r]^{i_1} & S/\ker(f)\times S^\perp \ar[u]|{i\times id} \ar[rr]^{f\cdot j} && T \ar[u]^k & S^\perp \ar[l]_j \ar[lu]_j }$$

Moreover, by the definition of $S^\perp$, $(f\cdot j)^*\tau$ can be written as $$p_1^*\tau_{H/\ker(f)}+p_2^*\tau_{S^\perp},$$
 where $\tau_{H/\ker(f)}=((f \cdot j)\circ i_1)^*\tau$ and $\tau_{S^\perp}=j^*\tau$.

Let us notice that the restriction of $q$ and $f\cdot j$ to maximal tori are finite coverings, and the one of $i_1$ is the natural inclusion into the first factor of the direct product. By use of theorems of the previous section, we construct $f^\#$ for $f$.

%We recall the local condition at Definition \ref{local condition}. That is, the orthogonal completion $\mathfrak{s}^\perp$ and $df(\mathfrak{h})$ are commutative. This condition implies that $df(\mathfrak{h})\oplus\mathfrak{s}^\perp$ has a structure of a sub Lie algebra of $\mathfrak{g}$. Therefore, we can define a new Lie group $H\times S^\perp$ with the same way in Proposition \ref{decomposition of injection for tori}, we can verify that $S^\perp\subseteq T$ is a torus. Let $j:S^\perp\hookrightarrow G$ be the inclusion.
%\begin{lem}\label{decomposition of local injection for general G}
%We can decompose $f$ as follows;
%$$H\xrightarrow{q}H/\ker(f)\xrightarrow{i_1}H/\ker(f)\times S^\perp\xrightarrow{f\cdot j}G$$
%where $q$ is the natural finite cover, $f\cdot j$ is defined by 
%$f\cdot j(h,s):=f(h)\cdot j(s)$ and ``$\cdot$'' means the multiplication in $G$.
%\end{lem}
%This lemma means that if $f$ is an injection, we have the following commutative diagram;

%We generalize the argument in the case of direct product of tori or finite cover of a torus. If we can define $q^\#, (f\cdot j)^\#$ and $i_1^\#$ which has the compatibility with $char(q), char(f\cdot j)$ and $char(i_1)$ respectively the following is well-defined and the following theorem is clear.
%\begin{dfn}\label{def.  of RL(f) for general Lie group}
%$$f^\#:=(f\cdot j)^\#\circ i_1^\# \circ q^\#$$
%\end{dfn}

%\begin{thm}\label{main theorem for general Lie group}
%$f^\#$ has a compatibility with $char(f)$.
%\end{thm}

%%%%直積のケース
\subsubsection{``Direct product''}\label{t.e.K for direct product for general G}
Let $H$ be a compact connected Lie group with torsion-free $\pi_1$, $S$ be a maximal torus of $H$, $S^\perp$ be a torus, and $\tau_H$ and $\tau_{S^\perp}$ be positive central extensions of $LH$ and $LS^\perp$ respectively. $i_1:H\to H\times S^\perp$, $p_1:H\times S^\perp \to H$ and $p_2:H\times S^\perp \to S^\perp$ are as usual. Then, $\tau:=p_1^*\tau_H+p_2^*\tau_{S^\perp}$ is a positive central extension of $L(H\times S^\perp)$ and the associated $H\times S^\perp$-equivariant twisting over $H\times S^\perp$.

Let us define 
$$i_1^\#:t.e.K(H\times S^\perp,\tau)\to t.e.K(H,\tau_H)$$
 by use of the family index map like the case of tori.
\begin{thm}\label{main theorem for K for general group in the case of direct products}
If we define as 
$$i_1^\#:=ind_{H\times S^\perp\to H}\circ i_1^\natural,$$
the following commutative diagram holds.
$$\begin{CD}
t.e.K(H\times S^\perp,\tau) @>i_1^\#>> t.e.K(H,\tau_H) \\
@VM.d._{H\times S^\perp}VV @VM.d._HVV \\
char(H\times S^\perp\tau) @>char(i_1)>> char(H,\tau_H) 
\end{CD}$$
\end{thm}

Let us start from the reduction to maximal tori. $n:=rank(H)$ and $m:=\dim(S^\perp)$.
\begin{pro}
The following commutative diagram holds.
$$\begin{CD}
t.e.K(H\times S^\perp,\tau) @>(i\times id)^\#>> t.e.K(S\times S^\perp,(i\times id)^*\tau) \\
@Vi_1^\#VV @Vi_1^\#VV \\
t.e.K(H,\tau_H) @>i^\#>> t.e.K(S,i^*\tau_H)
\end{CD}$$
where $(i\times id)^\#$ and $i^\#$ have been defined in Theorem \ref{naturality by FHT}.
\end{pro}
\begin{proof}
{\small
$$\xymatrix{
K^{\tau+n+m}_{H\times S^\perp}(H\times S^\perp) \ar[rr]^{(i\times id)^\natural} \ar@{}[rrd]|{(1)} \ar[d]^{i_1^\natural} && K^{(i\times id)^\natural\tau+n+m}_{S\times S^\perp}(H\times S^\perp) \ar[rr]^{(i\times id)^*} \ar@{}[rrd]|{(2)} \ar[d]^{i_1^\natural} && K^{(i\times id)^*\tau+n+m}_{S\times S^\perp}(S\times S^\perp) \ar[d]^{i_1^\natural} \\
K^{p_1^*\tau_H+n+m}_H(H\times S^\perp) \ar[rr]^{i^\natural} \ar[d]^{\Phi_{H\times S^\perp,H,\tau_H}}
 \ar@{}[rrd]|{(3)} && K^{p_1^*i^\natural\tau_H+n+m}_S(H\times S^\perp) \ar[rr]^{(i\times id)^*} \ar[d]^{\Phi_{H\times S^\perp,S,i^\natural\tau_H}} \ar@{}[rrd]|{(4)} && K^{i^*\tau_H+n+m}_S(S\times S^\perp) \ar[d]^{\Phi_{S\times S^\perp,S,i^*\tau_H}} \\
K_H^{\tau_H+n}(H)\otimes K^m(S^\perp) \ar[rr]^{i^\natural\otimes id} \ar[d]^{id\otimes ind_{S^\perp}} \ar@{}[rrd]|{(5)} && K_S^{i^\natural\tau_H+n}(H)\otimes K^m(S^\perp) \ar[rr]^{i^*\otimes id} \ar[d]^{id\otimes ind_{S^\perp}} \ar@{}[rrd]|{(6)} && K_S^{i^*\tau_H+n}(S)\otimes K^m(S^\perp) \ar[d]^{id\otimes ind_{S^\perp}} \\
K^{\tau_H+n}_H(H) \ar[rr]^{i^\natural} && K^{i^\natural\tau_H+n}_S(H) \ar[rr]^{i^*} && K^{i^*\tau_H+n}_S(S)
}$$}

(1), (2), (5) and (6) clearly commute.

(3) and (4) commute from Lemma \ref{naturality of Kunneth formula for j^natural for G} and
\ref{naturality of Kunneth formula for F^* for G} respectively.
\end{proof}
{\it Proof of Theorem \ref{main theorem for K for general group in the case of direct products}.}
Let us consider the following diagram.

Let $G:=H\times S^\perp$, $T:=S\times S^\perp$ and $k:=i\times id:S\times S^\perp=T\hookrightarrow H\times S^\perp=G$.
$$
\xymatrix{
t.e.K(G,\tau)\ar[rrr]^{i_1^\#} \ar[ddd]^{k^\#} \ar[rd]|{M.d._G} \ar@{}[rrrd]|{(1)} \ar@{}[rddd]|{(2)} & & & t.e.K(H,\tau_H) \ar[dl]|{M.d._H} \ar[ddd]^{i^\#} \ar@{}[lddd]|{(4)} \\
& char(G,\tau) \ar[r]^{char(i_1)} \ar[d]|{char(k)} \ar@{}[rd]|{(3)}  &
char(H,\tau_H) \ar[d]|{char(i)} & \\
& char(T,k^*\tau) \ar[r]^{char(i_1)} & char(S,i^*\tau_H) & \\
t.e.K(T,k^*\tau)
\ar[rrr]^{i_1^\#} \ar[ur]|{M.d._T} \ar@{}[rrru]|{(5)} & & & t.e.K(S,i^*\tau_H) \ar[ul]|{M.d._S}}
$$
We want to verify
the commutativity of (1). Since $char(i)$ is injective, this is equivalent to 
$$char(i)\circ M.d._H\circ i_1^\#=char(i)\circ char(i_1)\circ M.d._G.$$

The followings have been verified.

(2) and (4) follow from Theorem \ref{naturality by FHT}.

(3) follows from the definition of $char(i_1)$.

(5) follows from Theorem \ref{compatibility of t.e.K(i_1) with char(i_1)}.

That $i^\#\circ i_1^\#=i_1^\#\circ k^\#$ follows from the above lemma (the biggest square).

The following computation implies the theorem.

$$char(i)\circ M.d._H\circ i_1^\#=M.d._S\circ i^\#\circ i_1^\#$$
$$=M.d._S\circ i_1^\#\circ k^\#=char(i_1)\circ M.d._T\circ k^\#$$
$$=char(i_1)\circ char(k) \circ M.d._{G}$$
$$=char(i)\circ char(i_1)\circ M.d._{G}.$$
\begin{flushright} $\Box$ \end{flushright}

%有限被覆
\subsubsection{``Finite covering''}\label{t.e.K for finite cover for general G}
Let $H$, $G$ be compact connected Lie group with torsion-free $\pi_1$, $\tau$ be a positive central extension of $LG$ and the associated $G$-equivariant twisting over $G$, and $f:H\to G$ be a smooth group homomorphism satisfying the decomposable condition such that the restriction of it to maximal torus is a finite covering. Let $S$ and $T$ be maximal tori of $H$ and $G$ respectively such that $f(S)\subseteq T$. That is, the following commutative diagram holds.
$$\begin{CD}
H @>f>> G \\
@AiAA @AkAA \\
S @>f>> T
\end{CD}
$$
We verify the following theorem.

\begin{thm}\label{main theorem for K for general group in the case of finite covers}
If we define as
$$f^\#:=f^*\circ f^\natural,$$
the following commutative diagram holds.
$$\begin{CD}
t.e.K(G,\tau) @>f^\#>> t.e.K(H,f^*\tau)\\
@VM.d._GVV @VM.d._HVV \\
char(G,\tau) @>char(f)>> char(H,f^*\tau)\\
\end{CD}$$

\end{thm}
\begin{rmk}
In Theorem \ref{naturality by FHT}, the restriction of $f$ to maximal torus is supposed to be injective.
\end{rmk}
\begin{proof}
Let us start from the reduction to maximal tori. $n:=rank(H)=rank(G)$.
\begin{pro}
The following commutative diagram holds.
$$\begin{CD}
t.e.K(G,\tau) @>f^\#>> t.e.K(H,f^*\tau) \\
@Vk^\#VV @Vi^\#VV \\
t.e.K(T,k^*\tau) @>f^\#>> t.e.K(S,i^*f^*\tau)
\end{CD}$$
\end{pro}
\begin{proof}
It follows from the following commutative diagram.
$$\begin{CD}
K^{\tau+n}_G(G) @>f^\natural>> K^{f^\natural\tau+n}_{H}(G) @>f^*>> K^{f^*\tau+n}_{H}(H) \\
@Vk^\natural VV @Vi^\natural VV @Vi^\natural VV \\
K^{k^\natural\tau+n}_T(G) @>f^\natural>> K^{i^\natural f^\natural\tau+n}_{S}(G) @>f^*>> K^{i^\natural f^*\tau+n}_{S}(H) \\
@Vk^*VV @Vk^*VV @Vi^*VV \\
K^{k^*\tau+n}_T(T) @>f^\natural>> K^{f^\natural k^*\tau+n}_{S}(T) @>f^*>> K^{i^*f^*\tau+n}_{S}(S) 
\end{CD}$$
\end{proof}

%{\it Proof of Theorem \ref{main theorem for K for general group in the case of finite covers}.}
Let us consider the following diagram just like the proof of Theorem \ref{main theorem for K for general group in the case of direct products}.
$$\xymatrix{
t,e,K(G,\tau) \ar[rrr]^{f^\#} \ar[ddd]_{k^\#} \ar[dr]|{M.d._G} \ar@{}[dddr]|{\circlearrowleft} & & & 
t.e.K(H,f^*\tau) \ar[ddd]^{i^\#} \ar[dl]|{M.d._{H}} \ar@{}[dddl]|{\circlearrowleft}\\
& char(G,\tau) \ar[r]^{char(f)} \ar[d]|{char(k)} \ar@{}[dr]|{\circlearrowleft} & char(H,f^*\tau) \ar[d]|{char(i)} & \\
& char(T,k^*\tau) \ar[r]_{char(f)} & char(S,f^*k^*\tau ) \\
t.e.K(T,k^*\tau) \ar[rrr]_{f^\#} \ar[ur]|{M.d._T} & & & 
t.e.K(S,f^*k^*\tau) \ar[ul]|{M.d._{S}} \ar@{}[ulll]|{\circlearrowleft}
}
$$
Verified commutativity is represented by ``$\circlearrowleft$''. Moreover, we have known the commutativity of the biggest square already in the above lemma.

We want to verify that $M.d._H\circ f^\#=char(f)\circ M.d._G$. It can be verified with the same way in the case of Theorem \ref{main theorem for K for general group in the case of direct products}.
\end{proof}

\subsubsection{Proof of Theorem \ref{main theorem for K for general Lie group}.}
Let $f:H\to G$ be a smooth group homomorphism satisfying the decomposable condition, $\tau$ be a positive central extension of $LG$. $S$ and $T$ are chosen maximal tori such that $f(S)\subseteq T$.

Let us recall that we can define the orthogonal completion torus $S^\perp \subseteq T$ and decompose $f$ as the following sequence
$$H\xrightarrow{q}H/\ker(f)\xrightarrow{i_1}H/\ker(f)\times S^\perp\xrightarrow{f\cdot j}G$$
By restricting this sequence to maximal tori, we obtain the following commutative diagram
$$\xymatrix{
H \ar[r]^q & H/\ker(f) \ar[r]^{i_1} & H/\ker(f)\times S^\perp \ar[rr]^{f\cdot j} && G & \\
S\ar[r]^q \ar[u]^i & S/\ker(f) \ar[u]^i \ar[r]^{i_1} & S/\ker(f)\times S^\perp \ar[u]|{i\times id} \ar[rr]^{f\cdot j} && T \ar[u]^k & S^\perp \ar[l]_j \ar[lu]_j }.$$

Therefore, the following commutative diagram holds.
$$\begin{CD}
t.e.K(G,\tau) @>(f\cdot j)^\#>> t.e.K(H/\ker(f)\times S^\perp,(f\cdot j)^*\tau) @>i_1^\#>> \\
@VM.d._GVV @VM.d._{H/\ker(f)\times S^\perp}VV \\
char(G,\tau) @>char(f\cdot j)>> char(H/\ker(f)\times S^\perp,(f\cdot j)^*\tau) @>char(i_1)>> 
\end{CD}$$
$$\begin{CD}
@>i_1^\#>> t.e.K(H/\ker(f),i_1^*(f\cdot j)^*\tau) @>q^\#>> t.e.K(H,f^*\tau) \\
@. @VVM.d._{H/\ker(f)}V @VVM.d._HV \\
@>char(i_1)>> char(H/\ker(f),i_1^*(f\cdot j)^*\tau) @>char(q)>> char(H,f^*\tau) 
\end{CD}$$
The first and the third commutes from Theorem \ref{main theorem for K for general group in the case of finite covers}. The second commutes from Theorem \ref{main theorem for K for general group in the case of direct products}.

From $char(f)=char(q)\circ char(i_1)\circ char(f\cdot j)$, we obtain the conclusion.

%ループ群の表現論%
\section{The quasi functor $RL$}\label{section about RL for tori}

%一般論%
\subsection{Positive energy representations of $LT$}\label{PER of LT}

Representation theory for loop groups is well known (\cite{PS}). Especially we have an explicit description of irreducible positive energy representations of loop groups of tori.

Let us recall that $LT$ has a canonical decomposition $LT\cong T\times \Pi_T\times U$, where $T$ is the set of initial values of loops, $\Pi_T$ is the set of ``rotation numbers'' (naturally isomorphic to $\pi_1(T)\cong\pi_0(LT)$) and 
$$U:=\exp\biggl\{\beta :S^1\to\mathfrak{t} \bigg| \int_{S^1}\beta(s)ds=0\biggr\}$$
 is the set of derivatives of contractible loops whose initial values are $0$.

Let $\tau$ be a positive central extension of $LT$. The above decomposition is inherited partially, that is, $LT^{\tau}\cong (T\times\Pi_T)^{\tau}\otimes U^{\tau}$.

$U^\tau$ is called a Heisenberg group and has the unique irreducible positive energy representation $\rho_\mathcal{H}:U^\tau\to U(V_\mathcal{H}(U^\tau))$ at level $\tau$ up to equivalence. Concretely, $V_\mathcal{H}(U^\tau):=\widehat{S}((L\mathfrak{t}_\mathbb{C})_+)$ is a completion of the symmetric tensor algebra $S((L\mathfrak{t}_\mathbb{C})_+)$ of $(L\mathfrak{t}_\mathbb{C})_+$, which we call Heisenberg representation.
\begin{rmk}
The differential representation of it can be regarded as the ``completion'' of the Shr\"odinger representation of the infinite dimensional Heisenberg Lie algebra.
\end{rmk}

%The commutator in $(T\times\Pi)^\tau$ determines a homomorphism $\kappa^\tau:\Pi\to\Lambda $ as we explained at Definition \ref{morphism kappa}.
When we choose a character $\lambda\in\Lambda_T^\tau$, $U^\tau\otimes T^\tau$ acts on $V_\mathcal{H}(U^\tau)\otimes\mathbb{C}_\lambda$ by
$$\rho_{\ca{H}}\otimes \lambda(u\otimes t)(v\otimes z):=\rho_\ca{H}(u)(v)\otimes\lambda(t)z,$$
where $u\in U^\tau$, $t\in T^\tau$, $v\in V_\ca{H}(U^\tau)$ and $z\in \bb{C}_\lambda$.

Then $LT^\tau=U^\tau\otimes (T\times \Pi_T)^\tau$ acts on
$$V_{[\lambda]}:=\sum_{n\in\Pi_T}V_\mathcal{H}(U^\tau)\otimes\mathbb{C}_{\lambda+\kappa^\tau(n)}$$ where $\Pi_T$ permutes the components. The isomorphism classes of irreducible positive energy representations are in 1:1 correspondence with the points of $\Lambda_T^\tau/\kappa^\tau(\Pi_T)$ (\cite{PS} Proposition 9.5.11). Moreover, 
$$l.w._T(V_{[\lambda]})=[\lambda]_T.$$

%直積%
\subsection{Direct product}\label{RL for direct product}
Let $T_1$ and $T_2$ be tori, $\tau_1$ and $\tau_2$ be positive central extensions of $LT_1$ and $LT_2$ respectively, $i_1$ be the natural inclusion into the first factor of $T_1\times T_2$, and $p_j:T_1\times T_2\to T_j$ be the natural projection onto tha $j$' th factor ($j=1,2$). Then $T:=T_1\times  T_2$ is a torus and $\tau:=p_1^*\tau_1+p_2^*\tau_2$ is a positive central extension of $LT\cong LT_1\times LT_2$.

Using the above description, we will construct $i_1^!:R^\tau(LT)\to R^{i_1^*\tau}(LT_1)$ and verify that the following diagram commutes.
$$\begin{CD}
RL(T,\tau) @>i_1^!>> RL(T_1,i_1^*\tau) \\
@Vl.w._TVV @Vl.w._{T_1}VV \\
char(T,\tau) @>char(i_1)>> char(T_1,i_1^*\tau)
\end{CD}$$
In this situation, we have a natural isomorphisms $LT\cong LT_1\times LT_2$ and $LT^\tau\cong LT_1^{\tau_1}\otimes LT_2^{\tau_2}$, that is, for any $l_1$, $l_1'\in LT_1^{\tau_1}$ and $l_2$, $l_2'\in LT_2^{\tau_2}$, $l_1\otimes l_2\cdot l_1'\otimes l_2'=l_1l_1'\otimes l_2l_2'$.
Moreover, $LT^\tau\cong U_1^{\tau_1}\otimes U_2^{\tau_2}\otimes(T_1\times\Pi_1)^{\tau_1}\otimes(T_2\times\Pi_2)^{\tau_2}$, where $U_i$ is the group of contractible loops whose initial values are $0$ of $T_j$ ($j=1,2$).
%%%%%%%%%%%%%%%%%%%%%%%%%%%%%%%%%%%%%%%%%%%%%%%%%%%%%%%%%%%%%%%%%%%%%%%%%%%%%%%%%

From this commutativity, we can verify the following lemmas.
\begin{lem}\label{computation of char(i_1) in Section 5}
$$char(i_1)([\lambda]_T)= [{}^tdi_1(\lambda)]_{T_1}.$$
\end{lem}
\begin{proof}
Let us recall Lemma \ref{product formula for kappa}. %linklinklink
That is, $\kappa^\tau={}^tdp_1\circ \kappa^{\tau_1}\circ dp_1+{}^tdp_2\circ \kappa^{\tau_2}\circ dp_2$ when $\tau=p_1^*\tau_1+p_2^*\tau_2$.
%\begin{lem}%[product formula for $\kappa^\tau$, again]
%The homomorphism $\kappa^\tau$ is the composition of the following series;
%$$\Pi\xrightarrow{dp_1\oplus dp_2}\Pi_1\oplus\Pi_2\xrightarrow{\kappa^{\tau_1}
%\oplus\kappa^{\tau_2}}\Lambda_1\oplus\Lambda_2\xrightarrow{{}^tdp_1\oplus {}^tdp_2}\Lambda$$
%\end{lem}

From this formula, we have a bijection for any $\lambda\in \Lambda_T^\tau$

$$\{{}^tdi_1(\lambda)+\kappa^{\tau_1}(n_1)|n_1\in\Pi_{T_1}\}
\times\{{}^tdi_2(\lambda)+\kappa^{\tau_2}(n_2)|n_2\in\Pi_{T_2}\}$$
$$\xrightarrow{{}^tdp_1+ {}^tdp_2}\{{}^tdp_1\circ{}^tdi_1(\lambda)+{}^tdp_2\circ{}^tdi_2(\lambda)+
{}^tdp_1\circ\kappa^{\tau_1}\circ dp_1(n)+{}^tdp_2\circ\kappa^{\tau_2}\circ dp_2(n)|n\in \Pi_T\}$$
$$=\{\lambda+\kappa^\tau(n)|n\in\Pi_T\}.$$

Since ${}^tdi_1\circ {}^tdp_2=0$, we obtain the conclusion.
\end{proof}

\begin{lem}%[product formula of Heisenberg representation]
\label{product formula of Heisenberg representation}
Let $V_\mathcal{H}(U^\tau)$ be the Heisenberg representation space of $U^\tau$. Then 
$$V_\mathcal{H}(U^\tau)\cong V_\mathcal{H}(U_1^{\tau_1})\otimes V_\mathcal{H}(U_2^{\tau_2})$$
as representation spaces.
\end{lem}
\begin{proof}
This lemma follows from $U^\tau\cong U_1^{\tau_1}\otimes U_2^{\tau_2}$ and the product formula of the symmetric algebra
$$S((L\mathfrak{t}_\mathbb{C})_+)\cong S((L(\mathfrak{t}_1\oplus\mathfrak{t}_2)_\mathbb{C})_+)\cong S((L\fra{t}_{1\bb{C}})_+\oplus(L\fra{t}_{2\bb{C}})_+)$$
$$\cong S((L\mathfrak{t}_{1\mathbb{C}})_+)\otimes S((L\mathfrak{t}_{2\mathbb{C}})_+).$$
Factorization as representation spaces follows from the commutativity between $LT_1^{\tau_1}$ and $LT_2^{\tau_2}$ described above.
\end{proof}

\begin{lem}%[product formula of representation for tori]
\label{product formula of representation for tori}
Let $\lambda\in\Lambda_T^\tau$ be a $\tau$-twisted character of $T^\tau\cong T_1^{\tau_1}\otimes T_2^{\tau_2}$, then $$\mathbb{C}_\lambda\cong\mathbb{C}_{{}^tdi_1(\lambda)}\otimes\mathbb{C}_{{}^tdi_2(\lambda)}$$
as representation spaces.
\end{lem}
\begin{proof}

$t=t_1\otimes t_2\in T^\tau$ is also written as $\widetilde{i}_1(t_1)\cdot \widetilde{i}_2(t_2)$, where $\widetilde{i_k}:T_k^{\tau_k}\to T^\tau$ is induced from $i_k:T_k\to T$ ($k=1,2$). The following computation implies the statement.
$$\lambda(t)=\lambda(\widetilde{i}_1(t_1)\cdot \widetilde{i}_2(t_2))=\lambda(\widetilde{i}_1(t_1))\lambda(\widetilde{i}_2(t_2))$$
$$=(i_1^*\lambda)(t_1)(i_2^*\lambda)(t_2)$$
When we regard the set of twisted character $\Lambda^{\tau_j}_{T_j}$ as a subset of ${\rm Hom}(\Pi_{T_j^{\tau_j}},\bb{Z})$, $i_1^*\lambda$ corresponds to ${}^tdi_1(\lambda)$.
\end{proof}

The above lemmas imply the following theorem.
\begin{thm}\label{product formula for PER}
We have an isomorphism
$$V_{[\lambda]}\cong V_{[{}^tdi_1(\lambda)]}\otimes V_{[{}^tdi_2(\lambda)]}$$
as representation spaces.
\end{thm}
\begin{proof}
$$V_{[\lambda]}=\sum_{n\in\Pi_T}V_\mathcal{H}(U^\tau)\otimes\mathbb{C}_{\lambda+\kappa^\tau(n)}$$
$$\cong\sum_{n_1\in\Pi_{T_1},n_2\in\Pi_{T_2}}V_\mathcal{H}(U_1^{\tau_1})\otimes V_\mathcal{H}(U_2^{\tau_2})\otimes\mathbb{C}_{{}^tdi_1(\lambda)+\kappa^{\tau_1}(n_1)}\otimes\mathbb{C}_{{}^tdi_2(\lambda)+\kappa^{\tau_2}(n_2)}$$
$$\cong\sum_{n_1\in\Pi_{T_1},n_2\in\Pi_{T_2}}\bigl[V_\mathcal{H}(U_1^{\tau_1})\otimes\mathbb{C}_{{}^tdi_1(\lambda)+\kappa^{\tau_1}(n_1)}\bigr]\otimes
\bigl[V_\mathcal{H}(U_2^{\tau_2})\otimes\mathbb{C}_{{}^tdi_2(\lambda)+\kappa^{\tau_2}(n_2)}\bigr]$$
$$\cong\Bigl[\sum_{n_1\in\Pi_{T_1}}V_\mathcal{H}(U_1^{\tau_1})\otimes\mathbb{C}_{{}^tdi_1(\lambda)+\kappa^{\tau_1}(n_1)}\Bigr]
\otimes\Bigl[\sum_{n_2\in\Pi_{T_2}}V_\mathcal{H}(U_2^{\tau_2})\otimes\mathbb{C}_{{}^tdi_2(\lambda)+\kappa^{\tau_2}(n_2)}\Bigr]$$
$$=V_{[{}^tdi_1(\lambda)]}\otimes V_{[{}^tdi_2(\lambda)]}$$
Factorization as representation spaces follows from Lemma \ref{product formula of Heisenberg representation}, \ref{product formula of representation for tori} and the bijection in the proof of Lemma \ref{computation of char(i_1) in Section 5}.
\end{proof}
Since the dimension of $V_{[{}^tdi_2(\lambda)]}$ is infinite, the following follows immediately.
\begin{cor}\label{infinitely reducibility in Section 5}
When we regard an irreducible representation of $LT^\tau$ as a representation of $LT_1^{\tau_1}$, it is never finitely reducible if $\dim T_2\geq1$.
\end{cor}

Motivated by this observation, we define the new ``induced'' representation $i_1^!V_{[\lambda]}$ so that it has the compatibility with $char(i_1)$.
\begin{dfn}\label{def. of i_1^!}
Let $V$ be a finitely reducible representation space of $LT^\tau$.
$$i_1^!V:=\sum_{[\lambda_2]\in\Lambda_{T_2}^{\tau_2}/\kappa^{\tau_2}(\Pi_{T_2})}{\rm Hom}_{LT_2^{\tau_2}}(V_{[\lambda_2]},i_2^*V)$$
$$\cong{\rm Hom}_{LT_2^{\tau_2}}(\sum_{[\lambda_2]\in\Lambda_{T_2}^{\tau_2}/\kappa^{\tau_2}(\Pi_{T_2})}V_{[\lambda_2]},i_2^*V)$$
where, ${\rm Hom}_{LT_2^{\tau_2}}(V_{[\lambda_2]},i_2^*V)$ is the set of bounded intertwining operators. The topology of the left hand side is defined below.
\end{dfn}
\begin{pro}
$i_1^!V$ is a representation space of $LT_1^{\tau_1}$. The action is defined by
$$(l.F)(v):=l.(F(v)),$$
where $F\in i_1^!V$, $l\in LT_1^{\tau_1}$ and $v\in \sum_{[\lambda_2]\in\Lambda_{T_2}^{\tau_2}/\kappa^{\tau_2}(\Pi_{T_2})}V_{[\lambda_2]}$.
\end{pro}
\begin{proof}
It is sufficient to verify that $l.F$ is an intertwining operator. From the linearity of $l.F$, we can assume that $v\in V_{[\lambda_2]}$ for some $[\lambda_2]\in\Lambda_{T_2}^{\tau_2}/\kappa^{\tau_2}(\Pi_{T_2})$. Let $l\in LT_1^{\tau_1}$, $l'\in LT_2^{\tau_2}$ and $F\in {\rm Hom}_{LT_2^{\tau_2}}(V_{[\lambda_2]},i_2^*V)$.
$$(l.F)(l'.v)=l.(F(l'.v))=l.(l'.(F(v)))$$
$$=(ll').F(v)=(l'l).F(v)$$
$$=l'.(l.(F(v)))=l'.(l.F)(v).$$
The first and the last equality follow from the definition. The second one follows from that $F$ is an intertwining operator. The third and the fifth one follow from that a representation is a group homomorphism. The fourth one follows from the commutativity described above.

%\begin{pro}
%$i_1^!V_{[\lambda]}$ is a unitary representation space of $LT_1^{\tau_1}$.

%The inner product of ${\rm Hom}_{LT_2^{\tau_2}}(V_{[\lambda_2]},i_2^*V_{[\lambda]})$ is given by $(F_1,F_2):=(F_1(v),F_2(v))_{V_{[\lambda]}}$ where $F_1,F_2\in {\rm Hom}_{LT_2^{\tau_2}}(V_{[\lambda_2]},i_2^*V_{[\lambda]})$, $v$ is a chosen unit vector in $V_{[\lambda_2]}$. This definition is independent of the choice of $v$. The inner product of $i_1^!V_{[\lambda]}$ is given by the direct sum.

%The action of $LT_1^{\tau_1}$ is given by $(l_1\cdot F)(v):=l_1\cdot (F(v))$.
%\end{pro}
%\begin{rmk}
%Well-definedness of inner product is verified below.
%$LT_2^{\tau_2}$ equivariance of $l_1\cdot F$ is follows from the commutativity between $LT_1^{\tau_1}$ and %$LT_2^{\tau_2}$.
%\end{rmk}
\end{proof}
The following lemma is clear from the definition.
\begin{lem}
For any finitely reducible positive energy representation space $V_1$ and $V_2$ of $LT^\tau$ at level $\tau$,
$$i_1^!(V_1\oplus V_2)\cong i_1^!V_1\oplus i_1^!V_2.$$
\end{lem}
From this lemma, we can assume that $V$ is irreducible.
\begin{thm}\label{isom. between i_1^!V with V_lambda_2}
$$i_1^!V_{[\lambda]}\cong V_{[{}^tdi_1(\lambda)]},$$
therefore, $i_1^!$ has a compatibility with $char(i_1)$.

If we define an inner product of the left hand side can be defined by $(F_1,F_2):=(F_1(v),F_2(v))_{V_{[\lambda]}}$, where $v$ is a chosen unit vector in $V_{[{}^tdi_2(\lambda)]}$, this inner product does not depend on the choice of $v$, and the above isomorphism is isometric.
\end{thm}
\begin{proof}
Fix a completely orthonormal system $\{w_j\}_{j\in\mathbb{N}}$ of $V_{[{}^tdi_1(\lambda)]}$. From Theorem \ref{product formula for PER},
$V_{[\lambda]}\cong\sum_j\mathbb{C} w_j\otimes V_{[{}^tdi_2(\lambda)]}$ as a representation space of $LT_2^{\tau_2}$. So we can decompose $F\in {\rm Hom}_{LT_2^{\tau_2}}(V_{[\lambda_2]},i_2^*V_{[\lambda]})$ as 
$$F=\sum_jF_j,$$
where $F_j:V_{[{}^tdi_2(\lambda)]}\to\mathbb{C} w_j\otimes V_{[{}^tdi_2(\lambda)]}\cong V_{[{}^tdi_2(\lambda)]}$.

By Schur's lemma, $F_j=c_jid$ for some $c_j\in\mathbb{C}$.
Since $F$ determines an operator, for any $v\in V_{[{}^tdi_2(\lambda)]}$, $\sum_jF_j(v)=v\otimes\sum_jc_jw_j\in V_{[\lambda]}$. Therefore, $\sum_j|c_j|^2<\infty$.

Moreover, by Schur's lemma, if $[\lambda_2]\neq[{}^tdi_2(\lambda)]$, ${\rm Hom}_{LT_2^{\tau_2}}(V_{[\lambda_2]},i_2^*V_{[\lambda]})=0$. Therefore, we can define an isomorphism
$$\sum_{[\lambda]\in\Lambda_{T_2}^{\tau_2}/\kappa^{\tau_2}(\Pi_{T_2})}{\rm Hom}_{LT_2^{\tau_2}}(V_{[\lambda_2]},i_2^*V_{[\lambda]})={\rm Hom}_{LT_2^{\tau_2}}(V_{[{}^tdi_2(\lambda)]},i_2^*V_{[\lambda]})
\to V_{[i_1^*\lambda]}$$
by $$F\mapsto\sum_jc_jw_j.$$
Since $\sum_j|c_j|^2<\infty$, the infinite sum $\sum_jc_jw_j$ converges.

Moreover, since $\sum_jc_jw_j$ is determined by $F(v)=v\otimes\sum_jc_jw_j$, this isomorphism is independent of the choice of $v$ and completely orthonormal system.
\end{proof}
%\begin{dfn}\label{topology of i_1^!V}
%$i_1^!V_{[\lambda]}$ is topologized through the above isomorphism $i_1^!V_{[\lambda]}\cong V_{[{}^tdi_1(\lambda)]}$.
%\end{dfn}
From this theorem, we obtain the following conclusion.
\begin{cor}\label{main theorem for R in the case of direct product}
The following commutative diagram holds.
$$\begin{CD}
RL(T,\tau) @>i_1^!>> RL(T',i_1^*\tau) \\
@Vl.w._TVV @Vl.w._{T_1}VV \\
char(T,\tau) @>char(i_1)>> char(T_1,i_1^*\tau)
\end{CD}$$
\end{cor}

%有限被覆%
\subsection{Finite covering}\label{RL for finite cover}

In this section, we study in the case of a finite covering.

Let $q:T'\to T$ be a finite covering. We can compute $char(q)$ from Lemma \ref{pull back formula of kappa} %link
 with the same way in Lemma \ref{computation of char(i_1) in Section 5}.%link
\begin{lem}\label{formula for char(q)}
$$char(q)([\lambda]_T)=\sum_{[m]\in\Pi_T/dq(\Pi_{T'})}[{}^tdq(\lambda+\kappa^\tau(m))]_{T'},$$
where $m$ is a chosen representative element of $[m]$.
\end{lem}
\begin{rmk}
$\Pi_{T'}$-orbit $[{}^tdq(\lambda)+\kappa^\tau(dq(m))]_{T'}$ is independent of a choice of representative element $m$. %$[m]\in\Pi_T/dq(\Pi_{T'})$ means that $m$ is a chosen representative element of $\Pi_T/dq(\Pi_{T'})$.
\end{rmk}
\begin{proof}
It follows from the following.
$$\{{}^tdq(\lambda+\kappa^\tau(n))|n\in\Pi_T\}=\{{}^tdq(\lambda)+{}^tdq(\kappa^\tau(n))|n\in\Pi_T\}$$
$$=\{{}^tdq(\lambda)+{}^tdq(\kappa^\tau(dq(n')+m))|n'\in\Pi_{T'},m\in\Pi_T/dq(\Pi_{T'})\}$$
$$=\coprod_{m\in\Pi_T/dq(\Pi_{T'})}\{{}^tdq(\lambda+\kappa^\tau(m))+\kappa^{q^*\tau}(n')|n'\in\Pi_{T'}\}.$$
\end{proof}
Let us define the induced homomorphism $q^!$.
\begin{dfn}\label{def. of q^!V}
Let $V$ be a finitely reducible representation space of $LT^\tau$.
$$q^!V:=q^*V,$$
where the action of $LT'^{q^*\tau}$ on $q^*V$ is defined through the homomorphism $Lq$.
\end{dfn}
The following lemma is clear from the definition.
\begin{lem}
For any finitely reducible positive energy representation spaces $V_1$ and $V_2$ of $LT^\tau$ at level $\tau$,
$$q^!(V_1\oplus V_2)\cong q^!V_1\oplus q^!V_2.$$
\end{lem}
From this lemma, we can assume that $V$ is irreducible.

\begin{thm}\label{isom. between q^!V with}
$$q^!(V_{[\lambda]})\cong\sum_{m\in\Pi_T/dq(\Pi_{T'})}V_{[{}^tdq(\lambda+\kappa^\tau(m))]}.$$
Moreover, the following commutative diagram holds.
$$\begin{CD}
R^\tau(LT) @>q^!>> R^{q^*\tau}(LT') \\
@Vl.w._TVV @Vl.w._{T'}VV \\
char(T,\tau) @>char(q)>> char(T',q^*\tau)
\end{CD}$$
\end{thm}
\begin{proof}
Since $T$ and $T'$ are locally isomorphic, the set of the derivatives of contractible loops whose initial values are $0$ of $T$ corresponds to one of $T'$ by the natural way. That is, $LT^\tau\cong (T\times \Pi_T)^\tau\otimes U^\tau$, $LT'^{q^*\tau}\cong (T'\times \Pi_{T'})^{q^*\tau}\otimes U^{q^*\tau}$ and $U^\tau\cong U^{q^*\tau}$.
Therefore, the Heisenberg representations $V_\ca{H}$ of $U^\tau\subseteq LT^\tau$ coincide with one of $U^{q^*\tau}\subseteq LT'^{q^*\tau}$.
$$q^*(V_{[\lambda]})
=q^*(\sum_{n\in\Pi_T}V_\mathcal{H}\otimes\mathbb{C}_{\lambda+\kappa^\tau(n)})=\sum_{n\in\Pi_T}V_\mathcal{H}\otimes\mathbb{C}_{{}^tdq(\lambda+\kappa^\tau(n))}$$
$$=\sum_{m\in\Pi_T/dq(\Pi_{T'})}\sum_{n'\in\Pi_{T'}}V_\mathcal{H}\otimes
\mathbb{C}_{{}^tdq(\lambda+\kappa^\tau(m))+\kappa^{q^*\tau}(n')}$$
$$=\sum_{m\in\Pi_T/dq(\Pi_{T'})}V_{[{}^tdq(\lambda+\kappa^\tau(m))]}$$
Commutativity of the above diagram is clear from this description and Lemma \ref{char formula for finite cover}.
\end{proof}

%一般の場合%
\subsection{Local injection}\label{RL for local injection}

Let $f:T'\to T$ be a local injection.
From Theorem \ref{decomposition of local injection}, we can decompose $f$ as the following sequence
$$T'\stackrel{q}{\longrightarrow}T'/\ker(f)\stackrel{i_1}{\longrightarrow}T'/\ker(f)\times T^\perp\stackrel{f\cdot j}{\longrightarrow}T,
$$
where $q$ and $f\cdot j$ are finite coverings, $i_1$ is the natural inclusion into the first factor.

\begin{dfn}\label{def. of f^!}
$$f^!:=q^!\circ i_1^!\circ (f\cdot j)^!$$
\end{dfn}
The following theorem follows from Corollary \ref{main theorem for R in the case of direct product},
 Theorem \ref{isom. between q^!V with},
Theorem \ref{decomposition of local injection} 
and the same argument in the proof of Theorem \ref{main theorem for K for tori in Section 4}.
\begin{thm}\label{main theorem for R}
The following commutative diagram holds.
$$\begin{CD}
R^\tau(LT) @>f^!>> R^{f^*\tau}(LT')\\
@Vl.w._TVV @Vl.w._{T'}VV \\
char(T,\tau) @>char(f)>> char(T',f^*\tau)
\end{CD}$$
\end{thm}

\end{document}